\pdfoutput=1
\RequirePackage{ifpdf}
\ifpdf % We~are running pdfTeX in pdf mode
\documentclass[pdftex]{sigma}
\else
\documentclass{sigma}
\fi

\usepackage{pgfplots}
\usepackage{enumerate}

\pgfplotsset{compat=1.10}
\usepgfplotslibrary{fillbetween}
\usetikzlibrary{patterns}

\usepackage{tikz}
\usetikzlibrary{
 cd,
 decorations.markings
}
\tikzset{
 bordism/.style={thick},
 cut/.style={dashed},
 axis/.style={},
 dot/.pic={
 \fill[color=black] (0, 0) circle[radius=1.5pt]
 node[\tikzpictextoptions]{\tikzpictext};
 },
 oriented/.style={
 postaction=decorate,
 decoration={
 markings,
 mark=at position #1 with {\arrow{stealth}}
 }
 },
 oriented/.default=0.5,
 commutative diagrams/mark/.style={
 font={},
 shape=asymmetrical rectangle,
 anchor=center
 }
}

\newcommand\Coker{\operatorname{Coker}}

\newcommand\Fun{\operatorname{Fun}}
\newcommand\Hom{\operatorname{Hom}}

\newcommand\Ker{\operatorname{Ker}}

\newcommand\colim{\operatorname*{colim}}

\newcommand\pr{\operatorname{pr}}

\newcommand\tr{\operatorname{tr}}
\newcommand\vol{\operatorname{vol}}
\newcommand\sep{\operatorname{sep}}
\newcommand\ext{\operatorname{ext}}
\newcommand\res{\operatorname{res}}

\newcommand\TFT{\mathhyphen\mathrm{TFT}}

\newcommand\Bord{\mathhyphen\mathrm{Bord}}
\newcommand\Born{\mathrm{Born}}
\newcommand\CAlg{\mathrm{CAlg}}
\newcommand\CBorn{\mathrm{CBorn}}
\newcommand\CMod{\mathrm{CMod}}

\newcommand\Fam{\mathrm{Fam}}
\newcommand\Fin{\mathrm{Fin}}
\newcommand\GBord{\mathcal G\mathrm{Bord}}

\newcommand\Man{\mathrm{Man}}

\newcommand\PSh{\mathrm{PSh}}
\newcommand\Path{\mathrm{Path}}

\newcommand\Set{\mathrm{Set}}
\newcommand\Sh{\mathrm{Sh}}
\newcommand\SmSt{\mathrm{St}_{\mathrm{Man}}}
\newcommand\St{\mathrm{St}}
\newcommand\PSt{\mathrm{PSt}}

\newcommand\Vect{\mathrm{Vect}}
\newcommand\alg{\mathrm{alg}}
\newcommand\backsslash{\backslash\kern-0.7ex\backslash}
\newcommand\id{\mathrm{id}}

\newcommand\mathemdash{\textrm{---}}
\newcommand\mathhyphen{\textrm{-}}
\newcommand\otimeshat{\mathbin{\hat\otimes}}
\newcommand\pt{\mathrm{pt}}

\newcommand\sslash{/\kern-0.7ex/}
\newcommand\vN{\mathrm{vN}}
\newcommand{\point}{\mathrm{pt}_M}

\newcommand{\C}{\mathbb{C}}
\newcommand{\R}{\mathbb{R}}
\newcommand{\N}{\mathbb{N}}
\newcommand{\Z}{\mathbb{Z}}
\newcommand{\V}{V}
\newcommand{\W}{W}
\newcommand\G{\mathcal{G}}

\usepackage{bbm}

\newcommand\opn{\operatorname}
\newcommand\cF{\mathcal F}
\newcommand\cG{\mathcal G}
\newcommand\M{\mathbb M}
\newcommand\cS{\mathcal S}

\numberwithin{equation}{section}

\newtheorem{Theorem}{Theorem}[section]
\newtheorem{Corollary}[Theorem]{Corollary}
\newtheorem{Lemma}[Theorem]{Lemma}
\newtheorem{Proposition}[Theorem]{Proposition}
 { \theoremstyle{definition}
\newtheorem{Definition}[Theorem]{Definition}

\newtheorem{Example}[Theorem]{Example}
\newtheorem{Examples}[Theorem]{Examples}
\newtheorem{Remark}[Theorem]{Remark}
\newtheorem{Construction}[Theorem]{Construction}
\newtheorem{Notation}[Theorem]{Notation}}

\begin{document}
\allowdisplaybreaks

\newcommand{\arXivNumber}{2001.05721}

\renewcommand{\PaperNumber}{072}

\FirstPageHeading

\ShortArticleName{A Framework for Geometric Field Theories and their Classification in Dimension One}

\ArticleName{A Framework for Geometric Field Theories\\ and their Classification in Dimension One}

\Author{Matthias LUDEWIG~$^{\rm a}$ and Augusto STOFFEL~$^{\rm b}$}

\AuthorNameForHeading{M.~Ludewig and A.~Stoffel}

\Address{$^{\rm a)}$~Universit\"at Regensburg, Germany}
\EmailD{\href{mailto:matthias.ludewig@mathematik.uni-regensburg.de}{matthias.ludewig@mathematik.uni-regensburg.de}}

\Address{$^{\rm b)}$~Universit\"at Greifswald, Germany}
\EmailD{\href{mailto:arstoffel@gmail.com}{arstoffel@gmail.com}}

\ArticleDates{Received June 15, 2020, in final form July 12, 2021; Published online July 25, 2021}

\Abstract{In this paper, we develop a general framework of geometric functorial field theories, meaning that all bordisms in question are endowed with geometric structures. We~take particular care to establish a notion of smooth variation of such geometric structures, so that it makes sense to require the output of our field theory to depend smoothly on the input. We then test our framework on the case of $1$-dimensional field theories (with or~with\-out orientation) over a manifold $M$. Here the expectation is that such a field theory is equivalent to the data of a vector bundle over $M$ with connection and, in the nonoriented case, the additional data of a nondegenerate bilinear pairing; we prove that this is indeed the case in our framework.}

\Keywords{field theory; vector bundles; bordism}

\Classification{57R56; 14D21; 57R22}

\section{Introduction}\label{sec:introduction}

Inspired by work of Witten~\cite{MR953828}, Segal and Atiyah pioneered the mathematical description of quantum
field theories as functors~\cite{MR1001453, MR2079383}.
More precisely, they described a $d$-dimensional quantum field theory
$Z$ as a functor that assigns to a closed $(d-1)$-manifold $Y$ a
vector space~$Z(Y)$ and to a $d$-dimensional bordism $X$ from $Y$
to another closed $d$-manifold $Y'$ a linear map $Z(X)\colon
Z(Y)\to Z(Y')$. Moreover, $Z$ is required to be a symmetric monoidal
functor, which means that $Z$ applied to a disjoint union of manifolds
of dimension $d-1$ or $d$ corresponds to the tensor product of the
associated vector spaces or linear maps. Segal's paper focused on
\emph{conformal} field theories, which means that the manifolds
involved come equipped with conformal structures, while Atiyah
discusses \emph{topological} field theories, where the manifolds are
smooth, but not equipped with any additional geometric structure.

Our first goal in this paper is to develop a general framework for
\emph{geometric} field theories. This involves a general definition
of a ``geometric structure'' $\G$ on $d$-dimensional manifolds,
which then leads to the definition of a symmetric monoidal bordism
category $\GBord$ whose morphisms are $d$-dimensional bordisms
equipped with a $\G$-structure. This is much more general than the
conformal structures considered by Segal or the rigid structures
based on the action of a Lie group $G$ on a $d$-dimensional model
space $\mathbb M$ of Stolz and Teichner~\cite{MR2742432}. Then we
essentially follow~\cite{MR2742432} to define $\G$-field theories.
As discussed at length in that paper, it is crucial to ensure the
\emph{smoothness} of the field theories; intuitively, this means in
particular that the operator $Z(X)\colon Z(Y)\to Z(Y')$ associated to
a bordism $X$ from $Y$ to $Y'$ depends smoothly on the bordism $X$.
At a technical level, this means that we need ``family versions'' of
the bor\-dism category $\GBord$ and the target category $\Vect$ of
suitable vector spaces whose objects and morphisms are now
\emph{families} of the originally considered objects/morphisms,
parametrized by smooth manifolds. In~\cite{MR2742432} this is
implemented by considering $\GBord$ and $\Vect$ as categories
\emph{internal to the $2$-category} of smooth stacks, but it has
become clear that, for technical reasons, it is easier to construct
and to work with the complete Segal object in smooth stacks that
should be thought of as the ``nerve'' of the internal category we
considered before, as is done in the preprint~\cite{arXiv:1501.00967}.
We carry out our constructions for non-extended field theories.
It is possible to define
extended field theories using an extension of our approach, via
$d$-fold (or $d$-uple) Segal objects.
For one-dimensional field theories, which are the main object of study
in this article, these distinctions are irrelevant.

The second goal of this paper is to check whether this abstract and
involved definition yields something sensible in the simplest cases,
namely $1$-dimensional (oriented) topological field theories over a
manifold $M$. In other words, the geometric structure on
$1$-manifolds $X$ is simply a~smooth map $\gamma\colon X\to M$, or
such a map $\gamma$ plus an orientation on $X$.

\begin{Theorem}\label{MainThm_or}
 The groupoid of $1$-dimensional oriented topological field
 theories over $M$ is equi\-valent to the groupoid of
 finite-dimensional vector bundles with connections over $M$.
\end{Theorem}

\begin{Theorem}\label{MainThm_nonor}
 The groupoid of $1$-dimensional unoriented topological field
 theories over $M$ is equivalent to the groupoid of
 finite-dimensional vector bundles over $M$ equipped with a
 non-degenerate, possibly indefinite, field of symmetric bilinear forms on the
 fibers and a compatible connection.
\end{Theorem}

There are actually two versions of each of these results, depending on
whether all vector spaces involved are real or complex. For the field
theories, the two flavors come from the choice of the target category
as the category (of families of) real or complex vector spaces.
Similarly, the vector bundles over $M$ considered can be real or
complex.

Theorem~\ref{MainThm_or} is certainly the expected result. The basic
idea is that a vector bundle $E\to M$ with connection $\nabla$
determines a $1$-dimensional field theory $Z$ over $M$ which
associates to a~point~$x$ (interpreted as an object of $\GBord$) the
vector space $Z(x)=E_x$ given by the fiber over $x$, and to a path
$\gamma \colon [a,b]\to M$ (interpreted as a morphism in $\GBord$ from
$\gamma(a)$ to $\gamma(b)$) the linear map $Z(\gamma)\colon Z(x)\to
Z(y)$ given by parallel translation along the path $\gamma$.

In fact, there are closely related results in the literature, in work
of Freed~\cite[Appendix~A]{MR1337109}, and, in particular, in the
papers by Schreiber and Waldorf~\cite{MR2520993} and by Berwick-Evans and
Pavlov~\cite{arXiv:1501.00967}, whose title indeed seems a statement
of our first theorem. Indeed, our framework is closely related to
that of the latter paper; however, our goal to give a general definition of geometric
bordism categories leads to a different bordism category even in
dimension one, as explained below (see Section~\ref{SubsectionBordismsSegalApproach}).
In~\cite{MR2077672, MR2520993}, invariance under ``thin'' homotopies plays a~prominent role.
These concepts turned out to be not relevant to the present paper, as we were able to prove the main results (in particular Proposition~\ref{PropMultiplicativity}) without such assumptions.

This paper is organized as follows. In Section~\ref{SectionDiscussion}, we give a more detailed exposition of our
construction and discuss the differences to the papers cited above.
Afterwards, in Section~\ref{sec:smooth-field-theor}, we~define our
notion of geometry, and use it to define a \emph{smooth category of
 geometric bordisms}, for any geometry, in any dimension. Starting
in Section~\ref{sec:classification}, we restrict to the case of field
theories in dimension one. In particular, we prove a version of the
classification Theorems~\ref{MainThm_or} and~\ref{MainThm_nonor} under
the technical assumption that ``families of vector spaces'' are
finite-dimensional vector bundles. As discussed in Section~\ref{subsec:fam_vs} for a geometric field theory $Z$, unlike for
topological field theories, the vector space $Z(Y)$ associated to an
object $Y$ of the bordism category $\GBord$ is typically not finite-dimensional.
This in turn leads to the requirement that the vector
spaces~$Z(Y)$ need to be equipped with a topology or a ``bornological
structure'' (see Appendix~\ref{sec:bornological}) in order to
formulate the requirement that the operator $Z(X)\colon Z(Y)\to Z(Y')$
associated to a bordism $X$ from $Y$ to $Y'$ depends ``smoothly'' on
$X$. As also explained in Section~\ref{subsec:fam_vs}, the
appropriate notion of an ``$S$-family of (topological or bornological)
vector spaces'' needs to be more general than locally trivial bundles
over the parameter spaces $S$, namely sheaves of $\mathcal
O_S$-modules. These are the objects of the target category
appropriate for {\em general field theories}, and so we consider the
version of Theorems~\ref{MainThm_or} and~\ref{MainThm_nonor} for that
target category as the main result of this paper. This is proved in
Section~\ref{sec:values-in-sheaves} (see Theorem~\ref{thm:sheafy} and
Remark~\ref{rmk:final-version} for the precise statement).

\section{Discussion of the results}
\label{SectionDiscussion}

In this section we provide an informal overview of our framework of
geometric field theories, a~discussion of our motivations, and
comparisons to the existing literature.

\subsection{Bordisms in the Segal approach} \label{SubsectionBordismsSegalApproach}

In the presence of geometric structures, it is difficult to perform
the gluing of bordisms along their boundaries in a systematic way, as
needed to define composition in geometric bordism categories; see for
instance the discussion in the introduction of~\cite{arXiv:1501.00967}. Here the idea of Segal objects comes to the
rescue, as it allows one to instead consider only \emph{decomposition}
of bordisms along hypersurfaces, which is unproblematic. In this
approach, a category $C$ is encoded by its \emph{nerve}, that is, the
simplicial set $\mathcal C$, where $\mathcal C_0$ is the set of objects
and, for $n \geq 1$, $\mathcal C_n$ is the set of chains of $n$
composable morphisms; composition and identity morphisms in $C$
determine the simplicial structure maps between these sets.

To describe, at least roughly, the $d$-dimensional bordism category in
this way, we let $\mathcal{C}_n$ consist of $d$-dimensional manifolds
$X$ together with a collection of compact hypersurfaces $Y_k \subset
X$, for~$k=0, \dots, n$, as in Figure~\ref{fig:bordism}. This encodes
a chain of $n$ composable bordisms, the $k$th of~them being the
portion of $X$ lying between $Y_{k-1}$ and $Y_k$. Composition is
encoded by forgetting the marked hypersurfaces. To build in a
geometry $\mathcal G$ (for instance, orientations, Riemannian metrics,
or maps to a background manifold $M$), we just ask that $X$ is endowed
with that additional structure.

\begin{figure}
 \centering
 \begin{tikzpicture}[yscale=2]
 \foreach \x in {0, ..., 6} {
 \coordinate (a\x) at (\x, 0);
 \coordinate (b\x) at (\x, 1);
 }
 \draw (a0) to[bend right=2]
 (a1) to[bend left=6]
 (a2) to[bend right=9] coordinate[pos=1/2] (c0)
 (a3) to[bend left=12]
 (a4) to[bend right=4]
 (a5) to[bend left=5]
 (a6);
 \draw (b0) to[bend right=7]
 (b1) to[bend left=10]
 (b2) to[bend right=6] coordinate[pos=1/2] (c3)
 (b3) to[bend left=3]
 (b4) to[bend right=13]
 (b5) to[bend left=12]
 (b6);
 \draw (a2) ++(-0.3, 0.6) to[bend right=15]
 coordinate[pos=1/4] (d0)
 coordinate[pos=1/2] (c1)
 coordinate[pos=3/4] (d1)
 +(1.6,0);
 \draw (d0) to[bend left=20]
 coordinate[pos=1/2] (c2)
 (d1);
 \draw (a0) to [bend right=30] (b0);
 \draw (a0) to [bend left= 30] node[left] {$X$} (b0);
 \draw (a6) to [bend right=30] (b6);
 \draw[dashed] (a6) to [bend left= 30] (b6);
 \begin{scope}[ultra thick]
 \draw (a1) to [bend right=30] (b1);
 \draw[dashed] (a1) to [bend left= 30] (b1);
 \draw (c0) to [bend right=30] (c1);
 \draw[dashed] (c0) to [bend left= 30] (c1);
 \draw (c2) to [bend right=30] (c3);
 \draw[dashed] (c2) to [bend left= 30] (c3);
 \path (a4) to node{\Large$\cdots$} (b4);
 \draw (a5) to [bend right=30] (b5);
 \draw[dashed] (a5) to [bend left= 30] (b5);
 \end{scope}
 \path (a1) node[below] {$Y_0$}
 (c0) node[below] {$Y_1$}
 (a5) node[below] {$Y_n$};
 \end{tikzpicture}
 \caption{An object of $\mathcal C_n$ comprises the $d$-manifold $X$
 and the marked hypersurfaces $Y_0, \dots, Y_n$.}
 \label{fig:bordism}
\end{figure}
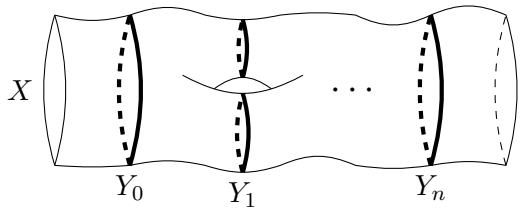

In particular, objects of $\mathcal C$ (i.e., elements of $\mathcal
C_0$) consist of a $d$-dimensional manifold $X$ with a marked compact
hypersurface $Y$, instead of just the $(d-1)$-dimensional manifold
$Y$. Now, this set of objects is much larger than what we would like
to have, since the portion of $X$ far away from $Y$ should be
irrelevant. This issue can be dealt with by promoting $\mathcal{C}_0$
from a set to a~\emph{groupoid}; we add isomorphisms that establish
suitable identifications between the pairs $Y \subset X$. (The same
approach applies later on, as we work fibered over $\Man$, by
promoting a~certain sheaf to a \emph{stack}.) This shifts the problem
to making a choice of such isomorphisms.

The choice made in~\cite{arXiv:1501.00967} is to say that morphisms in
$\mathcal C_0$ are maps $\varphi\colon Y \rightarrow Y^\prime$ that
have an~extension to a diffeomorphism between open neighborhoods of
$Y$ and $Y^\prime$ in $X$ and $X^\prime$. This makes the concrete
embedding of the hypersurface immaterial and ensures that the set of~isomorphism classes of objects is precisely the set of
$(d-1)$-dimensional manifolds, without any extra data. The issue with
this approach is that, while it works in the special case at~hand, it~does not generalize to arbitrary geometries $\mathcal{G}$, since we
are not allowed to restrict a~$\mathcal G$-structure on $X$ to one on the
hypersurface $Y$. Moreover, even for those $\mathcal G$ which make
sense in any dimension and allow restricting to hypersurfaces, it may
not be true that a $\mathcal G$-isomorphism is determined by its
data on a hypersurface.

Our choice for morphisms in $\mathcal C_0$ is designed to accommodate
for any geometry $\mathcal{G}$, in the sense of
Section~\ref{sec:geometries}, and is as follows. First we remark that
one of our axioms for a geometry is that one can always restrict it to
an open subset of a $\mathcal{G}$-manifold $X$. We then decree that a~mor\-phism between two pairs $X\subset Y$ and $X'\subset Y'$ in
$\mathcal C_0$ is determined by a $\mathcal{G}$-isomorphism defined on an
open neighborhood $U \subseteq X$ of $Y$, the underlying smooth map of
which sends $Y$ to $Y^\prime$; we~identify two such
$\mathcal{G}$-isomorphisms defined on, say, $U$ and $U^\prime$ if
they coincide on some smaller neighborhood $V \subset U \cap U^\prime$
of $Y$. Concisely, morphisms in $\mathcal{C}_0$ are \emph{germs of
 $\mathcal{G}$-isometries} at the marked hypersurfaces.

Further stages of the simplicial object $\mathcal C$ are constructed
in a similar fashion.

\subsection{Points versus germs of paths}

Our definition of $\mathcal C_0$ raises another difficulty, which is
generally unavoidable from our point of view: The set of isomorphism
classes of objects is \emph{huge}, as each different germ of the
geo\-met\-ric structure determines its own isomorphism class. This is
already true for the case of one-dimensional bordisms over a target
manifold $M$, which is considered in this paper. Here, an~object in
the bordism category can no longer be pictured as a point of $M$ (or,
more generally, a finite collection of such); instead, objects are {\em
 germs of paths} in $M$, a much larger space.

The main results of our paper (Theorems~\ref{MainThm_or} and~\ref{MainThm_nonor}) say that, at least in the one-dimensional case,
this does not make a difference: A field theory $Z \in 1\TFT(M)$ is
completely blind to the germ information, and its value on objects of
$\mathcal C$ contains no more data than that of a vector bundle over
$M$, as expected. This can be seen as a ``reality check'' for our
definition of geometric field theories.

A typical heuristic argument as to why the germs do not matter is that
the space of germs of~paths in $M$ deformation retracts to $M$. A
field theory $Z$ indeed defines a vector bundle on this space of
germs, viewed as a diffeological space. However, at this level of
generality, the familiar homotopy invariance of vector bundles breaks
down. So, instead, we will use the data assigned by $Z$ to higher
simplicial levels to show that $Z\vert_{\mathcal C_0}$ is determined
by a vector bundle on~$M$.

\subsection{Building in smoothness}
\label{sec:build-in-smoothness}

A second technical layer in our framework comes from the need to
formalize the idea that our field theories should be smooth. This is
already explained in detail in~\cite{MR2742432} and adapted to the
Segal approach in~\cite{arXiv:1501.00967}. The idea here is that a
\emph{smooth category} $\mathcal C$ is a complete Segal object in~the
$2$-category of (symmetric monoidal) stacks over the site of
manifolds; compare this with the preliminary description of above,
where $\mathcal C$ was a Segal object in the $2$-category of
(symmetric monoidal) groupoids. Thus, for each integer $n\geq 0$ and
each smooth manifold $S$ (a ``parameter space''), we have a groupoid
$\mathcal C_{n, S}$ of $S$-families of chains of $n$ composable
morphisms; this data is functorial in the variables $n$ and $S$.

To promote the bordism category to a smooth category, we need to fix
the meaning of ``$S$-family of bordisms'' $X/S$. In a nutshell, this
will be defined to be a submersion $\pi\colon X \to S$ such that each
fiber $\pi^{-1}(s)$ is a bordism in the previous sense. It remains to
explain what a~geometry~$\mathcal G$ is in this new context. Before,
$\mathcal G$ could be defined, technically, to be a sheaf or a~stack
on~the site $\Man$ of smooth manifolds; thus, to each $X$, corresponds
a set (or groupoid)~$\mathcal G(X)$ of~$\mathcal G$-structures on $X$.
To extend this to families, we introduce, in
Section~\ref{sec:smooth-field-theor}, a new site of \emph{families of
 $d$-dimensional manifolds}, denoted $\Fam^d$. Its objects are
submersions $\pi\colon X \to S$ with $d$-dimensional fibers. A
$d$-dimensional geometry is now simply a sheaf or stack on the site
$\Fam^d$. To illustrate this, consider the geometry $\mathcal G$ of
(fiberwise) Riemannian metrics; if $X/S \in \Fam^d$, then $\mathcal
G(X/S)$ is the set of inner products on the vertical tangent bundle
$\Ker(T\pi\colon TX \to TS)$.

\subsection{Appropriate families of vector spaces}
\label{subsec:fam_vs}
To promote the codomain of our field theories to smooth categories, we
must, likewise, specify what we mean by an $S$-family of vector
spaces. It is well known that for a topological field theory $Z$ the
vector space $Z(Y)$ associated to any object $Y$ of a topological
bordism category is finite-dimensional. So it is natural to declare
that an $S$-family of vector spaces is simply a finite-dimensional,
locally trival smooth vector bundle over $S$, and we will indeed
consider exclusively this case in Section~\ref{sec:classification}.
Notice that with this choice, a field theory $Z$, as a particular
example of a smooth functor, will assign to an $S$-family $X/S$ of
bordisms a linear map $Z(X/S)$ of vector bundles over $S$~-- an
$S$-family of linear maps.

For geometric field theories, the vector spaces $Z(Y)$ are typically
not finite-dimensional. For example, the quantum mechanical
description of a particle moving in a compact Riemannian manifold $N$
is given by a $1$-dimensional Riemannian field theory $Z$ which
associates to the object given by $Y=\{0\}\subset (-\epsilon,\epsilon)=X$
(with the standard Riemannian metric on the interval
$(-\epsilon,\epsilon)$) the ``vector space of functions on $N$''. Let
$X/S$ be an $S$-family of bordisms such that for every $s\in S$,
the fiber $X_s$ is a bordism from $Y$ to $Y'$ (where $Y$ and
$Y^\prime$ do not depend on $s$). Then the smoothness requirement for
$Z$ in particular says that the maps $S\to Z(Y')$ given by
$s\mapsto Z(X_s)v$ are smooth, for all $v \in Z(Y)$. If $Z(Y')$ is infinite
dimensional, a topology (or a bornological structure) on $Z(Y')$ is
needed to define a smooth map with target $Z(Y')$. In quantum
mechanics, the vector spaces are traditionally equipped with a Hilbert
space structure; for instance, in the case of a particle moving in a
Riemannian manifold, the vector space of functions on $N$ is
interpreted as $L^2(N)$, the Hilbert space of square-integrable
functions. However, as discussed in~\cite[Remark~3.15]{MR2742432}
there are difficulties formulating the smoothness of the functor $Z$
if the target category is built from families of Hilbert spaces;
instead, topological (or bornological) vector spaces are used. In the
quantum mechanics example, it is the space $C^\infty(N)$ of smooth
functions on $N$, equipped with its standard Fréchet topology.

It might seem appealing to define an $S$-family of topological or
bornological vector spaces to be a locally trivial bundle of such
spaces over $S$ with smooth transition functions. Such a definition,
though, has very undesirable consequences for the topology of the
space of field theories for a~fixed geometry $\mathcal G$; namely, for
any object $Y$ of $\GBord$, the isomorphism type of the topological
vector space $Z(Y)$ is invariant on the path component of $Z$ in the
space of field theories. The~heuristic reason is that if there is a
path of field theories $Z_t$, $t\in [0,1]$, then we have a family~$Z_t(Y)$ of topological vector spaces parametrized by $[0,1]$. If we
interpret that to mean a locally trivial vector bundle, then in
particular $Z_0(Y)$ and $Z_1(Y)$ are isomorphic. This is, in general,
an unexpected feature of field theories. For instance, it is
conjectured by Stolz and Teichner~\cite{MR2742432} that supersymmetric
Euclidean field theories provide cocycles for certain cohomology
theories; in particular, to $1\vert1$-dimensional correspond
$K$-theory classes. But~the dimension of a vector bundle representing
a $K$-theory class is not an invariant of the class (only ist \emph{virtual} dimension), so this should
also not be the case at the field theory level.

We choose to deal with this by dropping the local triviality condition
and defining an $S$-family of topological (or bornological) vector
spaces to be a sheaf of such spaces over $S$ which is a module over
the sheaf of smooth functions on $S$. This includes vector bundles
over $S$ by~asso\-ci\-ating to a vector bundle its sheaf of sections. It
then becomes a fact requiring proof that, under the additional
assumption that a field theory is topological, all the families of
vector spaces involved turn out to be locally trivial.

\subsection{Homotopy invariance considerations}

One focus in Berwick-Evans and Pavlov~\cite{arXiv:1501.00967} is to
endow the category of smooth categories (dubbed $C^\infty$-categories
there) with a model structure. This lets one conclude that the space
of field theories is insensitive to fine details in the definitions,
as long as everything remains weakly equivalent. It also lets one
compute with simplified models of the bordism category, since all that
matters is that the cofibrancy condition is met, and this is easy in
their model structure. We make no attempt to address questions of
homotopy meaningfulness in this paper; rather, our focus is on the
techniques dealing with the geometric situation.

Lastly, we remark that our bordism category $1\Bord^{\mathrm{or}}(M)$
possesses an obvious forgetful map to the bordism category of~\cite{arXiv:1501.00967}. Since in the model structure on the category
of $C^\infty$-categories considered there, weak equivalences are just
fiberwise equivalences of groupoids, the discussion above shows that
this forgetful map is \emph{not} a weak equivalence in this model
structure. Therefore, our result does \emph{not} follow from that in~\cite{arXiv:1501.00967}.

\section{Smooth functors and geometric field theories}\label{sec:smooth-field-theor}

Functorial field theories are functors from an appropriate bordism
category to a suitable target category. The bordisms in the domain
category might come equipped with geometric structures, in a sense to
be clarified in this section. After providing examples of geometric
structures in Section~\ref{subsec:ex_geometries} and recalling the
language of fibered categories and stacks in Section~\ref{subsec:stacks}, we provide a general definition of ``geometries''
in Section~\ref{sec:geometries}. In Section~\ref{sec:geom-bord-categ} we construct the geometric bordism category
$\GBord$, which is an example of a \emph{smooth category}, a concept
defined in Section~\ref{sec:smooth-categories}. In the final
Section~\ref{sec:geom-field-theor} we define geometric field
theories as ``smooth functors'' from $\GBord$ to a suitable smooth
target category.

\subsection{Examples of geometries}
\label{subsec:ex_geometries}

The goal of this subsection is to define what we mean by a {\em
 geometry} on smooth manifolds of a fixed dimension $d$, see
Definition~\ref{def:geom}. To motivate that abstract definition, we
begin by listing well-known structures on manifolds that will be
examples of ``geometries'', and distill their common features into our
Definition~\ref{def:geom}.

\begin{Examples}
 \label{ex:geom}
 The following are examples of ``geometries'' on a $d$-manifold $X$
 which we would like to capture in an abstract definition:
 \begin{enumerate}[1.]\itemsep=0pt
 \item A Riemannian metric or a conformal structure on $X$.
 \item A reduction of the structure group of the tangent bundle of
 $X$ to a Lie group $G$ equipped with a homomorphism $\alpha\colon
 G\to \opn{GL}(d)$. More explicitly, such a structure consists of
 a principal $G$-bundle $P\to X$ and a bundle map
 $\alpha_X\colon P\to \opn{Fr}(X)$ to the frame bundle of $X$
 (whose total space consists of pairs $(x,f)$ of points $x\in X$
 and linear isomorphisms $f\colon \R^d\to T_xX$). The~bundle map $\alpha_X$ is required to be $G$-equivariant, where
 the right action of $g\in G$ on $\opn{Fr}(X)$ is given by
 $(x,f)\mapsto (x,f\circ \alpha(g))$. Interesting special cases of
 reductions of the structure group include the following:
 \begin{enumerate}\itemsep=0pt
 \item[$(a)$] A $\opn{GL}^+(d)$-structure on $X$ is an orientation on
 $X$ (here and in the following three examples, the group $G$ is a
 subgroup of $\opn{GL}(d)$, and $\alpha\colon G\to \opn{GL}(d)$
 is the inclusion map).
 \item[$(b)$] An $\opn{SL}(d)$-structure on $X$ is a volume form on $X$.
 \item[$(c)$] An $\mathrm{O}(d)$-structure on $X$ is a Riemannian metric on $X$.
 \item[$(d)$] An $\mathrm{SO}(d)$-structure is Riemannian metric plus an
 orientation.
 \item[$(e)$] A $\opn{Spin}(d)$-structure on $X$ is a Riemannian metric
 plus a spin structure (here $\alpha$ is the composition of the
 double covering map $\opn{Spin}(d)\to \opn{SO}(d)$ and the
 inclusion $\opn{SO}(d)\to \opn{GL}(d)$).
 \end{enumerate}
 \item A rigid geometry is specified by a $d$-manifold $\M$
 (thought of as a ``model manifold'') and a~Lie group $G$ acting on
 $\M$ (thought of as ``symmetries'' of $\M$). Given this input, a
 $(G,\M)$-structure on $X$ is determined by the following data,
 which we refer to as a $(G,\M)$-atlas for~$X$:
 \begin{itemize}\itemsep=0pt
 \item an open cover $\{X_i\}_{i\in I}$ of $X$,
 \item embeddings $\phi_i\colon X_i\to \M$ for $i\in I$ (the charts
 of the atlas),
 \item group elements $g_{ij}\in G$ for $X_i\cap X_j\ne \varnothing$
 which make the diagram
 \[
 \begin{tikzcd}[column sep = small]
 &X_i\cap X_j\ar[dl,"\phi_j|_{X_i \cap X_j}"']\ar[dr,"\phi_i|_{X_i \cap X_j}"]&\\
 \M\ar[rr,"g_{ij}"]&&\M
 \end{tikzcd}
 \]
 commutative and satisfy a cocycle condition (these are the
 transition functions for the atlas). Two $(G,\M)$-atlases
 related by refinement of the covers involved define the same
 $(G,\M)$-structure on $X$ (as in the case of smooth atlases
 for $X$ defining the same smooth structure). Alternatively,
 analogous to smooth structures, $(G,\M)$-structures on $X$
 can be defined as \emph{maximal} $(G,\M)$-atlases for $X$.
 If $X$, $X^\prime$ are two manifolds with $(G, \M)$-structure,
 a morphism between them consists of a smooth map $f\colon X \to X^\prime$
 together with elements $h_{i^\prime i} \in G$ for each pair of charts $(X_i, \phi_i)$, $(X_{i^\prime}^\prime, \phi_{i^\prime}^\prime)$ with $f(X_i) \subset X_{i^\prime}^\prime$.
 such that $\phi_{i^\prime}^\prime \circ f = h_{{i^\prime}i} \cdot \phi_i$
 and subject to the coherence condition $h_{j^\prime j} \cdot g_{ji} = g^\prime_{j^\prime i^\prime} \cdot h_{i^\prime i}$.
 \end{itemize}
 Some concrete examples of rigid geometries are as follows:
 \begin{enumerate}\itemsep=0pt
 \item[$(a)$] For $G = \mathrm{SO}(d) \rtimes \R^d$, the Euclidean group of
 isometries of $\M=\R^d$, a $(G,\M)$-structure on $X$ can be
 identified with a flat Riemannian metric on $X$.
 \item[$(b)$] For $G = \mathrm{Spin}(d) \rtimes \R^d$ (where $\mathrm{Spin}(d)$
 acts on $\R^d$ through $\mathrm{SO}(d)$), a $(G, \M)$-structure on~$X$ consists of a flat Riemannian metric together with a spin structure.
 \item[$(c)$] For $\mathbb{M} = S^d$ and $G = \mathrm{Conf}(S^d)$, the
 group of conformal transformations of the sphere, a
 $(G,\M)$-structure on $X$ is a conformal structure on $X$.
 \item[$(d)$] For $\mathbb{M} = \R^d$ and $G = \mathrm{Aff}(d)$, the
 affine group, a $(G,\M)$-structure on $X$ is an affine
 structure on $X$.
 \item[$(e)$] If $\mathbb{M}$ is a simply connected manifold of constant
 sectional curvature $\kappa$ and isometry group~$G$, then a
 $(G,\M)$-structure on $X$ is a Riemannian metric on $X$ of
 constant curvature $\kappa$.
 \end{enumerate}
 Rigid geometries as described above are closely related to the notion
 of \emph{pseudogroups}, as developed by Cartan.
 The main difference is that the action of $G$ on $\M$
 is not required to be faithful (as, e.g., in the case of the spin group
 $\mathrm{Spin}(d)$ acting on $\R^d$ through $\mathrm{SO}(d)$).
 The~above notion was introduced by Stolz and Teichner, with an
 eye on supersymmetric field theories (see~\cite[Section~2.5]{MR2742432}).
 \item A smooth map $X\to M$ to some fixed manifold $M$.
 \item A principal $G$-bundle over $X$ (for a fixed Lie group $G$),
 or a principal $G$-bundle over $X$ with connection.
\end{enumerate}
\end{Examples}

\begin{Remark}
 In a physics context, the manifold $X$ is typically the relevant
 spacetime manifold and the geometry on $X$ is needed for the
 construction of some field theory. For example, a~Riemannian metric
 on $X$ allows the construction of the scalar field theory whose
 space of fields is the space $C^\infty(X)$ of smooth functions on
 $X$ and whose action functional is the ene\-rgy functional given by
 $S(f) := \int_X \lVert {\rm d}f\rVert^2 \, \vol$ (here $\vol$ is the volume
 form determined by the Riemannian metric. A fermionic analog of
 this field theory consists of fields which are spinors on~$X$; its~action functional is based on the Dirac operator. The construction
 of this field theory requires a Riemannian metric \emph{and} a spin
 structure on $X$, i.e., a reduction of the structure group to~$\opn{Spin}(d)$.
\end{Remark}

In many of the Examples~\ref{ex:geom}, the geometries on a fixed
$d$-manifold $X$ form just a \emph{set} (in~particular, in the
cases (1), $(2a)$, $(2b)$, $(2c)$, $(2d)$, $(3a)$, $(3c)$, $(3d)$, $(3e)$ and (4)). In~other cases, e.g., $(2e)$, $(3b)$, (5), there is more going on: these geometric structures can be
interpreted as the objects of a groupoid which contains non-identity
morphisms. For example, for a fixed group $G$, the principal
$G$-bundles $P\to X$ over $X$ form a groupoid, with the morphisms
from $P$ to $P'$ being the $G$-equivariant maps that commute with
the projection maps to $X$ (this is Example~\ref{ex:geom}(5)).

This suggests to think of a $\mathcal G$-structure on $X$ as an object of
a groupoid $\cG(X)$ associated to~$X$ (which might be \emph{discrete}
in the sense that the only morphisms in that groupoid are identity
morphisms, as in our examples (1), $(2a)$, $(2b)$, $(2c)$, $(2d)$, $(3a)$, $(3c)$, $(3d)$,
$(3e)$ and (4)). A~crucial feature of all our Examples~\ref{ex:geom} is
that the data of a geometry is \emph{local} in $X$ in a sense to be
made precise. For example, a Riemannian metric on $X$ is determined
by prescribing a Riemannian metric $g_i$ on each open subset $X_i$
belonging to an open cover $\{X_i\}_{i\in I}$ of $X$ in~such a way
that these metrics $g_i$, $g_j$ coincide on the intersection $X_i\cap
X_j$. In other words, the Riemannian metrics on $X$ form a {\em
 sheaf}. The same statement is true in our other examples of~geometric structures~$\cG(X)$, where the groupoid $\cG(X)$ is discrete.

In the case of non-discrete groupoids, for example if $\cG(X)$ is the
groupoid of principal $G$-bundles over $X$ (as in
Example~\ref{ex:geom}(5)), it is still true that $\cG(X)$ is {\em
 local in $X$}, but it is harder to formulate what that means. The
idea is that for any open cover $\{X_i\}_{i\in I}$ of $X$ the
groupoids associated to intersections of the $X_i$ determine the
groupoid $\cG(X)$ of principal bundles over~$X$, up to equivalence.
This is expressed by saying the groupoids $\cG(X)$ form a \emph{stack}
on the site of~mani\-folds. For the precise definition of stack we
refer the reader to Vistoli's survey paper~\cite[cf.~Definition~4.6]{MR2223406}, but it should be possible to follow our
discussion below without prior know\-ledge of stacks. In fact, we hope
that the following might motivate a reader not already familiar with
stacks to learn about them.

\subsection{Digression on stacks}
\label{subsec:stacks}

Our first example of a stack will be the stack $\Vect$ of vector
bundles. We note here the relevant structures.

\begin{itemize}\itemsep=0pt
\item For a fixed manifold $X$ let $\Vect(X)$ be the category of whose
 objects are smooth vector bundles $E\to X$ over $X$ and whose
 morphisms from $E\to X$ to $E'\to X$ are smooth maps $F\colon E\to
 E'$ which commute with the projection to $X$ and whose restriction
 $F_x\colon E_x\to E'_x$ to the fibers over $x\in X$ is a linear map
 for each $x\in X$.
\item Let $\Vect$ be the category whose objects are vector bundles
 $E\to X$ over some manifold $X$ and whose morphisms from $E\to X$ to
 $E'\to X'$ are pairs of smooth maps $\big(\hat f,f\big)$ for which the diagram
 \[
 \begin{tikzcd}
 E\ar[d]\ar[r,"\hat f"]&E'\ar[d]\\
 X\ar[r,"f"]&X'
 \end{tikzcd}
 \]
 is commutative, and the restriction $\hat f_x\colon E_x\to E'_{f(x)}$ of
 $\hat f$ to the fiber over $x$ is linear for all $x\in X$. Abusing
 notation, we often simply write $\phi\colon E\to E'$ for such a
 morphism $\phi=\big(\hat f,f\big)$.
\end{itemize}

There is an obvious functor $p\colon \Vect\to \Man$ to the category of
smooth manifolds that sends a vector bundle $E$ to its base space.
The category $\Vect(X)$ of vector bundles over $X$ is the \emph{fiber}
of $p$, i.e., the subcategory of $\Vect$ consisting of all objects
whose image under $p$ is $X$ and all morphisms of $\Vect$ whose image
under $p$ is the identity of $X$. The functor $p\colon \Vect\to \Man$
has two interesting properties:

\medskip\noindent
{\bf Existence of cartesian lifts.} Given a smooth map $f\colon X\to
 X'$ and a vector bundle $E'$ over~$X'$, we can form the pullback
 bundle $E:=f^*E'$ over $X$, which is the domain of a tautological
 morphism $\phi=\big(\hat f,f\big)\colon E\to E'$. The vector bundle map
 $\phi$ has the property that, for each $x\in X$, the linear map of
 fibers $\hat f_x\colon E_x\to E'_{f(x)}$ is an isomorphism.
 Morphisms with this property are called \emph{cartesian}, and the
 vector bundle morphism $\phi\colon E=f^*E'\to E$ is also referred to
 as the \emph{cartesian lift} of the morphism $f\colon X\to X'$.

 While the characterization of cartesian vector bundle morphisms
 $\phi\colon E\to E'$ as those which restrict to fiberwise
 isomorphisms is hands-on and concrete, it is more common to
 characterize them by the universal property we describe now. The
 advantage is that universal properties make sense in any category.

 A vector bundle morphism $\phi\colon E'\to E$ is \emph{cartesian}
 if, for any vector bundle map $\phi''\colon E''\to E$ and any map
 $f'\colon p(E'') \to p(E')$ such that $p(\phi'') = p(\phi) \circ
 f'$, there exists a unique vector bundle map $\phi'\colon E'' \to
 E'$ with $p(\phi') = f'$. This property can be expressed succinctly
 by saying that, given a commutative diagram consisting of the solid
 arrows in the diagram below, there exists a~unique morphism $\phi'$
 indicated by the dashed arrow that makes the whole diagram
 commutative. Here, $X = p(E)$, $f = p(\phi)$, etc., and by
 commutativity of the squares we mean that applying the functor
 $p\colon \Vect\to \Man$ to the top morphism in $\Vect$ gives the
 bottom morphism in $\Man$:
 \begin{equation}
 \label{eq:cartesian}
 \begin{tikzcd}[row sep = small]
 E''
 \ar[dd,mapsto]
 \ar[rrd,"\phi''" near end,bend left]
 \ar[rd,"\phi'",dashed]
 \\
 & E'\ar[r,"\phi"] & E
 \\
 X''\ar[rrd,"f''" near end,bend left]\ar[rd,"f'"]
 \\
 & X'\ar[from=uu,mapsto,crossing over]\ar[r,"f"] &
 X.\ar[from=uu,mapsto]
 \end{tikzcd}
 \end{equation}

 \smallskip\noindent
{\bf Descent property.} Let $f_i\colon X_i\to X$ be a collection of
 morphisms in $\Man$ that is a \emph{cover of~$X$} in the sense that
 all the $f_i$ are open embeddings and the union of the images
 $f_i(X_i)$ is all of~$X$. Then the category $\Vect(X)$ can be
 reconstructed, up to equivalence, from the categories $\Vect(X_i)$,
 $\Vect(X_i\cap X_j)$, and $\Vect(X_i\cap X_j\cap X_k)$ and the
 restriction functors between them. More precisely, from the diagram
 given by these categories and the restriction functors between them
 one can construct the \emph{descent category} $\Vect(\{X_i\to X\})$
 associated to the cover $\{X_i\to X\}$ (see~\cite[Section~4.1.2]{MR2223406}) and a restriction functor
 \[
 \Vect(X)\longrightarrow \Vect(\{X_i\to X\}),
 \]
 which is an equivalence.

It turns out that the existence of pullbacks and the descent property
can be formulated quite generally for functors $p\colon \cF\to \cS$ as
follows.

\begin{Definition}[cartesian morphism, fibration, prestack]
 Let $p\colon \cF\to \cS$ be a functor. A~morphism $\phi$ in the
 category $\cF$ is \emph{cartesian} if it satisfies the universal
 property expressed by the diagram~\eqref{eq:cartesian}; see also~\cite[Definition~3.1]{MR2223406}. The functor $p\colon \cF\to \cS$
 is called a \emph{Grothendieck fibration} and $\cF$ a \emph{category
 fibered over $\cS$} if, for any morphism $f\colon X'\to X$ in
 $\cS$ and object $E\in \cF$ with $p(E)=X$, there is a cartesian
 morphism $\phi\colon E'\to E$ with $p(\phi)=f$; see also~\cite[Definition~3.1]{MR2223406}.
 We say that a fibration $\cF \to \cS$ is a \emph{prestack} if every
 morphism of $\cF$ is cartesian.
\end{Definition}

\begin{Remark}
 That $\cF \to \cS$ is a {prestack} means that $\mathcal F \to
 \mathcal S$ is \emph{fibered in groupoids}, meaning that the
 subcategory $\mathcal F(S)$ of $\cF$ lying over $\id_S$ is a
 groupoid for every object $S$ of $\mathcal S$. Conversely, it turns
 out that a category fibered in groupoids is automatically a prestack~\cite[Proposition~3.22]{MR2223406}. To any prestack corresponds a~presheaf of groupoids on $\cS$, sending $S$ to $\cF(S)$. This sheaf
 depends on a choice of pullback for every object $F$ of $\cF$ along
 every morphism $f\colon S \to T$ in~$\cS$ (i.e., a cartesian arrow
 $\phi$ in $\cF$ with target $F$ and $p(\phi) = f$), but it is unique
 up to unique isomorphism.
\end{Remark}

As discussed above, for the functor $p\colon \Vect \to \Man$ the
cartesian morphisms in $\Vect $ are the vector bundle morphisms
$\phi\colon E'\to E$ that restrict to isomorphisms on fibers.
Moreover, given a smooth map $f\colon X'\to X$ between manifolds and a
vector bundle $E$ over $X$, the tautological bundle map $\phi$
from the pullback bundle $E':=f^*E$ to $E$ is a cartesian lift of $f$.
Hence $p\colon \Vect\to \Man$ is a Grothendieck fibration; in other
words, $\Vect$ is fibered over $\Man$.

The discussion of the descent property for $p\colon \Vect\to \Man$
above was based on the definition of a cover $\{X_i\to X\}$ of a
manifold $X$. So, before discussing descent in a general category
$\cS$, we~need to clarify what is meant by a ``cover'' of an object $X$
of~$\cS$.

\begin{Definition}[cover, Grothendieck topology, site~{\cite[Definition~2.24]{MR2223406}}] \label{def:cover-topology-site}
 Let $\cS$ be a category. A \emph{Grothendieck topology} on $\cS$ is
 the assignment, to each object $X$ of $\cS$, of a collection of sets
 of~morphisms $\{X_i\to X\}$, called \emph{covers of $X$}, so that
 the following conditions are satisfied:
 \begin{enumerate}[1.]\itemsep=0pt
 \item If $Y\to X$ is an isomorphism, the set $\{Y\to X\}$ is a
 cover.
 \item If $\{X_i\to X\}$ is a cover and $Y\to X$ is any morphism,
 then the pullbacks $X_i\times_XY$ exist, and the collection of
 projections $\{X_i\times_XY\to Y\}$ is a cover.
 \item If $\{X_i\to X\}$ is a cover and, for each index $i$, we have
 a cover $\{Y_{ij}\to X_i\}$, the collection of~composites
 $\{Y_{ij}\to X_i\to X\}$ is a cover of $X$ (here $j$ varies in a
 set depending on $i$).
 \end{enumerate}
 A category equipped with a Grothendieck topology is called a {\em
 site}.
\end{Definition}

\begin{Definition}[descent category, stack]
 If $\cF$ is a prestack over $\cS$, then there is a {\em
 descent category} $\cF(\{X_i\to X\})$ associated to a collection
 $\{X_i\to X\}$ of morphisms in $\cS$ (see~\cite[Section~4.1.2]{MR2223406}) and an associated functor
 \begin{equation}\label{eq:descent}
 \cF(X)\longrightarrow \cF(\{X_i\to X\}).
 \end{equation}
 A prestack $\cF\to \cS$ over a site $\cS$ is called a {\em stack},
 if, for every cover $\{X_i\to X\}$ of every object~$X$ in $\cS$,
 the functor~\eqref{eq:descent} is an equivalence.
\end{Definition}

For a stack, the groupoids $\cF(X_i)$ associated to the patches $X_i$, together with
transition data on double and triple intersections, determine $\cF(X)$.
In fact, this definition makes sense for general fibered categories, not only those fibered in groupoids (i.e., prestacks).
In this paper, we~will use the unqualified term \emph{stack} only for those
stacks which are fibered in groupoids, and say \emph{stack of
 categories} in the general case, when not all morphisms need to be
cartesian.
Similar to the case of sheaves, to any prestack,
there is a canonically associated stack,
its stackification~\cite[p.~18]{arXiv:math/0306176}.

Let $\Man^d$ be the category of $d$-dimensional manifolds and smooth maps.
We will always consider the Grothendieck topology on $\Man^d$ of jointly surjective open
embeddings. Precisely, a collection $\{X_i \to X\}$ is a cover, if each map
$X_i \to X$ is an open embedding and the images of the $X_i$ in $X$ cover $X$.

\begin{Definition}[geometry, preliminary!]
 \label{def:geom}
 A \emph{geometry on $d$-manifolds} is a stack $\cG\to \Man^d$ on~the
 site $\Man^d$ of manifolds of dimension $d$.
\end{Definition}

We end our digression on stacks with a few more general remarks.

\begin{Definition}[fibered functors, base-preserving natural
 transformations]
 Let $\mathcal F, \mathcal G \to \mathcal S$ be two Grothendieck
 fibrations. A functor $H\colon \mathcal F \to \mathcal G$ is called
 a \emph{fibered functor} if it commutes strictly with the
 projections to $\mathcal S$ and sends cartesian morphisms to
 cartesian morphisms. A~natu\-ral transformation $\xi\colon H \to K$
 between fibered functors is \emph{base-preserving} if, for any
 object $x \in \mathcal F$, the morphism $\xi_x\colon H(x) \to K(x)$
 maps to an identity morphism in $\mathcal S$.
\end{Definition}

\begin{Definition}[categories of (pre-)stacks]
For each site $\mathcal S$, we get a $2$-category $\PSt_{\mathcal S}$
of~Gro\-thendieck fibrations, fibered functors and base-preserving
natural transformations.
 The full subcategory of stacks will be denoted by $\St_{\mathcal S}$.
 We will omit the subscript when $\mathcal S = \Man$ is the site of
 smooth manifolds.
\end{Definition}

\subsection{Geometries in families}
\label{sec:geometries}

The preliminary Definition~\ref{def:geom} satisfactorily captures the
contravariance and locality aspects of a geometric structure.
However, as discussed in Section~\ref{sec:build-in-smoothness}, it is
crucial to work with \emph{families} of smooth manifolds. In
particular, we need to talk about geometric structures on families of~$d$-manifolds. This is formalized in Definition~\ref{def:geom-fam}
below by replacing the category $\Man^d$ in the preliminary definition
by the category $\Fam^d$ of families of $d$-dimensional manifolds,
equipped with a suitable Grothendieck topology.

\begin{Definition}[families of manifolds]
Denote by $\Fam$ the category of families of smooth manifolds, where an object, typically denoted by $X/S$, is simply a submersion $X \to S$, and a~morphism $X'/S' \to X/S$ is a fiberwise open embedding
\begin{equation*}
 \begin{tikzcd}
 X' \ar[r, "F"]\ar[d] & X \ar[d]
 \\ S' \ar[r, "f"] & S.
 \end{tikzcd}
\end{equation*}
By that we mean that the diagram commutes and the map $X' \to S'
\times_S X$ is an open embedding.
We denote by $\Fam^d$ the subcategory of families with $d$-dimensional fibers.
\end{Definition}

To turn $\Fam$ into a site, we declare a cover of the
object $X/S \in \Fam$ to be a collection of~morphisms $\{X_i/S_i \to X/S\}_{i \in I}$ such that the images of the $X_i$ form an open cover of $X$. This
satisfies the axioms of a Grothendieck topology: it is clear that
covers of covers determine a~cover, and we check the existence and
stability of base changes (condition~(2) of
Definition~\ref{def:cover-topology-site}) in the following lemma.

\begin{Lemma}
 If $\{ X_i/S_i \to X/S \}$ is a cover and $Y/T \to X/S$ is any
 morphism, then the fiber products $(X_i/S_i) \times_{(X/S)} (Y/T)$
 exist and determine a cover of $Y/T$.
\end{Lemma}

\begin{proof}
 Write $Y_i = X_i \times_X Y$, $T_i = S_i \times_S T$.
 Both are manifolds since the maps $X_i \to X$ and~$S_i \to S$ are submersions:
 The first by the requirement that the $X_i$ form an open cover of~$X$; to~see that the latter is, observe that the composition $X_i \to X \to S$ is a submersion, %\pagebreak
 which equals $X_i \to S_i \to S$, hence $S_i \to S$ must be a submersion, too.
 We now have a diagram as follows:
 \begin{equation} \label{CubeDiagram}
 \begin{tikzcd}[sep = scriptsize]
 Y_i \ar[rr] \ar[dd, dashed] \ar[rd] %, "\circ" mark
 && Y \ar[rd] \ar[dd, two heads]
 \\
 & X_i \ar[rr, crossing over] %"\circ"{mark, pos=1/4},
 && X \ar[dd, two heads]
 \\
 T_i \ar[rr] \ar[rd]
 && T \ar[rd]
 \\
 & S_i \ar[rr] \ar[from=uu, two heads, crossing over]
 && S.
 \end{tikzcd}
 \end{equation}
 The dashed map is obtained
 by the cartesian property of the bottom face (containing $S_i$ and~$T$), which implies that all faces of the cube commute.
 To see that $Y_i \to T_i$ is a submersion,
 let $(s_i, t) \in T_i = S_i \times_S T$ be a point
 that is the image of the point $(x_i, y) \in Y_i = X_i \times_X Y$.
 Further, consider a tangent vector to $(s_i, t)$, represented by a pair of paths $(\gamma_{S_i}, \gamma_T)$, where $\gamma_{S_i}\colon I \to S_i$, $\gamma_{S_i}(0) = s_i$, $\gamma_T\colon I \to T$, $\gamma_T(0) = t$ are paths that coincide after passing to $S$.
 Since $Y$ is a~submersion, we can find a lift $\gamma_Y\colon I^\prime \to Y$ (with $I^\prime \subset I$) of $\gamma_T$ with ${\gamma}_Y(0) = y$.
 Let $\tilde{\gamma}_X$ be the image of this lift to $X$.
 By construction, ${\gamma}_X(0)$ lies in the image of the map $X_i \to X$,
 hence (because $X_i$ is an open embedding), we can find a path $\gamma_{X_i}\colon I^{\prime\prime} \to X_i$ (with $I^{\prime\prime} \subset I^\prime$) that maps to $\gamma_X$.
 Now $(\gamma_{X_i}, \gamma_Y)$ represents a tangent vector to $(x_i, y)$ that maps to $(s_i, t)$ under the map~$Y_i \to S_i$.

 To see that $Y_i$ is a fiberwise open embedding, we have to show that $Y_i \to T_i \times_T Y$ is an open embedding.
 To see that it is injective, let $(x_i, y)$ and $(x_i^\prime, y^\prime) \in X_i \times_X Y = Y_i$ be two points that are mapped to the same point in $T_i \times_T \times Y$.
 First, it follows that $x_i$ and $x_i^\prime$ are mapped to the same point in $X$; since $X_i \to X$ is an embedding, we have that $x_i = x_i^\prime$.
 Secondly, since $(x_i, y)$ and $(x_i^\prime, y^\prime)$ coincide in $T_i \times_T Y$, we must have $y = y^\prime$.
 This shows injectivity.
 The map $Y_i = X_i \times_X Y \to T_i \times_T Y$ is a submersion by a similar argument as above. To see that it is also an immersion, take a tangent vector represented by a tuple $(\gamma_{X_i}, \gamma_Y)$ of curves $\gamma_{X_i}\colon I \to X_i$ and $\gamma_Y\colon I \to Y$ that coincide after passing to $X$.
 Then the pushforward of this tangent vector is represented by $(\gamma_{S_i}, \gamma_Y)$, where $\gamma_{S_i}$ is the image of $\gamma_{X_i}$ under the map $X_i \to S_i$.
 Assume that this pushforward is zero, i.e., $\dot{\gamma}_{S_i}(0) = 0$ and $\dot{\gamma}_Y(0) = 0$.
 Then because $\gamma_{X_i}$ and~$\gamma_Y$ are mapped to the same curve $\gamma_X\colon I \to X$ and $X_i \to X$ is an embedding, we have that also $\dot{\gamma}_{X_i}(0) = \dot{\gamma}_X(0) = \dot{\gamma}_Y(0) = 0$.
 Hence the tangent vector represented by $(\gamma_{X_i}, \gamma_Y)$ is zero as well and $Y_i\to T_i \times_T Y$ is an immersion.
 In total, we obtain that it is an open embedding.

 It only remains to see that $Y_i/T_i$ has the universal property of
 the fiber product $(X_i/S_i) \times_{(X/S)} (Y/T)$. In other words, given an
 arbitrary family $Z/U$ and maps $Z/U \to X_i/S_i$, $Z/U \to Y/T$ which
 agree on $X/S$, there exists a unique morphism $Z/U \to Y_i/T$ making the
 diagram
 \begin{equation} \label{UniversalPropertyDiagram}
 \begin{tikzcd}
 Z/U \ar[rrd, bend left=30] \ar[ddr, bend right=30] \ar[dr, dashed]& & \\
 & Y_i/T \ar[r] \ar[d] & Y/T \ar[d]\\
 & X_i/S_i \ar[r] & X/S
 \end{tikzcd}
 \end{equation}
 commute.
 By the cartesian property of two of the
 squares of the cube~\eqref{CubeDiagram}, there exist unique maps $Z \to Y_i$ and $U \to T_i$.
 We claim that these determine a morphism in $\Fam$, that is, the diagram
 \begin{equation*}
 \begin{tikzcd}
 Z \ar[r] \ar[d] & Y_i \ar[d]\\
 U \ar[r] & T_i
 \end{tikzcd}
 \end{equation*}
 commutes and the top map is a fiberwise open embedding. Both maps $Z
 \to T_i$ agree when postcomposed with $T_i \to S_i$ respectively
 $T_i \to T$, so they agree by the universal property of~$T_i = S_i
 \times_S T$. The map $Z \to Y_i$ is an a fiberwise open embedding
 because so is its composition $Z \to Y_i \to Y$ with another fiberwise open embedding.
 Hence this morphism fits in as the dashed morphism in~\eqref{UniversalPropertyDiagram}.
 By uniqueness of the maps $Z \to Y_i$ and $U \to T_i$, this is the unique morphism
 with this property.
\end{proof}

The above lemma shows that our notion of cover defines a Grothendieck
topology, which turns $\Fam$ into a site. By restricting to those
families $X/S$, where the fibers $X_s$, $s \in S$ are all
$d$-dimensional, we get a subcategory, $\Fam^d$. Since our covers do
not mix fiber dimensions, restricting to $d$-dimensional covers
turns also $\Fam^d \subset \Fam$ into a site.
This allows to talk about sheaves and stacks on $\Fam^d$.

We are now ready to give the main definition of this section.

\begin{Definition}[geometry]
 \label{def:geom-fam}
 A $d$-dimensional \emph{geometry} is a stack $\G$ on the site
 $\Fam^d$ of families of manifolds with $d$-dimensional fibers. By
 an $S$-family of $\G$-manifolds we will mean a family $X/S \in
 \Fam^d$ together with an object of $\G(X/S)$.
\end{Definition}

To each family $X/S$ are associated a natural \emph{relative tangent
 bundle} $T(X/S) = \Ker(TX \to TS)$ and variations: the relative
cotangent bundle $T^\vee(X/S) = \Coker\big(T^\vee S \to T^\vee X\big)$, their
tensor powers, etc. For emphasis, we sometimes call their sections
\emph{fiberwise} vector fields, differential forms, etc. This allows
us to define fiberwise versions (or ``families'') of many familiar
structures, such as Riemannian metrics, symplectic and complex
structures, connections on a~principal bundle, and so on. For
instance, a~Riemannian metric on $X/S$ is a positive-definite section
of the second symmetric power of $T^\vee (X/S)$. A family of
connections on a vector bundle $V \to X$ is a~differential operator
$\nabla\colon C^\infty(V) \to C^\infty\big(T^\vee (X/S) \otimes V\big)$
satisfying a version of the Leibniz rule involving the fiberwise
exterior derivative ${\rm d}\colon C^\infty(X) \to \Omega^1(X/S)$. Thus,
$\nabla$ allows us to perform parallel transport only along the fibers
of the submersion $X \to S$.

It is now mostly straightforward to adapt Examples~\ref{ex:geom} to
geometries in families. We spell this out in two cases.

\begin{Example}[families over a manifold $M$] \label{ExManifoldsOver}
 Any manifold $M$ represents a geometry on $d$-dimensional manifolds
 ($d$ arbitrary): $S$-families are manifolds $X/S$ together with a
 smooth function $\gamma\colon X \rightarrow M$, and morphisms $(X'/S',
 \varphi') \to (X/S, \varphi)$ are maps $F\colon X'/S' \to X/S$ such
 that $\gamma' = \gamma \circ F$. This in fact defines a sheaf on
 $\Fam^d$.
\end{Example}

\begin{Example}[rigid geometries]
 \label{Ex:RigidGeometries2}
 We recast the family version of rigid geometries~\cite[Section~2.5]{MR2742432} in the language of this paper. As
 in Example~\ref{ex:geom}(3), fix a $d$-dimensional model manifold~$\mathbb M$ and a Lie group $G$ acting on it. Then we define
 $\mathcal G$, the \emph{stack of $(G,\mathbb M)$-atlases}, to be the
 stackification of the prestack on $\Fam^d$ described as follows:
 \begin{enumerate}[1.]\itemsep=0pt
 \item An object lying over $X/S$ is given by a fiberwise open
 embedding
 \begin{equation*}
 \begin{tikzcd}
 X \ar[r,"\phi"]\ar[d, "p"] & \mathbb M\\S
 \end{tikzcd}
 \end{equation*}
 or, in other words, an open embedding $(p, \phi)\colon X \to S
 \times \mathbb M$.
 \item A morphism lying over $(f, F)\colon X^\prime/S^\prime \rightarrow
 X/S$ is given by a map $g\colon S'' \to G$ such that the diagram
 \begin{equation*}
 \begin{tikzcd}
 X' \ar[r, "F"]\ar[d, "{(p', \phi')}"] & X\ar[d, "{(p, \phi)}"]\\
 S'' \times \mathbb M \ar[r, "\bar g"] & S \times \mathbb M
 \end{tikzcd}
 \end{equation*}
 commutes, where $S^{\prime\prime} = p'(X') \subset S'$ and $\bar
 g\colon (s,x) \mapsto (f(s), g(s)\cdot x)$ is the map induced by
 $f$ and the action by $g$.
 \end{enumerate}
 The composition of morphisms is determined by composition of the $\bar
 g$. Note that every morphism is cartesian.

 The usual stackification procedure~\cite[p.~18]{arXiv:math/0306176}
 exactly recovers the more concrete definition of a rigid geometry
 given in~\cite[Definition~2.33]{MR2742432}: A section of the stack
 $\mathcal G$ over $X/S$, which we call a $(G, \mathbb M)$-atlas, is
 given by the following data: (1) a cover $\{ X_i/S_i \to X/S \}_{i
 \in I}$, (2) fiberwise embeddings $\phi_i\colon X_i \to \mathbb M$
 for each $i \in I$ (the charts of the atlas), and (3) transition
 functions $g_{ij}\colon S_i \times_S S_j \supset p(X_i \times_X X_j)
 \to G$ relating appropriate restrictions of $\phi_i$ and $\phi_j$,
 and satisfying a cocycle condition. Morphisms between atlases based
 on the same cover $\{ X_i/S_i \}_{i \in I}$ are given by collections
 of maps $h_i\colon S_i\to G$, $i\in I$, which interpolate the charts
 $X_i \to \mathbb M$. Moreover, atlases related by a refinement of
 covers must be declared equivalent; this is taken care of by the
 stackification machinery.
\end{Example}

\subsection{Simplicial prestacks and smooth categories}
\label{sec:smooth-categories}

It is well known (see, e.g.,~\cite[Section~2]{MR232393}) that a simplicial set $\mathcal C\colon
\Delta^{\mathrm{op}} \rightarrow \Set$ is equivalent to the nerve of a
category if, and only if, the \emph{Segal maps}
\begin{equation}
 \label{eq:Segal}
 (s_n^*, \dots, s_1^*) \colon\ \mathcal C_n
 \longrightarrow \mathcal C_1 \times_{\mathcal C_0} \cdots
 \times_{\mathcal C_0} \mathcal C_1
\end{equation}
are bijections for $n \geq 2$. Here, $s_i\colon [1] \rightarrow [n]$,
$i=1, \dots, n$, is the morphism sending $0 \mapsto i-1$ and $1
\mapsto i$, and fiber products are taken over the maps $d_0^*,d_1^*\colon \mathcal C_1 \to \mathcal C_0$ induced by the two maps
$d_0,d_1\colon [0] \rightarrow [1]$. (For references, see, e.g.,~\cite{MR1804411, MR232393}.)

This observation allows us to internalize the notion of a category in
other ambient (higher) cate\-gories. In this paper, we would like to
talk about categories endowed with a notion of~''smooth families'' of
objects and morphisms. Thus, we take as ambient the $2$-category
$\PSt$ of~prestacks on $\Man$.

\begin{Definition}[simplicial prestack]
A simplicial prestack (on manifolds) is a pseudofunctor
$\mathcal C\colon \Delta^{\mathrm{op}} \to \mathcal \PSt$.
\end{Definition}

\begin{Remark}
 Here the simplex category $\Delta$ is regarded as a $2$-category
 with only trivial $2$-morphisms, and all constructions are performed
 in the realm of bicategories. That $\mathcal C$ is a pseudofunctor
 then means that for two composable
 morphisms $\eta$, $\kappa$ in $\Delta$, the induced morphisms of
 stacks $\kappa^* \eta^*$ and $(\eta\kappa)^*$ agree only up to a
 coherent natural isomorphism, which is part of the data of~$\mathcal
 C$.
\end{Remark}

A smooth category will be a simplicial prestack satisfying suitable
conditions. Before introducing them, we fix some terminology.
Condition~(2) below assures that the simplicial set $n \mapsto
h_0 \mathcal C_n(S)$ is equivalent to the
nerve of a category $C$ (where $h_0 \mathcal C_n(S)$ is the set of
isomorphism classes of objects in $\mathcal C_n(S)$);
we call an object of $\mathcal C_1(S)$ an \emph{equivalence} if it
represents an invertible morphism in $C$.

\begin{Definition}[smooth category]
 \label{def:sm-cat}
 A \emph{smooth category} $\mathcal C$ is a simplicial stack
 $\mathcal{C}\colon \Delta^{\mathrm{op}} \rightarrow \St$ such that
 \begin{enumerate}[(1)]\itemsep=0pt
% \item for each $n$, $\mathcal C_n$ satisfies the stack condition,
 \item the Segal maps~\eqref{eq:Segal} are equivalences of stacks,
 and
 \item the degeneracy map $\mathcal C_0 \to \mathcal C_1$ gives an
 equivalence of the domain with the full substack of~equi\-valences
 in $\mathcal C_1$.
 \end{enumerate}
We will refer to morphisms and $2$-morphisms between smooth categories
as \emph{smooth functors} and~\emph{smooth natural transformations},
respectively.
\end{Definition}

The above conditions are modeled on \emph{complete Segal spaces},
which extends the nerve construction explained above, to give
a model for $(\infty, 1)$-categories (see~\cite{MR1804411} for further details).
Condition~(1) is the \emph{Segal condition}, (2) is the \emph{completeness condition}.
In other words, smooth stacks are complete Segal objects in $\St$.

\begin{Remark}
 The fiber products appearing in the definition (which
 are taken using $d_0^*, d_1^* $: $\mathcal{C}_1 \rightarrow
 \mathcal{C}_0$) are in the bicategorical sense, that is, they are
 what is sometimes called a homotopy fiber product.
\end{Remark}

\begin{Example}[smooth categories from smooth stacks]
 \label{Ex:SmoothCatFromStacks}
 Our most interesting examples of smooth categories will be the
 geometric bordism categories constructed below. However, to get the
 first examples, we now provide a way to construct a smooth category
 from a smooth stack. This is a version of Rezk's classification
 diagram construction~\cite{MR1804411}.

 Let $\mathcal{C}$ be a stack of categories (so that $\mathcal{C}(S)$
 does not need to be a groupoid). In our applications, $\mathcal{C}$
 will be the stack of vector bundles or a stack of sheaves of
 $C^\infty$-modules as in Section~\ref{sec:admissible-stacks}. We
 then construct a smooth category $\mathcal{C}_\bullet$ from this
 input as follows.

 Objects of $\mathcal C_n$ lying over $S \in \Man$ are tuples $(C_n,
 \dots, C_0; f_n, \dots, f_1)$, where the $C_j$ are objects of
 $\mathcal{C}(S)$ and $f_j\colon C_{j-1} \rightarrow C_j$ are morphisms in
 $\mathcal{C}(S)$ (i.e., morphisms in $\mathcal{C}$ covering the
 identity on $S$). Morphisms from an object $(C_n, \dots, C_0; f_n,
 \dots, f_1)$ over $S$ to an object $(C_n^\prime, \dots, C_0^\prime;
 f_n^\prime, \dots, f_1^\prime)$ over $T$ covering $f\colon S
 \rightarrow T$ are tuples $(\alpha_n, \dots, \alpha_0)$, where
 $\alpha_j\colon C_j \rightarrow C_j^\prime$ are cartesian arrows covering
 $f$ such that the diagram
 \begin{equation*}
 \begin{tikzcd}
 C_n \ar[d, "\alpha_n"']
 & \ar[l, "f_n"'] \cdots
 & C_1 \ar[l, "f_2"'] \ar[d, "\alpha_1"]
 & C_0 \ar[l, "f_1"'] \ar[d, "\alpha_0"]
 \\
 C^\prime_n
 & \ar[l, "f_n^\prime"] \cdots
 & C_1^\prime \ar[l, "f_2^\prime"]
 & C_0^\prime \ar[l, "f_1^\prime"]
 \end{tikzcd}
 \end{equation*}
 commutes. The simplicial structure of $\mathcal C_\bullet$ is so
 that face maps perform composition of morphisms and degeneracies
 insert identities. More explicitly, a morphism $\kappa\colon [n]
 \rightarrow [m]$ in $\Delta$ induces the functor $\kappa^*\colon
 \mathcal{C}_m\rightarrow \mathcal{C}_n$ with
 \begin{equation*}
 \kappa^*(\V_m, \dots, \V_0; f_m, \dots, f_1)
 = \big(V_{\kappa(n)}, \dots, \V_{\kappa(0)};
 f_n^\prime, \dots, f_1^\prime\big),
 \end{equation*}
 where
 \begin{equation*}
 f^\prime_j =
 \begin{cases}
 f_{\kappa(j)} \cdots f_{\kappa(j-1)+1}
 & \text{if}\quad \kappa(j-1) < \kappa(j),
 \\
 \id
 & \text{otherwise},
 \end{cases}
 \end{equation*}
 and
 \begin{equation*}
 \kappa^*(\alpha_m, \dots, \alpha_0)
 = \big(\alpha_{\kappa(n)}, \dots, \alpha_{\kappa(0)}\big).
 \end{equation*}
 This gives a (strict) functor $\mathcal C_\bullet\colon
 \Delta^{\mathrm{op}} \rightarrow \SmSt$, and it is obvious that it
 satisfies the Segal condition.
\end{Example}

\begin{Definition}[strictness]
 We say that a simplicial prestack $\mathcal{C}$ is \emph{strict} if
 it is a strict functor, that is, the natural isomorphisms
 $\kappa^*\eta^* \cong (\eta\kappa)^*$ are all identities. We say
 that a smooth functor between two strict smooth categories is
 \emph{strict} if it commutes on the nose with the structure maps in
 $\Delta$, as a natural transformation of strict functors.
\end{Definition}

The following is an easy structure result for the examples just
constructed, whose proof we~omit.

\begin{Lemma}
 \label{LemmaOnNaturalTransformations}
 Let $\mathcal{B}$ be a strict simplicial prestack and let $\mathcal V$
 be a smooth category of the type constructed in
 Example~$\ref{Ex:SmoothCatFromStacks}$. Let $Z, Y\colon \mathcal{B}
 \rightarrow \mathcal V$ be two strict functors. Then the map
 \begin{equation*}
 \mathrm{Nat}(Z, Y) \longrightarrow \mathrm{Nat}(Z_0, Y_0)
 \end{equation*}
 that restricts a smooth natural transformation to the corresponding
 $2$-morphisms of stacks at simplicial level zero is injective.
 Moreover, the map
 \begin{equation*}
 \mathrm{Nat}(Z, Y) \longrightarrow \mathrm{Nat}(Z_1, Y_1)
 \end{equation*}
 that restricts to simplicial level one is an isomorphism.
\end{Lemma}

\subsection{Symmetric monoidal structures}
%\label{sec:monoidal-structures}

To talk about field theories, we need to endow our smooth categories
with symmetric monoidal structures. This requires first to define
symmetric monoidal stacks.

\begin{Definition}[monoidal fibered category] A \emph{symmetric monoidal fibered
category} is a fibered category $\cF \to \cS$ with a (fibered) tensor product functor
\begin{equation*}
 \otimes\colon\ \cF \times_{\cS} \cF \longrightarrow \cF,
\end{equation*}
a (fibered) unit functor $\epsilon \colon \cS \to \cF$ (where $\cS \to
\cS$ denotes the trivial fibered category), together with a collection
of natural transformations (an associator $\alpha\colon (X \otimes Y)
\otimes Z \to X \otimes (Y \otimes Z)$, left and right unitors
$\lambda \colon 1 \otimes X \to X$, $\rho \colon X \otimes 1 \to X$
and a braiding $B \colon X \otimes Y \to Y \otimes X$). These data are
required to be compatible in the sense that they turn each fiber category
$\cF(S)$ into a symmetric monoidal category. A \emph{symmetric
 monoidal stack} is a symmetric monoidal fibered category that is
also a stack.
 \end{Definition}

Symmetric monoidal smooth stacks together with (strong) symmetric monoidal functors and natural transformations form a bicategory which we shall denote $\SmSt^\otimes$.

\begin{Definition}[symmetric monoidal smooth category]
\label{DefSymmMonSmCat}
 A \emph{symmetric monoidal structure} on a simplicial prestack
 $\mathcal{C}\colon \Delta^{\mathrm{op}} \rightarrow \SmSt$ is a lift
 of $\mathcal C$ to a pseudofunctor $\Delta^{\mathrm{op}} \to
 \SmSt^\otimes$. Symmetric monoidal smooth functors and natural
 transformations are likewise defined as $1$- and $2$-morphisms of
 simplicial objects in symmetric monoidal stacks.
 A \emph{symmetric monoidal smooth category} is a smooth category (Definition~\ref{def:sm-cat}) with a
 symmetric monoidal structure.
\end{Definition}

\begin{Example}
 If, in Example~\ref{Ex:SmoothCatFromStacks}, we start with a
 symmetric monoidal smooth stack as input, the result will naturally
 be a symmetric monoidal smooth category.
\end{Example}

\subsection{Geometric bordism categories}
\label{sec:geom-bord-categ}

Let $\G$ be a geometry for $d$-dimensional manifolds, i.e., a stack on
$\Fam^d$. In this section, we will define our symmetric monoidal
smooth category (Definition~\ref{DefSymmMonSmCat}) of $\G$-bordisms, denoted by $\GBord$. We start by
defining a smooth symmetric monoidal stack $\GBord_n$ for every object $n \in \Z_{\geq
 0}$. Afterwards, we define maps of symmetric monoidal stacks
\begin{equation*}
 \kappa^*\colon\ \GBord_n \longrightarrow \GBord_m
\end{equation*}
for every morphism $\kappa\colon [m] \rightarrow [n]$ in $\Delta$. These
maps will satisfy $(\kappa \circ \eta)^* = \eta^* \circ \kappa^*$, so
that we obtain a strict functor $\GBord\colon\Delta^{\mathrm{op}}
\rightarrow \St^\otimes$. At the end of the subsection, we comment on the
smooth category property of $\GBord$.

\begin{Definition}[the stack $\GBord_n$]
\noindent An \emph{object} of the stack $\GBord_n$ lying over a manifold $S$
consists of the following data:
 \begin{enumerate}\itemsep=0pt
 \item[(O1)] A family of $d$-dimensional manifolds, that is,
 an object $X/S \in \Fam^d$.
 \item[(O2)] A $\mathcal G$-structure on $X/S$, that is, an object
 $G_{X/S}$ of $\G({X/S})$.
 \item[(O3)] Smooth functions $\rho_a\colon X \rightarrow \R$ for $a
 = 0, \dots, n$, subject to the following conditions:
 \begin{enumerate}[$(a)$]\itemsep=0pt
 \item $\rho_0 \geq \rho_1 \geq \dots \geq \rho_{n}$,
 \item whenever $\rho_a(x) = 0$, $d\rho_a$ does not vanish on the vertical tangent space
 $T_x (X/S)$,
 \item the subspaces
 \begin{equation}
 \label{SetsXab}
 X_a^b := \{ x \in X \mid \rho_{a}(x) \geq 0 \geq \rho_{b}(x) \},
 \end{equation}
 are proper over $S$ for all $0 \leq a \leq b \leq n$.
 \end{enumerate}
 \end{enumerate}
 We will abuse notation and abbreviate $X/S$ for objects, keeping in
 mind that the collection $\{\rho_a\}$ and the object $G_{X/S}$ are
 also part of the data. The subspace $X_0^n$ defined above will be
 called the \emph{core} of $X/S$.

 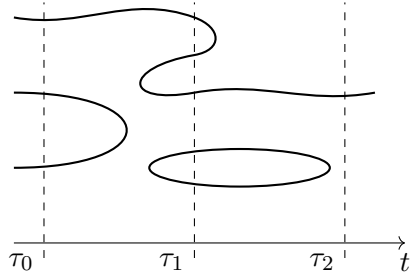
\begin{figure}
 \centering
 \begin{tikzpicture}[x=2cm]
 \draw[bordism] (-.2, 3) to[out=-10, in=160]
 (1, 3) to[out=-20, in=10, looseness=2]
 (1, 2.5) to[out=190, in=190, looseness=5]
 (1, 2) to[out=10, in=190]
 (2.2, 2);
 \draw[bordism] (-.2, 2) .. controls +(0:1) and +(0:1) .. (-.2, 1);
 \draw[bordism] (1.3, 1) circle[x radius = .6, y radius = .25];
 \draw[cut] (0, -.2) -- (0, 3.2);
 \draw[cut] (1, -.2) -- (1, 3.2);
 \draw[cut] (2, -.2) -- (2, 3.2);
 \draw[axis, ->] (-.2, 0) -- (2.4, 0) node[below]{$t$};
 \draw \foreach \i in {0,1,2} {(\i,0) node[below left] {$\tau_\i$}};
 \end{tikzpicture}
 \caption{An object of $1\Bord_2(\pt)$, representing a pair of
 composable $1$-dimensional bordisms. The cut functions are
 $\rho_i = t - \tau_i$.}
% \label{fig:1-bord}
\end{figure}

 \emph{Morphisms} of the stack $\GBord_n$ are going to be
equivalence classes of maps, where a \emph{map}
\begin{equation*}
\big(X/S;\rho_0, \dots, \rho_n; G_{X/S}\big) \longrightarrow
\big(Y/T; \rho_0^\prime, \dots,\rho_n^\prime; G_{Y/T}\big)
\end{equation*}
lying over a morphism $f\colon S \rightarrow
T$ consists of the following data:
\begin{enumerate}\itemsep=0pt
\item[(M1)] An open neighborhood $U$ of the core $X_0^n$ such that for some $\varepsilon >0$,
\begin{gather*}
 U \supseteq \rho_n^{-1}(-\infty, \varepsilon) \cap \rho_0^{-1}(-\varepsilon, \infty)
\end{gather*}
and
\begin{gather*}
F(U) \supseteq (\rho_n^\prime)^{-1}(-\infty, \varepsilon) \cap (\rho_0^\prime)^{-1}(-\varepsilon, \infty).
\end{gather*}
\item[(M2)] A smooth map $F\colon U \rightarrow Y$ covering $f$, which is
a fiberwise open embedding onto a neighborhood of the core $Y_0^n$,
 such that for each $0 \leq a \leq n$, there exist positive smooth functions~$\zeta_a$ on $U$ such that $F^*\rho^\prime_a = \zeta_a\rho_a|_U$.
\item[(M3)]
 A morphism $\varphi\colon G_{X/S}|_{U/S} \to G_{Y/T}|_{F(U)/T}$
covering the morphism $(F, f) \colon U/S \rightarrow F(U)/T$ in $\Fam^d$.
\end{enumerate}
We declare two maps to be \emph{equivalent} if they have a common
restriction to a smaller neighborhood of the core. Here a {\em
 restriction} of a map $(f, F, \varphi)$ is a map of the form
 $(f, F|_V, \varphi|_V)$
 for some open neighborhood $V \subseteq U$ of the core satisfying (M1).
\end{Definition}

\begin{Remark}
 The condition relating $F^*\rho'_a$ and $\rho_a|_U$ in (M1) is
 equivalent to saying that these two functions have the same sign.
 We express it in terms of the positive function $\zeta_a$ so that
 our definition still makes sense, without change, in the
 supermanifold case.
\end{Remark}

\begin{Remark}
 To simplify the presentation, we implicitly choose a cleavage for
 the stack $\mathcal{G}$ (i.e., a preferred choice of pullback
 arrows) in order to define ``restrictions'' of objects in
 $\mathcal{G}$ such as $G_{X/S}|_{U/S}$. Explicitly, this means a
 triangle $G_{X/S} \leftarrow G_{U/S} \rightarrow G_{Y/T}$ consisting
 of an object $G_{U/S}$ together with an arrow $G_{U/S} \to G_{X/S}$
 covering the inclusion $U/S \hookrightarrow X/S$ and an arrow
 $G_{U/S} \to G_{Y/T}$ covering the map $(f, F)\colon U/S \to Y/T$.
 These data are unique up to unique isomorphism.
\end{Remark}

Morphisms are composed as follows. Suppose that, in addition to the
map $(f, F, \varphi)$ described above, we are given a second map
$(f^\prime, {F}^\prime, {\varphi}^\prime)$ starting at $Y/T$.
Choose a subset $V \subset U$ satis\-fying~(M1) and
such that $F(V) \subset U^\prime$ (where $U^\prime$ is the domain of $F^\prime$).
Then $\big(f^\prime \circ f, F^\prime \circ F|_V, \varphi^\prime \circ \varphi|_{V/S}\big)$
is a representative for the composition.
It is straightforward to show
that this operation on morphisms is independent of the choice of
representatives.

 Each of the stacks $\GBord_n$ is a symmetric monoidal
 stack with the tensor product given by fiberwise disjoint union,
 that is, $X/S \otimes Y/S := (X \amalg Y)/S$. A $\mathcal
 G$-structure on $X/S \otimes Y/S$ is obtained by the stack property,
 using the obvious cover $\{ X \to X \amalg Y,\ Y \to X \amalg Y \}$.
 Moreover, it~is clear that the pullback maps $\kappa^*$ are
 monoidal, so that $\GBord$ is a symmetric monoidal smooth category.

This concludes the construction of a category $\GBord_n$ with a
projection onto $\Man$, where an~object $\big(X/S; \rho_0, \dots, \rho_n;
G_{X/S}\big)$ is mapped to $S$.

 \begin{Remark}
 We think of an object in $\GBord_n(\mathrm{pt})$ as a sequence of
 $n$ composable bordisms, where the indiviual bordisms are the sets
 $X_{a-1}^a$, $a=1, \dots, n$, defined in~\eqref{SetsXab}. These
 are $d$-ma\-ni\-folds with boundary by requirement $(b)$, unless we are
 in the degenerate case where $\rho_{a-1}$ and $\rho_a$ have common
 zeros. Objects in $\GBord_n(S)$ are thought of as families of
 such bordisms, parametrized by $S$.
 \end{Remark}

\begin{Example}
 For $n = 0$, we just have one cut function $\rho_0$, and the core
 $X_0^0 = \rho_0^{-1}(0)$ is a~codimension $1$ submanifold of $X$, by
 condition (O3)$(b)$ above. This condition also implies that~$X_0^0$
 intersects the fibers of $X \to S$ transversally and hence
 determines a $(d-1)$-dimensional submanifold of each fiber $X_s$.
 Similarly, for $n = 1$, the core $X_0^1$ determines, for each $s \in
 S$ such that $\rho_0(s) < \rho_1(s)$ everywhere, a $d$-manifold
 $X_0^1 \cap X_s$ with boundary.
 It is compact,
 by condition (O3)$(c)$, thus a bordism between $X_0^0 \cap X_s$ and
 $X_1^1 \cap X_s$. In this way, $\GBord_0$ and $\GBord_1$ comprise
 families of bordisms and their boundaries. However, in the presence
 of a nontrivial geometry $\mathcal G$, there is the additional data
 of a neighborhood of the core together with a geometric structure on
 this neighborhood. This prohibits us to simply work directly with
 the cores.
\end{Example}

\begin{Lemma}
For each $n$, $\GBord_n$ is a stack.
\end{Lemma}

\begin{proof}
 We have to show that $\GBord_n$ is a prestack and that it satisfies
 descent.

 First, we show that $\GBord_n$ is a fibered category. To this end,
 let $f\colon S^\prime \to S$ be a morphism, and let $B= \big(X/S; \rho_0,
 \dots, \rho_n; G_{X/S}\big)$ be an object of $\GBord_n$ over $S$. Let
 $X^\prime = S^\prime \times_S X$ be the fiber product over $S$.
 This is an object of $\Fam$ over $S^\prime$, and we obtain a
 morphism $(f, F)\colon X^\prime/S^\prime \to X/S$ in $\Fam$. Set
 $\rho_a^\prime = (F^\prime)^* \rho_a$. Now $B^\prime =
 \big(X^\prime/S^\prime; \rho_0^\prime, \dots, \rho_n^\prime;
 F^*G_{X^\prime/S^\prime}\big)$ is an~object of $\GBord_n$ over $S^\prime$,
 and $(f, F, \varphi) \colon B^\prime \to B$ is a morphism in $\GBord_n$
 which is a cartesian lift of $f$.

 To show that $\GBord_n$ is a prestack, by~\cite[Proposition~3.22]{MR2223406}, it now suffices to show that for any
 manifold $S$, $\GBord_n(S)$ is a groupoid. Let
 \begin{equation*}
 (\mathrm{id}, F,
 \varphi)\colon\ \big(X/S; \rho_0, \dots, \rho_n; G_{X/S}\big) \to \big(Y/S;
 \rho_0^\prime, \dots, \rho_n^\prime; G_{Y/S}\big)
 \end{equation*}
 be a map covering the
 identity, where $F$ is defined on some open neighborhood $U$ of the
 core~$X_0^n$. Because $F|_U$ is a fiberwise diffeomorphism onto its
 image covering the identity, it is in fact a~diffeomorphism onto its
 image. Since $\mathcal{G}$ is a stack on $\Fam$ and $\varphi\colon
 G_{X/S}|_{U/S} \to G_{Y/S}|_{F(U)/S}$ covers the invertible morphism
 $(f, F|_U)$, it is invertible. Hence $\big(\mathrm{id}, F^{-1},
 \varphi^{-1}\big)$ is a morphism in~the direct opposite of
 $(\mathrm{id}, F, \varphi)$. It is clear that both compositions of
 these maps are restrictions of~the identity map, hence equivalent to
 the identity. This shows that each morphism in $\GBord_n(S)$ is
 invertible, hence it is a groupoid.

 To verify the descent property, let $S$ be a manifold and
 $\{S_i\}_{i \in I}$ a covering family of $S$. The objects of the category $\GBord_n(\{S_i \to
 S\})$ of descent data are tuples $\{B_i\}_{i \in I}$ with
 $B_i = \big(X_i/S_i; \rho_1^{i}, \dots, \rho_n^{i}; G_{X_i/S_i}\big)$ and
 isomorphisms $F_{ij}\colon B_i|_{S_{ij}} \to B_j|_{S_{ij}}$, $i, j
 \in I$, where $S_{ij} = S_i \times_S S_j$. Morphisms in this
 category are tuples $\{H_i\}_{i \in I}$ of morphisms $H_i\colon B_i \to B_i^\prime$ in
 $\GBord_n$ that are compatible with the morphisms $F_{ij}$ in the
 sense that $F_{ij} \circ H_i|_{S_{ij}} = H_j|_{S_{ij}} \circ
 F^\prime_{ij}$. The functor $\GBord_n(S) \to \GBord_n(\{S_i \to
 S\})$ is given by restriction, i.e., by pullback along the
 inclusions $S_{ij} \to S_i$, respectively $S_i \to S$. It is clear
 that this functor is fully faithful, as morphisms are locally
 determined and glue.

 To see that it is essentially surjective, let $\{B_i\}_{i \in I}$,
 $\{F_{ij}\}_{i, j \in I}$ be an object as above. The maps~$F_{ij}$
 are defined on some open neighborhood $U_{i}$ of the core of
 $X_i|_{S_{ij}}$ and come with morphisms $\varphi_{ij}\colon
 G_{X_i/S_i}|_{U_{i}/S_{ij}} \to G_{X_j/S_j}|_{U_{j}/S_{ij}}$ (where
 $U_j = F_{ij}(U_i)$). We need to show that the $B_i$ glue together
 to a bordism $B$ over $S$.

\begin{figure}

\centering

\begin{tikzpicture}
 \draw (0, 5) .. controls (2,4.5) and (3.6, 5.5).. (5,5) -- (5, 2);
 \draw[name path=X, dotted] (0, 3.8) .. controls (2, 3.1) and (3.6, 4.7).. (5, 4.2);
% \draw[name path=U, dotted] (0, 3.8) .. controls (1, 3.1) and (2.2, 3.9) .. (3, 4.2);
 \draw (0, 5) -- (0, 2) .. controls (1.3, 1.5) and (3.1, 2.7).. (5,2);
% \draw[name path=D, dotted] (0, 2.8) .. controls (1, 2.4) and (2, 3.0).. (3, 3.0);
% \draw[name path=B, dotted] (3, 3.0) .. controls (3.6, 3.0) and (4.5, 2.8) .. (5, 2.6);
 \draw[name path=Y, dotted] (0, 2.8) .. controls (2, 2.4) and (3.6, 3.3).. (5, 2.9);
 \draw[ultra thick] (0, 3.4) .. controls (1.3, 2.9) and (3.1, 3.9).. (5,3.5);
 \draw[name path=A, dashed] (3, 4.5) .. controls (3.6, 4.7) and (4.2, 4.9) .. (5, 4.5);
 \draw[name path=B, dashed] (3, 2.7) .. controls (3.6, 2.7) and (4.5, 2.5) .. (5, 2.3);
 \draw[dashed] (3, 4.5) -- (3, 2.7);
\tikzfillbetween[of=X and Y]{opacity=0.1};
\tikzfillbetween[of=A and B]{opacity=0.3};
\node at (4.5, 4) {$U$};

 \draw (0, 1.4) -- (5, 1.4);
 \node at (2.5, 1) {$S_1$};

 \draw (6, 5) .. controls (7, 5.2) and (7.5, 4.7) .. (8, 4.7) .. controls (9, 4.7) and (10, 5.5) .. (11,5);
 \draw[name path=Z, dotted] (6, 3.8) .. controls (7, 4) and (7.5, 3.5) .. (8, 3.5) .. controls (9, 3.5) and (10, 4.3) .. (11, 3.8);
 \draw[name path=W, dotted] (6, 3.2) .. controls (7, 3.4) and (7.5, 2.9) .. (8, 2.9) .. controls (9, 2.9) and (10, 3.7) .. (11, 3.2);
 \draw (11,5) -- (11, 2);
 \draw (6, 2) .. controls (7, 2.2) and (7.5, 1.7) .. (8, 1.7) .. controls (9, 1.7) and (10, 2.5) .. (11,2);

\tikzfillbetween[of=Z and W]{opacity=0.1};

 \draw (6, 5) -- (6,2);
 \draw[ultra thick] (6, 3.5) .. controls (7, 3.7) and (7.5, 3.2) .. (8, 3.2) .. controls (9, 3.2) and (10, 4) .. (11,3.5);
 \draw[name path=A, dashed] (6, 4.2) .. controls (6.6, 4.4) and (7.2, 4.6) .. (8, 4.3);
 \draw[name path=B, dashed] (6, 2.4) .. controls (6.6, 2.4) and (7.5, 2.2) .. (8, 2.3);
 \draw[dashed] (8, 4.3) -- (8, 2.3);
\tikzfillbetween[of=A and B]{opacity=0.3};
\node at (7.5, 4.0) {$F(U)$};

 \draw (6, 1.4) -- (11, 1.4);
 \node at (8.5, 1) {$S_2$};

 \draw[->] (4.5, 2.8) .. controls (5, 2.3) and (6, 2.3) .. (6.5, 2.8);
 \node at (5.5, 2.2) {$F$};
 \node at (0.5, 4.5) {$X_1$};
 \node at (10.5, 4.5) {$X_2$};

\end{tikzpicture}

\caption{Picture of an object in $\GBord_0(\{S_1, S_2 \to S\})$.
The light shaded regions depict the neighborhoods $(\rho^i_0)^{-1}(-\varepsilon, \varepsilon)$ of the core $(X_i)_0^0$ (drawn as a thick line).
When attempting to glue, the difficulty is to shrink $X_i$ to $\tilde{X}_i$ in such a way that the corresponding restriction $\tilde{F}$ of $F$ maps $\tilde{X}_1|_{S_{12}}$ diffeomorphically to $\tilde{X}_2|_{S_{12}}$.}
\end{figure}
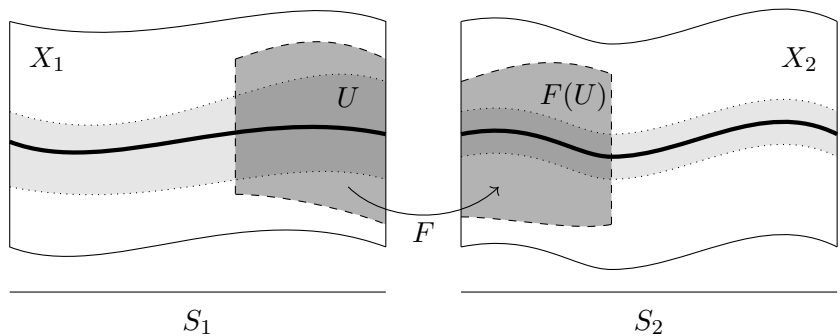

 To begin with, we assume that the given cover of $S$ consists of two
 elements $S_1$ and $S_2$, with $S_{12} = S_1 \times_S S_2$. Let
 $F\colon U/S_{12} \to F(U)/S_{12}$ be the gluing diffeomorphism.
 Let $\chi_1$, $\chi_2$ be a~smooth partition of unity, i.e.,
 non-negative functions with $\mathrm{supp}(\chi_i) \subset S_i$ and
 $\chi_1 + \chi_2 = 1$. We~denote the lifts of $\chi_i$ to ${X}_1$
 and ${X}_2$ via the projections $X_1 \to S_1$ and $X_2 \to S_2$ by
 the same letter. Now for $a= 0, 1, \dots, n$, consider the
 functions defined by
 \begin{gather*}
 \tilde{\rho}_a^1 = \chi_1 \cdot \rho_a^1 + \chi_2 \cdot F^* \rho_a^2,
 \\[.5ex]
 \tilde{\rho}_a^2 = \chi_1 \cdot \big(F^{-1}\big)^*\rho_a^1 + \chi_2 \cdot \rho_a^2.
 \end{gather*}
 Let $\varepsilon >0$ be as in assumption (M3) on $F$ and set
 \begin{equation*}
 \tilde{X}_i = \big(\tilde{\rho}_0^i\big)^{-1}(-\infty, \varepsilon) \cap \big(\tilde{\rho}_n^i\big)^{-1} (-\varepsilon, \infty) \subseteq X_i, \qquad i = 1, 2.
 \end{equation*}
 Let $\tilde{F}$ be the restriction of $F$ to $U_1 \cap
 \tilde{X}_1$. Observe that by construction, $U_1 \cap \tilde{X}_1 =
 \tilde{X}_1|_{S_{12}}$, and that $\tilde{F}$ maps
 $\tilde{X}_i|_{S_{12}}$ diffeomorphically onto
 $\tilde{X}_j|_{S_{12}} = F(U) \cap X_2$. Hence, $\tilde{X}_1$ and
 $\tilde{X_2}$ glue together, via $\tilde{F}_{12}$, to a family of
 manifolds over $S$. Since by construction, for each $a$, the
 functions $\rho_a^1$ and~$\rho_a^2$ coincide over $S_{12}$, they
 combine to give functions $\rho_0, \dots, \rho_n$ on $X$. Let
 $G_{\tilde{X}_i/ S_i} = G_{X_i/S_i}|_{\tilde{X}_i/S_i}$ be the
 restriction of the objects of $\mathcal{G}$ contained in the data of
 $B_i$ and let~$\tilde{\varphi}$ be the corresponding restriction of
 the gluing isomorphism $\varphi$ in $\mathcal{G}$ convering $F$
 (which then covers~$\tilde{F}$). Since $\mathcal{G}$ satisfies
 descent, these objects $G_{\tilde{X}_1/ S_1}$, $G_{\tilde{X}_2/
 S_2}$ glue together to an object~$G_{X/S}$ over~$X/S$. In total,
 we~obtain an object $B = \big(X/S; \rho_0, \dots, \rho_n; G_{X/S}\big)$ over
 $S$, together with maps $B \to B_i$ compatible with the gluing
 isomorphisms. In other words, we have found a~preimage to the given
 object in~$\GBord_n(\{S_i \to S\})$.

 Next, assume that the cover $\{S_i\}_{i \in I}$ is locally finite.
 Then $I$ must be countable, and we can identify $I = \N$.
 By replacing the functions $\rho_a^i$ by $\lambda_i \cdot \rho_a^i$
 for suitable constants $\lambda_i >0$, we may achieve that the value
 $\varepsilon >0$ in the condition (M3) on the cut functions $F_{ij}$
 can be chosen as $\varepsilon = 1$ for all $i$ and $j$. This
 replaces $B_i$ by isomorphic objects in $\GBord_n$ and hence
 does not change the object in the category of descent data. Observe
 here that since the cover is locally finite, there are only finitely
 many non-trivial transition functions $F_{ij}$ for each fixed $i$,
 and hence only finitely many constraints for each $\lambda_i$.
 Hence the modification is indeed possible. Now, using the previous
 case, we may one by one glue together the corresponding bordisms
 $B_1, \dots, B_k$ to a bordism $B^{(k)}$. Since the cover is locally
 finite, this may be achieved in such a way that over each compact
 subset of $S$, the bordisms $B^{(k)}$ are all canonically equivalent
 for $k$ large, hence they all glue together to a bordism $B$ over
 $S$.

 For a general cover $\{S_i\}_{i \in I}$, we choose a locally finite
 subcover, which leads to an equivalent category of descent data.
\end{proof}

\begin{Remark}
 The modification $\rho_a^i \mapsto \lambda_i \rho_a^i$ in the proof
 above was made to deal with the following technical problem: Without
 this change, the construction of $B^{(k+1)}$ might require a~smaller
 choice of $\varepsilon >0$ than the one needed in the construction
 of $B^{(k)}$. Therefore, the total space $X^{(k+1)}$ of $B^{(k+1)}$
 may be strictly smaller than that of $B^{(k)}$ and might degenerate in
 the limit $k \to \infty$.
\end{Remark}

The stacks $\GBord_n$ form a simplicial object in an obvious way:
Given an order-preserving map $\kappa\colon [n] \rightarrow [m]$, we
obtain a functor $\kappa^*\colon \GBord_m \rightarrow \GBord_n$,
given by removing or duplicating the functions $\rho_a$. Forgetting a
cut function has the interpretation of gluing bordisms along the
boundary determined by it. Regarding the morphisms, we would like to
say that they remain the same under $\kappa^*$. Precisely, what
happens is the following. Notice that a map in $\GBord_m$ is also a
map in $\GBord_n$. Hence, if a morphism in $\GBord_m$ is
represented by a map $F$, we let~$\kappa^* [F]$ be again the morphism
represented by $F$, but with respect to the equivalence relation in
$\GBord_n$. This makes sense because any two maps which are
equivalent in $\GBord_m$ are also equivalent in $\GBord_n$, as
follows directly from the definition.

\begin{Theorem}
$\GBord$ is a symmetric monoidal smooth category.
\end{Theorem}

\begin{proof}
To verify Definition~\ref{def:sm-cat}(1), we have to show that the maps
\begin{equation*}
\GBord_n(S) \longrightarrow \GBord_1(S) \times_{\GBord_0(S)} \cdots \times_{\GBord_0(S)} \GBord_1(S)
\end{equation*}
are equivalences for each manifold $S$ and each $n \in \N$. This is
clear, since gluing bordisms over a fixed parameter space is
straightforward.

To verify Definition~\ref{def:sm-cat}(2),
we consider the categories $h_0\GBord(S)$ determined by the
Segal set $n \mapsto h_0 \GBord_n(S)$, for each manifold $S$.
Objects $B$ of the stack $\GBord_1(S)$ then determine morphisms $[B]$
in $h_0 \GBord(S)$, and we have to analyze which objects give
rise to invertible morphisms this way.
We claim that for such a bordism $B = (X/S; \rho_0, \rho_1; G_{X/S})$,
the morphism~$[B]$ is invertible if and only if $B$ is ``thin'', meaning that
$\rho_0$ and $\rho_1$ have the same zero set.
Given this claim, the completeness condition Definition~\ref{def:sm-cat}(2) follows,
since $B$ is isomorphic to $B^\prime = (X/S; \rho_0, \rho_0; G_{X/S})$,
which is in the image of the degeneracy map $\GBord_0(S) \to \GBord_1(S)$.
To show the claim, assume that $[B]$ is invertible.
Let $M = (X/S; \rho_0; G_{X/S})$ be the left boundary of $B$, an object of $\GBord_0(S)$.
Then there exists a bordism $B^{-1}$ such that $[B] \circ \big[B^{-1}\big] = \id_{[M]}$.
The composition $[B] \circ \big[B^{-1}\big]$ can be represented by an element $B_2 \in \GBord_2(S)$
with the property that $s_1^* B_2 = B$ and $s_2^* B_2 = B^{-1}$,
where $s_i\colon [1] \rightarrow [n]$, $i=1, 2$,
is the morphism sending $0 \mapsto i-1$ and $1 \mapsto i$.
On the other hand, since $[B] \circ \big[B^{-1}\big] = \id_{[M]}$, the object
$c^* B_2 = (X/S; \rho_0, \rho_2; G_{X/S})$ must be isomorphic,
as an object of $\GBord_1(S)$,
to the thin bordism $\mathrm{id}_M = (X/S; \rho_0, \rho_0; G_{X/S})$.
Here $c\colon [1] \to [2]$ maps $0 \to 0$ and $1 \to 2$.
Therefore, $c^*B_2$ must be thin, i.e., $\rho_0$ and $\rho_2$ have the same zero sets.
This implies that also $\rho_1$ must have the same zero set, which implies that both~$B$ and $B^{-1}$ are thin as well.
Conversely, if $B$ is thin, then, as mentioned above, it is isomorphic to $B^\prime = (X/S; \rho_0, \rho_0; G_{X/S})$, which implies that $[B] = \id_{[M]}$. In particular, $[B]$ is invertible.
\end{proof}

We conclude this section by giving some examples of geometric bordism categories.

\begin{Example}[``no geometry''] %\label{Ex:NoGeometry}
 If we choose $\mathcal{G}$ to be the trivial stack on $\Fam^d$
 (whose fibers are single points), we get the realization of the
 $d$-dimensional bordism category in the world of~smooth categories,
 to be denoted $d\Bord$. Objects of $d\Bord_n$ consist simply of a
 family~$X/S$ of~manifolds parametrized by $S$, together with cut
 functions $\rho_0, \dots, \rho_n$. In this case, morphisms from
 $X/S$ to $Y/T$ are just given by smooth fiber-preserving maps $F$
 that are defined on a~neighborhood $U$ of the core, are fiberwise
 open embeddings and send $X_a^b$ to $Y_a^b$; two such maps $F$,
 $F^\prime$ (defined on $U$, $U^\prime$) are identified if they
 coincide on a smaller neighborhood $V \subset U \cap U^\prime$ of
 $X_0^n$.
\end{Example}

\begin{Example}[bordisms over a manifold]
 \label{ExampleBordismsOverAManifold}
 If $\mathcal{G}$ is the $d$-dimensional geometry represented by a
 manifold $M$, as in Example~\ref{ExManifoldsOver}, we get the
 category $d\Bord(M)$ of bordisms over $M$. Clearly, this
 specializes to the previous example if one takes $M$ to be a point.
 This is our main example in the second part of this paper (where
 moreover $d=1$).
\end{Example}

\begin{Example}[orientations] \label{Ex:Orientations}
 If $\mathcal G$ is the $d$-dimensional geometry of fiberwise
 orientations, as in~Example~\ref{Ex:RigidGeometries2}, we denote the
 resulting smooth category by $d\Bord^{\mathrm{or}}$. In this case,
 the fibers $X_s$ of an object $X/S \in d\Bord^{\mathrm{or}}_n$ carry
 orientations, and the maps $F$ are required to be
 orientation-preserving when restricted to the fibers. That this is
 a condition rather than additional data for~$F$ reflects the fact
 that $\mathcal{G}$ is a sheaf on $\Fam^d$, i.e., discrete as a
 stack.
\end{Example}

\subsection{Geometric field theories}
\label{sec:geom-field-theor}

Conceptually, a field theory should be a symmetric monoidal functor
from a suitable bordism category to the category of vector spaces. To
put this concept into our setup, we need our source and target
categories, as well as the functor, to be smooth. As source we take
$\GBord$ for some geometry $\G$, as defined above. To specify the
target, we need to fix a notion of ``smooth family'' of~vector spaces.
Initially, we will study field theories taking values in the smooth
category $\Vect$ of~finite-dimensional vector bundles on $\Man$,
obtained by applying the procedure of
Example~\ref{Ex:SmoothCatFromStacks} to the stack (of categories) of
finite-dimensional vector bundles. Later, in Section~\ref{sec:values-in-sheaves}, we will consider the more general case of
$C^\infty$-modules.
These categories are symmetric monoidal using the tensor product of
vector spaces, respectively modules.

\begin{Definition}[geometric field theories]
 \label{def:gft}
 Let $\G$ be a $d$-dimensional geometry. A ($d$-dimen\-sio\-nal) {\em
 field theory with geometry $\G$} is a symmetric monoidal smooth
 functor
 \begin{equation*}
 Z\colon\ \GBord \longrightarrow \Vect.
 \end{equation*}
 A morphism of field theories is a smooth, symmetric monoidal natural
 transformation.
\end{Definition}

Denoting by $\mathrm{Fun}^{\otimes}$ the groupoid of functors between
smooth categories, together with smooth, invertible natural transformations,
we denote by
\begin{equation*}
 \G\mathrm{FT} := \mathrm{Fun}^{\otimes}(\GBord, \Vect).
\end{equation*}
the groupoid of functorial field theories for a given geometry $\mathcal{G}$.

Field theories as functors between smooth categories are complicated
objects, due to the fact that stacks form a $2$-category. A field
theory $Z$ consists of the following data. First, for every object
$[n] \in \Delta$, there is a map of stacks
\begin{equation*}
Z_n\colon\ \GBord_n \longrightarrow \Vect_n.
\end{equation*}
However, since stacks form a 2-category, we cannot expect that these
strictly commute with the structure maps in $\Delta$; instead, for
each morphism $\kappa\colon [m] \rightarrow [n]$ in $\Delta$, there is
a 2-morphism (i.e., a natural isomorphism) $\zeta^\kappa\colon Z_m
\kappa^*\rightarrow \kappa^*Z_n$. If $\eta\colon [k] \rightarrow [m]$
is another map in $\Delta$, then the corresponding 2-morphisms have to
satisfy the coherence condition
\begin{equation}
 \label{eq:coh-sm-fun}
 \zeta^{\kappa\circ \eta}
 = \eta^* \zeta^\kappa \circ \zeta^\eta \kappa^*.
\end{equation}
Visually, this is depicted by
\begin{equation*}
 \begin{tikzcd}
 \GBord_n \ar[r, "\kappa^*"] \ar[d, "Z_n"'] &
 \GBord_m
 \ar[d, "Z_m"]
 \ar[dl, "\zeta^\kappa", Rightarrow, shorten <=3ex, shorten >=3ex]
 \ar[r, "\eta^*"] &
 \GBord_k \ar[d, "Z_k"]
 \ar[dl, "\zeta^\eta", Rightarrow, shorten <=3ex, shorten >=3ex]
 \\
 \Vect_n \ar[r, "\kappa^*"'] &
 \Vect_m \ar[r, "\eta^*"'] &
 \Vect_k
 \end{tikzcd}
 \cong
 \begin{tikzcd}
 \GBord_n \ar[r, "\eta^*\kappa^*"] \ar[d, "Z_n"'] &
 \GBord_m
 \ar[d, "Z_k"]
 \ar[dl, "\zeta^{\kappa\circ \eta}",
 Rightarrow, shorten <=3ex, shorten >=3ex]
 \\
 \Vect_n \ar[r, "\eta^*\kappa^*"'] &
 \Vect_k.
 \end{tikzcd}
\end{equation*}
(Here, we are using the strictness of $\GBord$ and $\Vect$, that is,
the fact that $(\kappa\eta)^*$ and $\eta^*\kappa^*$ are \emph{equal};
otherwise the identification of the two diagrams would involve,
additionally, the coherence data $\eta^*\kappa^* \cong
(\kappa\eta)^*$.) Fortunately, in our setting, the data of a field
theory can be simplified considerably, as the following lemma shows.

\begin{Lemma}[strictification]
 \label{LemmaStrictification}
 Let $\mathcal B$ be a strict smooth category and let $\mathcal V$ be
 a smooth category obtained from the procedure of
 Example~$\ref{Ex:SmoothCatFromStacks}$. Then any smooth functor
 $Z\colon \mathcal B \rightarrow \mathcal V$ has a~canonical
 strictification, that is, there exists a canonically isomorphic
 functor $Z' = \{Z'_n,\, \zeta'^\kappa\}$ such that $Z'_m \kappa^* =
 \kappa^*Z'_n$ and $\zeta'^\kappa = \id$ for every morphism
 $\kappa\colon [m]\rightarrow [n]$ in $\Delta$.
\end{Lemma}

\begin{proof}
 Fix an integer $n \geq 0$ and $X \in \mathcal B_n$, and denote by
 $\kappa_i\colon [0] \to [n]$, $0\leq i\leq n$, the morphism in
 $\Delta$ mapping $0 \mapsto i$. Set
 \begin{equation*}
 V_i = Z_n(\kappa_i^* X) \qquad\text{and}\qquad
 W_i = \kappa_i^*(Z_n(X)),
 \end{equation*}
 and note that we have natural isomorphisms
 \begin{equation*}
 \zeta^i = \zeta^{\kappa_i}_X\colon\ V_i \to W_i.
 \end{equation*}
 By our assumption on $\mathcal V$, $Z_n(X) \in \mathcal V_n$ is a
 chain of morphisms
 \begin{equation*}
 Z_n(X) = \big( W_0 \xrightarrow{\alpha_1} W_1 \xrightarrow{\alpha_2} \cdots
 \xrightarrow{\alpha_n} W_n \big)
 \end{equation*}
 in the stack of which $\mathcal V$ is the nerve. Towards defining
 the functor $Z'_n\colon \mathcal B_n \to \mathcal V_n$, set
 $Z'_n(X)$ to be chain of morphisms $V_0 \to \cdots \to V_n$ such the
 diagram below commutes:
 \begin{equation*}
 \begin{tikzcd}
 V_0 \ar[r] \ar[d, "\zeta^0"]
 & V_1 \ar[r] \ar[d, "\zeta^1"]
 & \cdots\ar[r]
 & V_n \ar[d, "\zeta^n"]
 \\
 W_0 \ar[r, "\alpha_1"]
 & W_1 \ar[r, "\alpha_2"]
 & \cdots\ar[r, "\alpha_n"]
 & W_n.
 \end{tikzcd}
 \end{equation*}
 The above diagram also fixes the effect of $Z'_n$ on morphisms of
 $\mathcal B_n$, if we insist that the collection $\big(\zeta^0, \dots,
 \zeta^n\big)$ defines a natural transformation $\xi_n\colon Z'_n \to
 Z_n$.

 It remains to show that the collection $\{ Z'_n \}$ defines a strict
 smooth functor, that is, the diagram
 \begin{equation*}
 \begin{tikzcd}
 \mathcal B_n \ar[r, "\kappa^*"] \ar[d, "Z'_n"]
 & \mathcal B_m \ar[d, "Z'_m"]
 \\
 \mathcal V_n \ar[r, "\kappa^*"]
 & \mathcal V_m
 \end{tikzcd}
 \end{equation*}
 commutes strictly for every morphism $\kappa\colon [m] \to [n]$ in
 $\Delta$. Now, $\kappa^*Z'_n(X)$ is the chain of morphisms
 \begin{equation}
 \label{eq:1}
 V_{\kappa(0)} \to V_{\kappa(1)} \to \cdots \to V_{\kappa(m)},
 \end{equation}
 obtained from $Z'_n(X)$ by appropriate compositions or insertion of
 identities. On the other hand, $Z'_m(\kappa^*X)$ is a chain of~mor\-phisms of the form
 \begin{equation}
 \label{eq:2}
 Z_0(\kappa_0^* (\kappa^* X))
 \to Z_0(\kappa_1^* (\kappa^* X))
 \to \cdots
 \to Z_0(\kappa_m^* (\kappa^* X)).
 \end{equation}
 By strictness of $\mathcal B$ as a simplicial object, we have
 $\kappa_i^* \kappa^* X = (\kappa \circ \kappa_i)^* X =
 \kappa_{\kappa(i)}^* X$, so that the $i$th object in \eqref{eq:2} is
 $Z_0(\kappa_i^*(\kappa^* X)) = V_{\kappa(i)}$. It remains to see
 that the morphisms in the chains~\eqref{eq:1} and~\eqref{eq:2}
 are identical. Consider the commutative diagram below:
 \begin{equation*}
 \begin{tikzcd}[column sep=small]
 Z'_m(\kappa^*X) \ar[d, "\xi_m"]
 & =
 & Z_0(\kappa_0^* (\kappa^* X)) \ar[r] \ar[d, "\zeta^0_{\kappa^*X}"]
 & Z_0(\kappa_1^* (\kappa^* X)) \ar[r] \ar[d, "\zeta^1_{\kappa^*X}"]
 & \cdots \ar[r]
 & Z_0(\kappa_m^* (\kappa^* X)) \ar[d, "\zeta^m_{\kappa^*X}"]
 \\
 Z_m(\kappa^*X) \ar[d, "\zeta^\kappa_X"]
 & =
 & U_0 \ar[r] \ar[d, "\kappa_0^*\zeta^\kappa_X"]
 & U_1 \ar[r] \ar[d, "\kappa_1^*\zeta^\kappa_X"]
 & \cdots \ar[r]
 & U_m \ar[d, "\kappa_m^*\zeta^\kappa_X"]
 \\
 \kappa^* Z_n(X)
 & =
 & W_{\kappa(0)} \ar[r]
 & W_{\kappa(1)} \ar[r]
 & \cdots \ar[r]
 & W_{\kappa(m)}.
 \end{tikzcd}
 \end{equation*}
 We will be done if we show that the $i$th composite vertical map is
 equal to $\zeta^{\kappa(i)}_X\colon V_{\kappa(i)} \to
 W_{\kappa(i)}$, since in this case we can replace the top row by~\eqref{eq:1} and still have a commutative diagram. But this fact is
 simply the coherence condition~\eqref{eq:coh-sm-fun}, applied to the
 case $\eta = \kappa_i$.
\end{proof}

\begin{Remark}[field theories as strict functors]
% \label{RemarkOnStrictness}
 Since our bordism categories, as well as the smooth category of
 vector bundles, satisfy the assumptions of the above lemma, it
 follows that we make no mistake by defining field theories as strict
 symmetric monoidal functors $Z\colon \GBord \rightarrow \Vect$, and
 their morphisms as strict natural transformations. We will work in
 this context in the next section, which simplifies our life
 considerably.
\end{Remark}

\section{Classification of one-dimensional field theories}
\label{sec:classification}

In this section, we discuss the classification of one-dimensional
field theories over a manifold $M$. Let us briefly discuss the
classical case (with ordinary categories and $M = \pt$) in order to
see what to expect. The one-dimensional (ordinary) bordism category
$1\Bord$ is easy to describe. The objects, compact zero-dimensional
manifolds, are just finite collections of points. To understand the
morphisms, one needs the classification of compact, connected
one-dimensional manifolds with boundary; this classification is very
simple (say, using Morse theory). Apart from the circle, which is the
only closed example, we have two elbows (the one with two incoming
boundary components and zero outcoming boundary components, as well as
its dual) and the interval (with one incoming and one outgoing
boundary component).

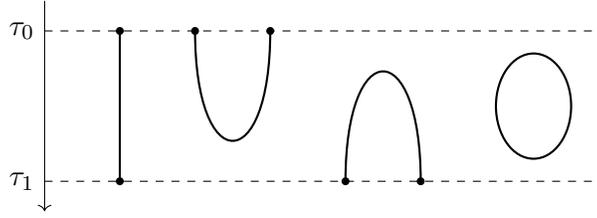
\begin{figure}[h]
 \centering
 \begin{tikzpicture}[x=1cm, y=2cm]
 \draw[bordism]
 (1, 1) pic{dot}
 -- (1, 0) pic{dot} ;
 \draw[bordism]
 (2, 1) pic{dot}
 to[out=-90, in=-90, looseness=5]
 (3, 1) pic{dot} ;
 \draw[bordism]
 (4, 0) pic{dot}
 to[out=90, in=90, looseness=5]
 (5, 0) pic{dot} ;
 \draw[bordism]
 (6.5, 0.5) circle[x radius = .5, y radius = .35];
 \draw[cut] (0, 1) -- (7.5, 1);
 \draw[cut] (0, 0) -- (7.5, 0);
 \draw[axis, ->] (0, 1.2) -- (0, -0.2) ;
 \draw \foreach \i in {0,1} {(0, 1- \i) node[left] {$\tau_\i$}};
 \end{tikzpicture}
 \caption{All possible connected unoriented one-dimensional bordisms.
 We call them interval, left elbow, right elbow, and circle,
 respectively. The cut functions are $\rho_i = t - \tau_i$, so
 these pictures are read from top to bottom.}
% \label{fig:un-1-bord}
\end{figure}

We briefly recall the well-known construction of one-dimensional field theories from vector spaces.

\begin{Construction}[unoriented $1\TFT$s]
 \label{ConstructionUnorientedTFTs}
 In the unoriented case, a field theory can be obtained from the data
 of a finite-dimensional vector space $V$ over $\mathbb{K} = \R$ or
 $\C$ together with a symmetric nondegenerate bilinear form $\beta$
 as follows:
 \begin{enumerate}[1.]\itemsep=0pt
 \item To a collection of $k$ points, we assign the $k$-fold tensor
 product $V^{\otimes k}$. In particular, the ground field $\mathbb{K}$
 corresponds to the empty set.
 \item To the interval, we assign the identity homomorphism on $V$.
 \item To the elbow with two incoming boundary components, we assign
 the bilinear form $\beta$.
 \item To the elbow with two outgoing boundary components, we assign
 $\tau := \sum_{i=1}^n \varepsilon_i b_i \otimes b_i \in V \otimes
 V$, where $b_1, \dots, b_n$ is a generalized orthonormal basis for
 $\beta$. This means that $\beta(b_i, b_j) = \varepsilon_i
 \delta_{ij}$, where $\varepsilon = \pm 1$ depending on the
 signature of $\beta$.
 \item To the circle, we assign the number $n = \dim(V)$.
 \end{enumerate}
\end{Construction}

 There are several things to check in order to see that this defines
 a field theory: For example, one has to check that $\beta \circ \tau
 = n$, as well as the \emph{snake identity}
 \begin{equation}
 \label{SnakeIdentity}
 (\beta \otimes \id) \circ (\id \otimes \tau) = \id.
 \end{equation}
 Conversely, any one-dimensional field theory $Z$ determines such a
 pair $(V, \beta)$: just set $V = Z(\mathrm{pt})$ and $\beta$ to be
 the value of the elbow with two incoming boundary components. Then
 it follows from the snake identity~\eqref{SnakeIdentity} that $V$
 must be finite-dimensional and $\beta$ must be nondegenerate.
 Moreover, the fact that $\beta$ must be symmetric follows from the
 observation that the elbows have an automorphism that switches the
 two boundary components.

The above construction can be upgraded to an equivalence of categories
\begin{equation*}
 1\TFT \cong \Vect^{\sim}_\beta,
\end{equation*}
where $\Vect^\sim_\beta$ is the groupoid of finite-dimensional vector
spaces equipped with a nondegenerate symmetric bilinear form, with
maps being isometries.

Things change if we equip our bordisms with non-discrete data. In the
following, we will consider the geometry where objects $X/S$ come
equipped with a smooth map $\gamma\colon X \rightarrow M$, where $M$
is some fixed target manifold, as in
Example~\ref{ExampleBordismsOverAManifold}. As a first approximation,
we can think of objects of the bordism category as points in $M$,
while morphisms are essentially paths in $M$. The corresponding smooth
category is denoted by $1\Bord(M)$, and the category of field theories
will be denoted by
\begin{equation*}
 1\TFT(M) := \mathrm{Fun}^\otimes\bigl(1\Bord(M), \Vect\bigr).
\end{equation*}
It turns out that this groupoid is equivalent to the groupoid
$\Vect^\sim_{\nabla, \beta}(M)$, the objects of which
 are finite-dimensional vector bundles over $M$ with a fiberwise
nondegenerate symmetric bilinear form and a compatible connection,
and the morphisms of which are connection-preserving isometries.

\begin{Theorem}[classification of $1$-TFTs]
 \label{Thm:Classification}
 There is an equivalence of categories
 \begin{equation*}
 1\TFT(M) \cong \Vect^{\sim}_{\nabla, \beta}(M),
 \end{equation*}
 which is natural in $M$.
\end{Theorem}

The remainder of this section is dedicated to the proof of this
result. First, in Section~\ref{SectionConstructionOfFieldTheories},
we explain how to construct elements of $1\TFT(M)$ from a vector
bundle with connection and bilinear form, in a functorial way, cf.\
Proposition~\ref{PropConstructionOfFieldTheories} below. While a
little tricky in the detail, this is more or less the standard
construction. The main work of the proof is done in
Section~\ref{SectionFunctorsOnPathM}, where we restrict our attention
to the path subcategory of the bordism category, where all issues already arise.
First, we restrict to the case that paths have sitting
instants near the marked points (see Definition~\ref{DefPathCat} for the precise definition),
which is rather standard. The main
new idea is then to reduce the general case to this one using
so-called \emph{modification functions}. The proof is then finished in
Section~\ref{SectionFinishProof}. Finally, in
Section~\ref{SectionTheOrientedCase}, we comment on the oriented case.

\medskip

In order to prove Theorem~\ref{Thm:Classification}, one needs a result
that reconstructs a connection from parallel transport data. To set
up the one we use, denote by $C^\infty([0, 1], M)$ the set of smooth
maps from $[0, 1]$ to $M$. It has a natural (infinite-dimensional)
smooth manifold structure modelled on a~nuclear Fréchet space, and
there are smooth evaluation maps
\begin{equation*}
 \mathrm{ev}_t \colon\quad C^\infty([0, 1], M) \rightarrow M,
 \qquad
 \gamma \mapsto \gamma(t)
\end{equation*}
for $t \in [0, 1]$. Thus, given a vector bundle $\V$ over $M$, we can
form the pullback bundles $\mathrm{ev}_t^*\V$. The tensor product
$\mathrm{ev}_0^* \V^\vee\otimes \mathrm{ev}_1^*\V$ is the vector
bundle over $C^\infty([0, 1], M)$ whose fiber at a path~$\gamma$ is
given by $\mathrm{Hom}\big(\V_{\gamma(0)}, \V_{\gamma(1)}\big)$. Finally,
given $0 \leq a \leq b \leq 1$, we let $\mathfrak{s}_{a, b}$ be the
smooth map from $C^\infty([0, 1], M)$ to itself defined by
\begin{equation*}
 (\mathfrak{s}_{a, b}\gamma)(t) = \gamma(a + (b-a)t).
\end{equation*}

\begin{Proposition}
 \label{PropMultiplicativity}
 Let $\V$ be a vector bundle over $X$ and let $P$ be a smooth section
 of the bundle $\mathrm{ev}_0^* \V^\vee \otimes \mathrm{ev}_1^* \V$
 over $C^\infty([0, 1], X)$. Assume that $P$ maps constant paths to
 the identity and that we have
 \begin{equation} \label{Multiplicativity}
 P(\mathfrak{s}_{a, 1} \gamma) \circ P(\mathfrak{s}_{0, a}\gamma)
 = P(\gamma),
 \end{equation}
 for all $\gamma \in C^\infty([0, 1], M)$, where $\mathfrak s_{a, b}$
 is the cutting-and-rescaling map defined above. Then there exists a
 unique connection $\nabla$ on $\V$ such that $P$ is the parallel transport
 along $\nabla$.
\end{Proposition}

Similar results were obtained by Freed~\cite[Proposition~B1]{MR1337109} and Schreiber and Waldorf~\cite[Lemma~4.1]{MR2520993}; cf.\ also~\cite[Lemma~4.9]{arXiv:1501.00967}.

\begin{proof}
 For $v \in T_pM$, let $\gamma \in C^\infty([0, 1], M)$ be a path
 such that $\gamma(t) = p$ and $\dot{\gamma}(t) = v$ for some $t \in
 (0, 1]$. For any section $u$ of $\V$, set
 \begin{equation*}
 \nabla_v u(p)
 := -\frac{\rm d}{{\rm d} \varepsilon}\bigg|_{\varepsilon=0}
 P(\mathfrak{s}_{t-\varepsilon, t} \gamma)
 u\bigl(\gamma(t-\varepsilon)\bigr),
 \end{equation*}
 noting that $P(\mathfrak{s}_{t-\varepsilon, 1} \gamma)
 u(\gamma(t-\varepsilon)) \in \V_p$ for each $\varepsilon$, hence
 differentiation makes sense. We proceed to show that this
 definition is independent of the choice of $\gamma$ and $t$ and
 defines a connection on~$\V$. In fact, in a local trivialization of
 the bundle $\V$, we have
 \begin{align*}
 (\nabla_v u)^i(p)
 & = -\frac{\rm d}{{\rm d} \varepsilon}\bigg|_{\varepsilon=0}
 P^i_j(\mathfrak{s}_{t-\varepsilon, t} \gamma)
 u^j\bigl(\gamma(t-\varepsilon)\bigr)
 \\
 & = - P^i_j(\mathfrak{s}_{t, t} \gamma)
\frac{\rm d}{{\rm d} \varepsilon}\bigg|_{\varepsilon=0}
 u^j\bigl(\gamma(t-\varepsilon)\bigr)
 - \bigg(\frac{\rm d}{{\rm d} \varepsilon}\bigg|_{\varepsilon=0}
 P^i_j(\mathfrak{s}_{1-\varepsilon, 1} \gamma)\bigg) u^j(p)
 \\
 & = (\partial_v u^i)(p)
 - u^j(p) {\rm d P}^i_j(\mathfrak{s}_{t, t} \gamma)
\frac{\rm d}{{\rm d} \varepsilon}\bigg|_{\varepsilon=0}
 \mathfrak{s}_{t-\varepsilon, t}\gamma,
 \end{align*}
 where we used that $(\mathfrak{s}_{t, t} \gamma)(t) \equiv p$ and
 $P^i_j(\mathfrak{s}_{t, t} \gamma) = \delta^i_j$, since $P$ maps
 constant paths to the identity. Now notice the vector field
 $\frac{\rm d}{{\rm d} \varepsilon}|_{\varepsilon=0} \mathfrak{s}_{t -
 \varepsilon, t}\gamma$ along the constant path $\mathfrak{s}_{t,
 t}\gamma$ is given by
 \begin{equation*}
 \bigg(\frac{\rm d}{{\rm d} \varepsilon}\bigg|_{\varepsilon=0}
 \mathfrak{s}_{t-\varepsilon, t}\gamma\bigg)(s)
 = \frac{\rm d}{{\rm d} \varepsilon}\bigg|_{\varepsilon=0}
 (\mathfrak{s}_{t-\varepsilon, t}\gamma)(s)
 = \frac{\rm d}{{\rm d} \varepsilon}\bigg|_{\varepsilon=0}
 \gamma(t-\varepsilon + \varepsilon s) = s\dot{\gamma}(t) = s v.
 \end{equation*}
 Hence $\nabla_v u(p)$ is independent of the choice of $\gamma$.
 Define
 \begin{equation*}
 \omega^i_j(p) v := -{\rm d} P^i_j(p) V,
 \end{equation*}
 where $p$ denotes the path constant equal to $p$ and $V$ denotes the $T_p M$-valued function $V(t) = tv$. Observe that $\omega^i_j$ defines a matrix of one-forms on $T_p M$. Then
 \begin{equation*}
 (\nabla_v u)^i(p) = \partial_v u^i(p) + u^j(p)\,\omega^i_j(p) v,
 \end{equation*}
 hence $\nabla_v$ is a connection with Christoffel symbols
 $\omega^i_j$. Finally, fix $\gamma \in C^\infty([0, 1], X)$ and
 $u_0 \in \gamma(0)$ and let $u(t) := P(\mathfrak{s}_{0, t}\gamma)
 u_0$. Then we have $u(0) = u_0$ and using $(\mathfrak{s}_{t -
 \varepsilon, t} \gamma)^{\cdot}(t) = t \dot{\gamma}(t)$, we obtain
 \begin{align*}
 \frac{\nabla}{{\rm d} t} u(t)
 & = -\frac{\rm d}{{\rm d}\varepsilon}\bigg|_{\varepsilon=0}
 P(\mathfrak{s}_{t-\varepsilon, t} \gamma) u(t-\varepsilon) \\
 & = -\frac{\rm d}{{\rm d}\varepsilon}\bigg|_{\varepsilon=0}
 P(\mathfrak{s}_{t-\varepsilon, t} \gamma)
 P(\mathfrak{s}_{0, t-\varepsilon}) u_0\\
 & = -\frac{\rm d}{{\rm d} \varepsilon}\bigg|_{\varepsilon=0}
 P(\mathfrak{s}_{0, t} \gamma) u_0 = 0.
 \end{align*}
 Hence $u(t)$ is the parallel transport of $u_0$ along $\gamma$.
\end{proof}

\subsection{Construction of field theories from vector bundles}
\label{SectionConstructionOfFieldTheories}

Let $M$ be some fixed target manifold. In this section, we construct
a field theory from the data of an object $(\V, \nabla, \beta) \in
\Vect_{\nabla, \beta}^\sim(M)$. More precisely, we prove the
following proposition.

\begin{Proposition}[construction of field theories]
 \label{PropConstructionOfFieldTheories}
 There is a functor
 \begin{equation*}
 \Phi \colon\ \Vect_{\nabla, \beta}^\sim (M) \longrightarrow 1\TFT(M),
 \end{equation*}
 which is fully faithful and natural in $M$.
\end{Proposition}

Breaking down our general definition to this special case, an object
of $1\Bord(M)_n$ lying over a manifold $S$ is given by a family $X/S$
of one-dimensional manifolds $X_s$, $s \in S$, together with a~map
$\gamma\colon X \rightarrow M$ and functions $\rho_0, \dots, \rho_n
\colon X \rightarrow \R$ which cut out codimension-one submanifolds
\begin{equation*}
 X_a^a = \{ x\in X \mid \rho_a(x) = 0\}.
\end{equation*}
The properness assumption (O3c) implies that the restrictions $X_a^a \cap
X_s$ to the fibers are compact, i.e., finite collections of points.
More generally, we have the following result.

\begin{Lemma}
 \label{Lemma:Sigma00IsCover}
 The submanifold $X_a^a$ of $X$ is a finite covering of $S$ $($with possibly empty fibers$)$.
\end{Lemma}

\begin{proof}
 Let $s \in S$ and $x \in X_a^a \cap X_s$, and denote by $\pi\colon X
 \rightarrow S$ the projection. Since ${\rm d} \rho_a|_{X_s}(x)\allowbreak \neq 0$,
 $X_a^a$ intersects the fiber $X_s$ transversally, that is, ${\rm d}\pi(x)$
 is an isomorphism when restricted to the tangent space $T_x X_a^a$.
 Therefore, $\pi|_{X_a^a}$ is a local diffeomorphism. Furthermore,
 by assumption, $X_a^a$ is proper over $S$, meaning that
 $\pi|_{X_a^a}$ is a proper map. However, a proper local
 diffeomorphism is a covering map the fibers of which have at most
 finitely many points (possibly zero).
\end{proof}

To prove Proposition~\ref{PropConstructionOfFieldTheories}, we start
by constructing a field theory from the data of a vector bundle with
non-degenerate bilinear form and compatible connection.

Let $(\V, \nabla, \beta) \in \Vect_{\nabla, \beta}^\sim(M)$. Our goal
is to construct a field theory $Z_{\V, \nabla, \beta}$, which will be
the value of $(\V, \nabla, \beta)$ under the functor $\Phi$ in
Proposition~\ref{PropConstructionOfFieldTheories}. Recall that the
smooth functor $Z_{\V, \nabla, \beta}$, as a morphism of simplicial
objects, will consist of a sequence of stack maps $1\Bord(X)_n \to
\Vect_n$, $n \in \Delta$. At the $n$th simplicial level, we must have
\begin{equation*}
 Z_{\V, \nabla, \beta}(X/S; \rho_0, \dots, \rho_n; \gamma)
 = (\W_0, \dots, \W_n; {f}_1, \dots, {f}_n)
\end{equation*}
for some vector bundles $\W_a$ over $S$ and vector bundle maps
${f}_a\colon \W_{a-1} \rightarrow \W_a$. To define $\W_a$, for each
$a = 0, \dots, n$, set first $\tilde{\W}_a := (\gamma|_{X_a^a})^*\V$,
which makes sense since, for each $a=0, \dots, n$, $X_a^a$ is a
codimension-one submanifold of $X$. By
Lemma~\ref{Lemma:Sigma00IsCover}, $X_a^a$ is a finite covering of $S$,
hence any small enough open $U \subset S$ is covered by $U_1, \dots,
U_k \subset X_a^a$ such that the projection map $\pi$ provides
diffeomorphisms $\pi|_{U_j}\colon U_j \rightarrow U$. Hence we can
set
\begin{equation}
 \label{DefinitionOfWOverU}
 \W_a|_U :=
 \big(\pi|_{U_1}^{-1}\big)^* \tilde{\W}_a
 \otimes \cdots \otimes
 \big(\pi|_{U_k}^{-1}\big)^* \tilde{\W}_a,
\end{equation}
where $\W_a|_U = \C$ if $k=0$.
These vector bundles glue together to a vector bundle $\W_a$ over $S$.

To define ${f}_a$ for each $a = 1, \dots, n$, consider the subsets
$X_{a-1}^a$. Let $Y^1, \dots, Y^k$ be the connected components of
$X_{a-1}^a|_U$, where $U \subset S$ is a small connected open
as above for both $a-1$ and $a$. The map
${f}_a\vert_U$ will be the tensor product of maps ${f}^j_a$, where
${f}^j_a$ is determined by the connected component $Y^j$. Each ${f}^j$
will be a vector bundle map
\begin{equation}
 \label{RequiredAlphaJ}
 {f}^j_a \colon\
 \bigotimes_{Z} \big(\pi|_{Z}^{-1}\big)^* \tilde{\W}_{a-1}
 \longrightarrow
 \bigotimes_{Z^\prime} \big(\pi|_{Z^\prime}^{-1}\big)^* \tilde{\W}_{a-1},
\end{equation}
where $Z$ runs over the connected components of $Y^j \cap
X_{a-1}^{a-1}$ and $Z^\prime$ runs over the connected components of
$Y^j \cap X_a^a$ (by possibly making $U$ smaller, we can assume that
the projection map is a diffeomorphism to $U$ when restricted to any
one of these sets $Z$ and $Z^\prime$). The tensor product ${f}_a|_U
:= {f}^1_a \otimes \cdots \otimes {f}^k_a$ is then indeed a vector bundle
map $\W_{a-1}|_U\rightarrow \W_a|_U$, by definition~\eqref{DefinitionOfWOverU}.

Now, each $Y^j$ is, essentially, either a circle bundle or an interval
bundle over $U$; this only fails to be the case if $\rho_{a-1}(x) =
\rho_a(x)$ at some points $x \in X_{a-1}^a|_U$. To address this
issue, we let
\begin{equation*}
 U^j_\circ := \bigl\{
 s \in U \mid \rho_{a-1}(y) \neq \rho_a(y)
 \text{ for each } y \in Y^j \cap X_s
 \bigr\},
\end{equation*}
which is an open subset of $U$. The complement $U \setminus
U^j_\circ$ is the set where $Y^j$ is a ``thin bordism'', in the sense
that
\begin{equation*}
 Y^j|_s = \big(X_{a-1}^{a-1} \cap Y^j\big)\big|_s =\big (X_{a}^{a} \cap Y^j\big)\big|_s
 \qquad\text{for}\quad
 s \in U \setminus U^j_\circ.
\end{equation*}
Note in particular that for $s \in U \setminus
U^j_\circ$, $Y^j|_s$ consists of finitely many points.

\begin{Lemma}
 For each $j=1, \dots, k$, write $Y^j_\circ := Y_j|_{U^j_\circ}$.
 Then $Y^j_\circ \to U^j_\circ$ is a fiber bundle whose fibers are
 compact one-dimensional manifolds with boundary $($thus, either
 intervals or circles$)$. Moreover, if $Y^j_\circ$ is a circle
 bundle, then $U^j_\circ = U$ and $Y^j_\circ = Y^j$.
\end{Lemma}

\begin{proof}
By construction, the total space $Y^j_\circ$ is a compact manifold and the
projection $\pi|_{Y_\circ^j}$: $Y^j_\circ \to U_\circ^j$ is a proper
submersion, hence a fiber bundle (with possibly empty fibers).
The last statement follows from the fact that circle bundles cannot degenerate,
by the requirements on the functions $\rho_a$.
\end{proof}

By possibly shrinking $U$ further, we may assume moreover that all
these bundles are trivial. We now define ${f}^j$ case by case.

Suppose first that $Y^j_\circ$ is a circle bundle, so that $Y^j_\circ
= Y^j$ and $U^j_\circ = U$. In this case, $Y^j \cap X_{a-1}^{a-1} =
Y^j \cap X_{a}^{a} = \varnothing$, hence we have to produce a vector
bundle map from the trivial line bundle to itself, that is, a function
on $U$. Choose a trivialization $\varphi\colon U \times S^1
\rightarrow Y^j$. Now set
\begin{equation*}
 {f}^j := \tr P(\varphi),
\end{equation*}
where $P(\varphi)$ is the parallel transport around the loops
$\varphi_s\colon S^1 \rightarrow Y^j \subset X$ given by $\varphi_s(t)
= \varphi(s, t)$ with respect to the pullback connection of
$\gamma^*\V \rightarrow X$. We claim that ${f}^j$ is
independent of the choice of $\varphi$. If $\tilde{\varphi}$ is
another trivialization of $Y^j$ that induces the same orientation on
the fibers and agrees with $\varphi$ at the basepoint $1 \in S^1$,
then it is just a reparametrization of $\varphi$, and our claim
follows from the invariance of parallel transport under
reparametrizations. Without the assumption on basepoints,
$P(\varphi_s)$ and $P(\tilde\varphi_s)$ are conjugates for each $s \in
U$, so the trace ${f}^j$ is still independent of the choice of
$\varphi$. Finally, if $\tilde\varphi$ induces the opposite
orientation, then $P(\tilde\varphi_s) = P(\varphi_s)^{-1}$. This
yields the same trace, since $P(\varphi)$ preserves the bilinear form $\beta$;
the calculation is
\begin{align*}
 \tr P(\varphi_s)
= \!\sum_{i=1}^n \varepsilon_i \beta\bigl(P(\varphi_s) b_i, b_i\bigr)
 = \!\sum_{i=1}^n \varepsilon_i \beta\bigl( b_i, P(\varphi_s) b_i\bigr)
=\! \sum_{i=1}^n \varepsilon_i \beta\bigl( P(\varphi_s)^{-1} b_i, b_i\bigr)
 = \tr P(\tilde\varphi_s),
\end{align*}
where $b_1, \dots, b_n$ is a generalized orthonormal basis for $\beta$
and we used the symmetry of $\beta$ (note that $\beta$ is \emph{not}
assumed to be Hermitian in the complex case).

Suppose now that $Y^j_\circ$ is a bundle of intervals. In this case,
we have the following lemma.

\begin{Lemma}
 There exists a smooth map $\varphi\colon U \times [0, 1] \rightarrow
 Y^j \subset X_{a-1}^a$ such that $(\pi \circ \varphi)(s, t) = s$ and
 such that $\varphi|_{U^j_\circ \times [0, 1]}$ is a trivialization
 of the interval bundle $Y^j_\circ$.
\end{Lemma}

\begin{proof}
 If $U^j_\circ = U$, the lemma is clear, because then $Y^j$ is an interval
 bundle over $U$, which must be necessarily trivial: it admits two
 nowhere agreeing sections, given by the zero sets of~$\rho_{a-1}$,
 respectively $\rho_a$. In general, over $U$, our bordism is
 isomorphic to a bordism of the form $((\R \times
 U)/U; \rho_0, \rho_1; \gamma)$, where $\rho_i(t, s) = t-\tau_i(s)$,
 $i=0, 1$, for smooth functions $\tau_0, \tau_1\colon U \rightarrow
 \R$ with $\tau_0 \leq \tau_1$.

 In that case, we have $U^j_\circ = \{s
 \in U \mid \tau_0(s) \neq \tau_1(s)\}$, and
 \begin{equation*}
 \varphi_s(t) = \tau_0(s) + \bigl(\tau_1(s)-\tau_0(s)\bigr)t
 \end{equation*}
 gives the desired parametrization.
\end{proof}

In particular, this means that the paths $\varphi_s\colon [0, 1]
\rightarrow X$ given by $\varphi_s(t) := \varphi(s, t)$ map to the
fibers $Y^j|_s$, and, for $s \in U \setminus U^j_\circ$, $\varphi_s$
is constant (since for such $s$, $Y^j_\circ$ is a collection of
finitely many points).

Let $P(\varphi)$ be the vector bundle isomorphism between the bundles
$(\gamma \circ \varphi\circ(\id \times i))^*\V$, $i=0, 1$, over $U$
given over $s \in U$ by parallel translation along the path
$\varphi_s$. Now notice that $\varphi \circ(\id \times i)$, $i=0, 1$
is a section of $\pi\colon X \rightarrow S$, with image contained in either
$X_{a-1}^{a-1}$ or $X_a^a$; hence
\begin{equation*}
 \bigl(\gamma \circ \varphi\circ(\id \times i)\bigr)^*\V
 = \bigl(\gamma \circ \big(\pi|_{C}^{-1}\big)\bigr)^*\V
 = \big(\pi|_{C}^{-1}\big)^* \tilde{\W}_b
\end{equation*}
for some connected component $C$ of $Y^j \cap X_{b}^{b}$; here either
$b=a-1$ or $a$. For all $s \in U$, since~$X_{a-1}^a|_s$ is a (possibly degenerate)
interval, $\big(V^j \cap X_{a-1}^{a-1}\big)|_s$ has either zero, one or two
elements. Correspondingly, $\big(V^j \cap X_{a}^{a}\big)|_s$
has two, one or zero elements. In either case, we can use the
bilinear form $\beta$ to turn $P(\varphi)$ into a morphism of the
required form~\eqref{RequiredAlphaJ}. This defines~${f}^j$ in this
case. As before, we use the parametrization independence as well as the
fact that parallel transport preserves $\beta$ in order to show that
this definition is independent of the choice of $\varphi$.

This defines $Z_{\V, \nabla, \beta}$ on objects. On morphisms in
$1\Bord(M)$, we declare that $Z_{\V, \nabla, \beta}$ acts by pullbacks
in the obvious way. This concludes the definition of $Z_{\V, \nabla,
 \beta}$. Of course, there are several things to check in order to
show that this is a field theory. However, all checks can be made
\emph{pointwise}, hence are very similar to the classical arguments
outlined above (cf.~Construction~\ref{ConstructionUnorientedTFTs}).

\begin{proof}[Proof of Proposition~\ref{PropConstructionOfFieldTheories}]
 Of course, we set $\Phi(\V, \nabla, \beta) = Z_{\V, \nabla, \beta}$
 for $(\V, \nabla, \beta) \in \Vect^{\sim}_{\nabla, \beta}(M)$, where
 $Z_{\V, \nabla, \beta}$ is the field theory constructed above. We
 now discuss how $\Phi$ acts on morphisms in~$\Vect_{\nabla,
 \beta}^\sim(M)$. To this end, let $(\V, \nabla, \beta)$ and
 $(\V^\prime, \nabla^\prime, \beta^\prime)$ be vector bundles on $M$
 with connection and a compatible bilinear form, and let
 $\alpha\colon \V \rightarrow \V^\prime$ be a vector bundle
 isomorphism preserving these additional structures. We now define
 the smooth natural transformation
 \begin{equation*}
 \eta^\alpha := \Phi(\alpha)
 \colon\ Z_{\V, \nabla, \beta}
 \longrightarrow Z_{\V^\prime, \nabla^\prime, \beta^\prime}.
 \end{equation*}
 First, we look at the simplicial level zero. If $X/S = (X/S;
 \rho_0; \gamma)$ is a \emph{single point}, meaning that~$X_0^0$ has
 connected fibers (in other words, $\pi\colon X_0^0 \rightarrow S$ is
 a diffeomorphism), we have
 \begin{equation*}
 Z_{\V, \nabla, \beta}(X/S)
 = \bigl(\gamma \circ \big(\pi|_{X_0^0}^{-1}\big)\bigr)^* \V, \qquad
 Z_{\V^\prime, \nabla^\prime, \beta^\prime}(X/S)
 = \bigl(\gamma \circ \big(\pi|_{X_0^0}^{-1}\big)\bigr)^* \V'.
 \end{equation*}
 Hence we can set
 \begin{equation*}
 \eta^\alpha_{X/S}
 := \bigl(\gamma \circ \big(\pi|_{X_0^0}^{-1}\big)\bigr)^* \alpha
 \end{equation*}
 in this case. Any object in $1\Bord(M)_0$ can, at least locally, be
 uniquely decomposed into a~union of single-point-objects just
 discussed, and hence the requirement that $\eta^\alpha$ is symmetric
 monoidal determines it on all of $1\Bord(M)_0$.

 By Lemma~\ref{LemmaOnNaturalTransformations}, any smooth natural
 transformation $\eta\colon Z_{\V, \nabla, \beta} \rightarrow
 Z_{\V^\prime, \nabla^\prime, \beta^\prime}$ is determined by its
 component $\eta_0$ at the simplicial level zero; however, it is not
 clear that $\eta^\alpha$ defined above on~simplicial level zero
 indeed extends to all higher simplicial levels. Here, again by
 Lemma~\ref{LemmaOnNaturalTransformations} it~suffices to consider
 the simplicial level one. To this end, let $X/S = (X/S; \rho_0,
 \rho_1;\gamma)$ be an~object of $1\Bord(M)_1$ and write $Z_{\V,
 \nabla, \beta}(X/S) = (\W_0, \W_1; f_0)$ and $Z_{\V, \nabla,
 \beta}(X/S) = (\W_0, \W_1; f_0)$. We~have to check that the
 diagram
 \begin{equation*}
 \begin{tikzcd}
 \W_1 \ar[d, "\eta^\alpha_1"']
 & \W_0 \ar[l, "f_0"'] \ar[d, "\eta^\alpha_0"]
 \\
 \W^\prime_1
 & \W_0^\prime \ar[l, "f_0^\prime"']
 \end{tikzcd}
 \end{equation*}
 commutes, where $\eta^\alpha_i = \eta^\alpha_{d_i^*X/S}$, with
 $d_i\colon [0]\rightarrow [1]$ the usual boundary maps. Since the
 morphisms~$f_0$ and $f^\prime_0$ are essentially given by parallel
 transport, respectively the bilinear form $\beta$, it~is now easy to
 check that this diagram commutes for all bordisms if and only if
 $\alpha$ intertwines the connections and bilinear forms on $\V$,
 respectively $\V^\prime$.

 Finally, we show that $\Phi$ is fully faithful. To this end, let
 $\point \in 1\Bord(M)_0$ be the \emph{universal point}, which is the
 object over $M$ given by
 \begin{equation}
 \label{UniversalPoint}
 \point = \bigl(
 (\R \times M)/M;
 \rho_0 = \pr_{\mathrm{\R}};
 \gamma = \pr_M
 \bigr).
 \end{equation}
 For any natural transformation $\eta \colon Z_{\V, \nabla, \beta}
 \rightarrow Z_{\V^\prime, \nabla^\prime, \beta^\prime}$, the
 component $\eta_{\point}$ is a vector bundle isomorphism $\V
 \rightarrow \V$. In particular, if $\alpha\colon \V \rightarrow \V$ is a
 vector bundle isomorphism preserving connections and bilinear forms,
 tracing through the above definitions shows that
 $\eta^\alpha_{\point} = \alpha$. Hence if $\eta^\alpha =
 \eta^{\alpha^\prime}$, this implies $\alpha = \alpha^\prime$; in
 other words, $\Phi$ is faithful. Conversely, it is easy to check
 that $\Phi\big(\eta_{\point}\big) = \eta$ for any natural transformation
 $\eta$, so that $\Phi$ is also full.
\end{proof}

\subsection{Functors from the path category}
\label{SectionFunctorsOnPathM}

Let $M$ be a fixed target manifold. In this section, we restrict our
attention to the smooth path category of $M$, a certain subcategory of
$1\Bord(M)$ which is somewhat easier to describe.

\begin{Definition}[smooth path category]
 \label{DefPathCat}
 Write $\Path(M)$ for the full smooth subcategory of $1\Bord(X)$
 consisting of those objects $(X/S; \rho_0, \dots, \rho_n; \gamma) \in
 1\Bord(M)_n$ such that each of the sets $X_a^b$, $0 \leq a \leq b \leq
 n$, defined in~\eqref{SetsXab} has connected fibers. Let moreover
 $\Path_c(M)$ be the full subcategory of $\Path(M)$ consisting of those
 objects $(X/S; \rho_0, \dots, \rho_n; \gamma)$ such that $\gamma$ is
 fiberwise constant in a neighborhood of $X_a^a$ for each $0 \leq a
 \leq n$.
\end{Definition}

Objects in $\Path(M)_n$ over $S \in \Man$ can be thought of as
$S$-families of paths in $M$ with $n + 1$ marked points, while the
full subcategory $\Path_c(M)_n$ consists of those paths that have
sitting instants at the marked points.

\begin{Notation}[standard objects] \label{Notation:StandardObjects}
 We denote by
 \begin{equation*}
 (\gamma; \tau_0, \dots, \tau_n) := \bigl((\R\times S)/S; \rho_0,
 \dots, \rho_n; \gamma\bigr) \in \Path(M)_n
 \end{equation*}
 the object where $\R \times S \rightarrow S$ is the projection onto
 the second factor, $\gamma\colon \R \times S \rightarrow M$ is a smooth
 map, and the cut functions $\rho_0, \dots, \rho_n$ are given by
 $\rho_j(x, s) = x - \tau_j(s)$ for smooth functions $\tau_j\colon S
 \rightarrow \R$ satisfying $\tau_0 \leq \dots \leq \tau_n$.
\end{Notation}

\begin{Remark}
 Going through the definition of the morphisms in $1\Bord(M)$ shows
 that morphisms between standard objects over $S=\mathrm{pt}$ are
 (equivalence classes of) diffeomorphisms $F$ of~$\R$, which must be
 orientation preserving, as they need to preserve the sign of the cut
 functions~$\rho_i$.
\end{Remark}

\begin{Remark}
 \label{RemarkStackification}
 Denote by $\Path(M)^\circ_n$ the full subcategory of the fibered
 category $\Path(M)_n$ consisting of all objects $(\gamma; \tau_0,
 \dots, \tau_n)$ and the morphisms between them. This forms a prestack
 over $\Man$; since any object in $\Path(M)_n$ is locally isomorphic to
 $(\gamma; \tau_0, \dots, \tau_n)$ for suitable $\gamma$ and~$\tau_a$,
 the stack $\Path(M)_n$ is a stackification of $\Path(M)^\circ_n$. In
 particular, this means that any map from $\Path(M)_n$ to a smooth stack
 $\mathcal V$ is determined on $\Path(M)^\circ_n$ (up to unique
 isomorphism), and, conversely, any map from $\Path(M)^\circ_n$ to
 $\mathcal V$ extends, uniquely up to unique isomorphism, to~$\Path(M)_n$.

 Similar remarks hold for the subcategory $\Path_c(M)_n^\circ \subseteq
 \Path_c(M)_n$; a path $(\gamma; \tau_0, \dots, \tau_n)$ is contained in
 $\Path_c(M)_n$ if for any $s \in S$ and each $j=0, \dots, n$, $t \mapsto
 \gamma(t, s)$ is constant near $t = \tau_j(s)$.
\end{Remark}

\begin{Construction}
 \label{ConstructionZVNabla}
 Given a vector bundle $\V$ with connection $\nabla$, we can define a
 functor
 \begin{equation*}
 Z_{\V, \nabla} \colon\ \Path(M) \longrightarrow \Vect
 \end{equation*}
 as follows. For an $S$-family $(\gamma; \tau_0, \dots, \tau_n)$ of
 paths in $M$, we set
 \begin{equation*}
 Z(\gamma; \tau_0, \dots, \tau_n)
 = (\W_0, \dots, \W_n, P_1, \dots, P_n),
 \end{equation*}
 where $\W_a := (\gamma\circ(\tau_a \times \id))^*\V$ is a vector
 bundle over $S$ and, for each $s\in S$, $P_j(s)$ is the parallel
 transport via $\nabla$ along the path $t \mapsto \gamma(t, s)$, $t \in
 [\tau_{j-1}, \tau_j]$. To a morphism between two standard objects
 $(\gamma; \tau_0, \dots, \tau_n)$ and $(\gamma^\prime;\tau_0^\prime,
 \dots, \tau_n^\prime)$ we assign the identity; this is well-defined
 because automorphisms in $\Path(M)$ are reparametrizations of paths
 that fix the marked points, and parallel transport is invariant under
 reparametrizations. This defines the functor on the subcategory
 $\Path^\circ(M) \subseteq \Path(M)$ of standard objects; by
 Remark~\ref{RemarkStackification}, we get a functor on all of
 $\Path(M)$, unique up to unique isomorphism.
\end{Construction}

Now let $(\V, \nabla)$ and $(\V^\prime, \nabla')$ be two vector
bundles with connection. Clearly, any vector bundle isomorphism
$\alpha$ defines a natural transformation
\begin{equation*}
 \begin{tikzcd}[column sep = large]
 \Path(M)_0
 \ar[r, bend left, "(Z_{\V, \nabla})_0", ""' name = U]
 \ar[r, bend right, "(Z_{\V^\prime, \nabla^\prime})_0"', "" name = D]
 \ar[from=U, to=D, Rightarrow, "\eta_0^\alpha"]
 &
 \Vect_0
 \end{tikzcd}
\end{equation*}
at simplicial level zero. As in the proof of Proposition~\ref{PropConstructionOfFieldTheories}, it follows that the condition
for $\eta_0^\alpha$ to extend to higher simplicial levels is precisely
the condition that $\alpha$ preserves connections.

Hence if
$\alpha\colon \V \rightarrow \V^\prime$ is a connection-preserving
isomorphism of vector bundles, we get a~natu\-ral transformation
$\eta^\alpha\colon Z_{\V, \nabla} \rightarrow Z_{\V^\prime,
 \nabla^\prime}$. This yields a functor
\begin{equation*}
\begin{aligned}
\Phi\colon\ \Vect^\sim_\nabla(M)
 & \longrightarrow \mathrm{Fun}\bigl(\Path(M), \Vect\bigr),
 \\ (\V, \nabla) & \longmapsto Z_{\V, \nabla},
 \\ \alpha & \longmapsto \eta^\alpha.
\end{aligned}
\end{equation*}
The fundamental result is now the following.

\begin{Theorem}
 \label{ThmEquivalencePathCat}
 The functor $\Phi$ just constructed is an equivalence of categories.
\end{Theorem}

\begin{Remark}
 At first sight, there may seem to be a clash of notations with the functor
 $\Phi\colon \Vect_{\nabla, \beta}^\sim(M) \rightarrow 1\TFT(M)$
 constructed in Section~\ref{SectionConstructionOfFieldTheories}.
 However, it is easy to check that in~fact the $\Phi$ just constructed
 is the composition of the functor $\Phi$ from before with the
 restriction functor $1\TFT(M) \rightarrow \mathrm{Fun}(\Path(M),
 \Vect)$. In particular, $Z_{\V, \nabla}$ is the restriction of
 $Z_{\V, \nabla, \beta}$ to~$\Path(M)$. Notice that the information
 about $\beta$ is lost in this restriction process.
\end{Remark}

\begin{Remark}[simplification]
 \label{RemarkSimplification}
 Let $\gamma, \eta\colon \R \rightarrow M$ be two smooth paths and
 suppose that \mbox{$\gamma(t) = \eta(t)$} for $t$ in some neighborhood of
 $a \in \R$. Then the identity map of $\R$ induces an~isomorphism
 $[\id]_a\colon (\gamma, a) \rightarrow (\eta, a)$. If now $\xi\colon \R
 \rightarrow M$ is a third path with $\xi(t) = \gamma(t) = \eta(t)$
 for~$t$ near $a$, we have the commutative diagram
 \begin{equation*}
 \begin{tikzcd}[column sep = small]
 & (\gamma, a) \ar[dr, "{[\id]_a}"] \ar[dl, "{[\id]_a}"'] &
 \\ (\eta, a) \ar[rr, "{[\id]_a}"'] & & (\xi, a)
 \end{tikzcd}
 \end{equation*}
 in $\Path(M)_0$. A smooth functor $Z\colon \Path(M) \rightarrow
 \Vect$ now comes with canonical coherent isomorphisms between the
 vector spaces $Z(\gamma, a)$, $Z(\eta, a)$ and $Z(\xi, a)$, given by
 the various $Z([\id]_a)$; this means that we can (and will)
 assume in the future that $Z(\gamma, a)$ is \emph{equal to} $Z(\eta,
 a)$ for paths that are equal near $a$.
\end{Remark}

In this section, we will prove the following weaker version of
Theorem~\ref{ThmEquivalencePathCat}, which states that $\Phi$ is an
equivalence when considered as a functor to $\mathrm{Fun} (\Path_c(M),
\Vect)$. The proof of Theorem~\ref{ThmEquivalencePathCat} will then
be completed by Proposition~\ref{PropConstructionOfSection} from
Section~\ref{SectionGeneralCase}, which reduces the general case to
the one just below.

\begin{Proposition}
 \label{PropClassificationSittingInstants}
 The composition $\res \circ\, \Phi$ is an equivalence of
 categories, where
 \begin{equation*}
 \res\colon\
 \mathrm{Fun}\bigl(\Path(M), \Vect\bigr)
 \longrightarrow \mathrm{Fun}\bigl(\Path_c(M), \Vect\bigr)
 \end{equation*}
 is the obvious restriction functor.
\end{Proposition}

We start our preparations for the proof of the above proposition with
a couple of lemmas, for which we fix a smooth functor $Z\colon
\Path(M)\rightarrow \Vect$. The following lemma uses Notation~\ref{Notation:StandardObjects}.

\begin{Lemma}[invertibility]
 \label{LemmaIso}
 For all paths $(\gamma; a, b)$ in $\Path(M)_1$, the vector bundle map $Z(\gamma; a, b)$ is
 invertible.
\end{Lemma}

\begin{proof}
 Invertibility can be checked pointwise, hence we may assume that
 $\gamma$ is a single path (i.e., a family over the point). We have
 $Z(\gamma; a, a) = \id$ as $Z(\gamma; a, a)$ is the image of
 $Z(\gamma; a)$ under the degeneracy $[1] \rightarrow [0]$. Since
 the set of invertible linear maps is open and the value of
 $Z(\gamma; a, b)$ depends smoothly (in particular continuously) on
 $a$ and $b$, we have
 \begin{equation*}
 b_0 := \inf\{ b \mid Z(\gamma; a, b) \text{ is not invertible} \} > 0.
 \end{equation*}
 Suppose that $b_0 < \infty$. Then, since the set of $b$ such that
 $Z(\gamma; a, b)$ is not invertible is a closed set, the infimum is
 actually a minimum. Therefore $Z(\gamma; a, b_0)$ is not invertible,
 but $Z(\gamma; a, b)$ is invertible for each $b< b_0$. Now
 \begin{equation*}
 Z(\gamma; a, b_0) = Z(\gamma; b, b_0) Z(\gamma; a, b)
 \end{equation*}
 for all $b \in [0, b_0]$. If now $b < b_0$, then $Z(\gamma; a, b)$
 is invertible by definition of $b_0$. On the other hand, since
 $Z(\gamma; b_0, b_0) = \id$, the linear map $Z(\gamma; b, b_0)$ is
 invertible for $b$ close enough to $b_0$. This leads to a
 contradiction to the assumption that $Z(\gamma; a, b_0)$ is not
 invertible, as for such $b$ close to $b_0$, the right hand side is a
 composition of two invertible maps. Hence we must have $b_0 =
 \infty$, which proves the lemma.
\end{proof}

\begin{Lemma}[trivial action]
 \label{LemmaActionTrivial}
 For any path $\gamma$ in $M$ that is constant near $a$ and any
 orientation-preserving diffeomorphism $F$ of $\R$ such that $F(a) =
 a$, the isomorphism
 \begin{equation*}
 Z([F]_a)\colon\ Z(\gamma; a)
 \longrightarrow Z(\gamma \circ F; a) = Z(\gamma; a)
 \end{equation*}
 is the identity. Here, the last identity is the one given by
 Remark~$\ref{RemarkSimplification}$.
\end{Lemma}

\begin{proof}
 Let $U \ni a$ be a small neighborhood in which $\gamma$ is constant
 and let $G\colon \R \rightarrow \R$ be a~dif\-feomorphism satisfying
 \begin{equation*}
 G(t) =
 \begin{cases}
 F(t), & t\ \text{ near }\ a, \\
 t, & t \not\in U.
 \end{cases}
 \end{equation*}
 Since $G$ has the same germ at $a$ as $F$, we have $[F]_a = [G]_a$.
 Moreover, $G(U) = U$, so $\gamma = \gamma \circ G$ and therefore $G$
 defines an automorphism $[G]_{[a, a+1]}$ of $(\gamma, a, a+1)$.
 Thus, the diagram
 \begin{equation*}
 \begin{tikzcd}
 Z(\gamma; a)
 \ar[rr, "{Z(\gamma; a, a+1)}"]
 \ar[d, "{Z([G]_a})"']
 & & Z(\gamma; a+1) \ar[d, "{Z([G]_{a+1})}"]
 \\
 Z(\gamma; a) \ar[rr, "{Z(\gamma; a, a+1)}"']
 & & Z(\gamma; F(a+1))
 \end{tikzcd}
 \end{equation*}
 commutes. However, $Z(\gamma, a, a+1)$ is invertible by
 Lemma~\ref{LemmaIso}, and $G$ is the identity near $t = a + 1$, so
 that $Z([G]_{a+1}) = \id$. Hence $Z([F]_a) = Z([G]_a) = \id$.
\end{proof}

Lemma~\ref{LemmaActionTrivial} lets us simplify our analysis as
follows. Denote by $T_a$ the translation diffeomorphism given by
\begin{equation}
 \label{Translations}
 T_a(t) = t+a.
\end{equation}
For any path $(\gamma, a)$ which is constant near $t=a$, this induces
an isomorphism $[T_{-a}]_a \colon (\gamma, a) \rightarrow (\gamma
\circ T_{a}, 0)$. Now, let $Z\colon \Path_c(M) \rightarrow \Vect$ be
a functor that satisfies the simplifying assumption of
Remark~\ref{RemarkSimplification}. We obtain a canonical isomorphism
\begin{equation}
 \label{DefinitionOfTau}
 \mathcal{T}_{\gamma, a}
 := Z([T_{-a}]_a) \colon\ Z(\gamma, a) \longrightarrow
 Z(\gamma \circ T_{a}, 0)
 = Z(\gamma(0), 0),
\end{equation}
as $\gamma \circ T_{a}$ has a sitting instant at $t=0$; here by abuse of
notation, $\gamma(0)$ denotes the constant path equal to $\gamma(0)$.
 From an orientation preserving diffeomorphism $F$, we get a commutative
diagram
\begin{equation*}
\begin{tikzcd}
 Z(\gamma, a) \ar[d, "{Z([F]_a)}"'] \ar[rr, "{Z([T_{-a}]_a)}"]
 & &
 Z(\gamma \circ T_{a}, 0)
 \ar[d, "{Z([T_{-F(a)} \circ F \circ T_{a}]_{0})}"]
 \\
 Z\big(\gamma \circ F^{-1}, F(a)\big) \ar[rr, "{Z([T_{-F(a)}]_{F(a)})}"']
 & &
 Z\big(\gamma \circ F^{-1} \circ T_{F(a)}, 0\big).
\end{tikzcd}
\end{equation*}
Now both $\gamma \circ T_{a}$ and $\gamma \circ F^{-1} \circ T_{F(a)}$
have a sitting instant at $t=0$, so the assumption from
Remark~\ref{RemarkSimplification} on $Z$ tells us that the two vector
spaces in the right column agree; $Z(\gamma \circ T_{a}, 0) = Z\big(\gamma
\circ F^{-1} \circ T_{F(a)}, 0\big)$. A priori, $Z\big([T_{-F(a)} \circ F
\circ T_{a}]_{0}\big)$ could be a nontrivial automorphism of this vector
space; however, since $\big(T_{-F(a)} \circ F \circ T_{a}\big)(0) = 0$, we
have $Z\big([T_{-F(a)} \circ F \circ T_{a}]_{0}\big) = \id$ by
Lemma~\ref{LemmaActionTrivial}. We see that the isomorphisms
$\mathcal{T}_{\gamma, a} = Z([T_{-a}]_a)$ satisfy the equivariance
property
\begin{equation}
 \label{EquivarianceT}
 \mathcal{T}_{\gamma, a}
 = \mathcal{T}_{\gamma \circ F^{-1}, F(a)} \circ Z([F]_a)
\end{equation}
for any diffeomorphism $F$ and all $a \in \R$. Using this
equivariance property, we obtain the following lemma.

\begin{Lemma}[normalization]
 \label{LemmaSimplification}
 Given a strict functor ${Z}\colon \Path_c(M)^\circ \rightarrow
 \Vect$, there exists a~strict functor $\tilde{Z}\colon
 \Path_c(M)^\circ \rightarrow \Vect$ together with a natural
 isomorphism $\mathcal{T}\colon Z \rightarrow \tilde{Z}$ such that
 $\tilde{Z}(\gamma, a) = \tilde{Z}(\gamma(a), 0)$ for all standard
 objects $(\gamma, a)$. For such a functor $\tilde{Z}$
 \begin{equation*}
 Z([F]_a)\colon\ \tilde{Z}(\gamma, a)
 = \tilde{Z}(\gamma(a), 0) \longrightarrow
 \tilde{Z}\big(\gamma \circ F^{-1}, F(a)\big)
 = \tilde{Z}(\gamma(a), 0)
 \end{equation*}
 acts as the identity for any diffeomorphism $F$ of $\R$.
\end{Lemma}

\begin{proof}
 We may assume that $Z$ satisfies the simplifying assumptions of
 Remark~\ref{RemarkSimplification}. Set
 \begin{equation*}
 \tilde{Z}(\gamma; \tau_0, \dots, \tau_n)
 := \bigl(\W_0, \dots, \W_n; {f}_1, \dots, {f}_n\bigr),
 \end{equation*}
 with $\W_j := Z(\gamma(\tau_j), 0)$, and
 \begin{equation*}
 {f}_j := \mathcal{T}_{\gamma, \tau_j}
 \circ Z(\gamma;\tau_{j-1}, \tau_j)
 \circ \mathcal{T}_{\gamma, \tau_{j-1}}^{-1},
 \end{equation*}
 where the $\mathcal{T}_{\gamma, a}$ are defined as in~\eqref{DefinitionOfTau}. Moreover, for a diffeomorphism $F$, set
 $Z\big([F]_{[\tau_0, \tau_n]}\big) := \id$. By the equivariance property~\eqref{EquivarianceT}, this gives a well-defined functor $\tilde
 Z\colon \Path_c(M)^\circ \rightarrow \Vect$.

 Finally, let $\mathcal{T}$ be given by
 \begin{equation*}
 \mathcal{T}_{(\gamma; \tau_0, \dots, \tau_n)}
 = (\mathcal{T}_{\gamma, \tau_0}, \dots, \mathcal{T}_{\gamma, \tau_n}).
 \end{equation*}
 This clearly defines a natural transformation $Z \rightarrow
 \tilde{Z}$. That $\tilde{Z}([F]_a)$ acts as the identity also
 follows directly from~\eqref{EquivarianceT}.
\end{proof}

We are now in a position to prove the main result of this section.

\begin{proof}[Proof of Proposition~\ref{PropClassificationSittingInstants}]
 That $\res\circ\Phi$ is full and faithful is shown just as
 in the proof of~Proposition~\ref{PropConstructionOfFieldTheories}. It therefore remains to show
 that $\res\circ\Phi$ is essentially surjective. Moreover, it~suffices to consider the functor on the subcategory of standard
 objects.

 Let $Z\colon \Path_c(M)^\circ \rightarrow \Vect$ be a strict smooth
 functor. We assume moreover that $Z$ is \emph{normalized} in the sense
 of Lemma~\ref{LemmaSimplification}; in other words, $Z(\gamma, a) =
 Z(\gamma(a), 0)$ for all paths $\gamma$ and all $a \in \R$, and
 $Z([F]_a) = \id$ for all diffeomorphisms $F$ on $\R$. In particular,
 this implies that
 \begin{equation}
 \label{ParametrizationInvariance}
 Z(\gamma \circ F; a, b) = Z(\gamma; F(a), F(b))
 \end{equation}
 as map from $Z(\gamma(a), 0)$ to $Z(\gamma(b), 0)$.

 First let us extract a vector bundle from $Z$. To this end, let $\point$ be
 the universal point introduced in \eqref{UniversalPoint}. Note that $\point
 \in \Path_c(M)$ and set $\V := Z(\point)$. For any path $(\gamma, a)$, our
 assumption on $Z$ then implies that $Z(\gamma, a) = \V_{\gamma(a)}$.
 Now in general for objects $(\gamma; \tau_0, \dots, \tau_n) \in
 \Path_c(M)^\circ_n$, we have
 \begin{equation*}
 Z(\gamma; \tau_0, \dots, \tau_n)
 = (\W_0, \dots, \W_n; {f}_1, \dots, {f}_n).
 \end{equation*}
 If $d_i\colon [0] \rightarrow [n]$, $i = 0, \dots, n$, is the map
 with image $i \in [n]$, we have
 \begin{equation*}
 \W_i = d_i^*Z(\gamma; \tau_0, \dots, \tau_n)
 = Z(d_i^*(\gamma; \tau_0, \dots, \tau_n))
 = Z(\gamma; \tau_i) = \V_{\gamma(\tau_i)},
 \end{equation*}
 hence the $\W_i$ are already determined by $\V$.

 It remains to determine the vector bundle maps ${f}_i$; we will
 use Proposition~\ref{PropMultiplicativity} for this. To~obtain a
 section $P$ as in the proposition, we use modification functions,
 which are defined as follows.

 \begin{Definition}[modication function]
 \label{DefModFunction}
 A two-sided \emph{modification function} is a smooth function
 $\chi\colon \R \rightarrow [0, 1]$ such that
 \begin{enumerate}[(1)]\itemsep=0pt
 \item $\chi$ is nondecreasing,
 \item $\chi(t) = 0$ for $t$ near zero,
 \item $\chi(t) = 1$ for $t$ near one.
 \end{enumerate}
 By $\chi_{a, b}$, $a \leq b$ we denote
 the function given by
 \begin{equation*}
 \chi_{a, b}(t) := a + (b-a)\chi\bigg( \frac{t-a}{b-a}\bigg),
 \end{equation*}
 for $t \in \R$. This function is then only nonconstant on $[a, b]$
 and takes values in $[a, b]$.
 \end{Definition}

 Later, in Definition~\ref{def:mod-fun-2}, we will introduce also
 left and right modification functions, as well as their family
 versions. For now, we drop the adjective ``two-sided''. If $\chi$ is
 a modification function and $\gamma \in C^\infty([0, 1], M)$, then
 $\gamma \circ \chi$ is a path that is defined on all of $\R$ and
 which is constant on $(-\infty, \varepsilon] \cup [1-\varepsilon,
 \infty)$ for some $\varepsilon >0$. Hence
 \begin{equation*}
 P(\gamma) := Z(\gamma \circ \chi; 0, 1)
 \end{equation*}
 is well-defined for each $\gamma \in C^\infty([0, 1], X)$. Note
 that $P(\gamma)$ maps $Z(\gamma \circ \chi, 0) = Z(\gamma(0); 0) =
 \V_{\gamma(0)}$ to $Z(\gamma \circ \chi; 1) = Z(\gamma(1);1) =
 \V_{\gamma(1)}$. This construction works in families and therefore
 we get a smooth section $P$ of the bundle $\mathrm{ev}_0^* \V^\vee
 \otimes \mathrm{ev}_1^*\V$, as required. The crucial result, which
 will be shown in Lemma~\ref{LemmaIndependenceOfChi} below, is then
 that $P(\gamma)$ is independent of the choice of modification
 function.

 We need to check that $P$ is multiplicative, in the sense of~\eqref{Multiplicativity}. For this, we must use the next simplicial
 level. For $a \in [0, 1]$ fixed, define $\xi\colon \R \rightarrow
 [0, 1]$ by
 \begin{gather*}
 \xi(t) :=
 \begin{cases}
 a\chi(t), & t \in [0, 1],
 \\ a + (1 - a)\chi(t-1), & t \in [1, 2].
 \end{cases}
 \end{gather*}
 Then $\xi^\prime(t) \geq 0$ everywhere, and $\xi$ is constant near
 $t= 0, 1, 2$ (with values $0$, $a$, $1$). Hence $t \mapsto \xi_2(t) :=
 \xi(2t)$ is a modification function in the sense of
 Definition~\ref{DefModFunction}. Therefore, by independence of the
 modification function (Lemma~\ref{LemmaIndependenceOfChi} below), we have
 \begin{equation*}
 P(\gamma)
 = Z(\gamma \circ \chi; 0, 1)
 = Z(\gamma \circ \xi_2; 0, 1)
 = Z(\gamma \circ \xi; 0, 2).
 \end{equation*}
 Notice that in the last step, we used the parametrization independence~\eqref{ParametrizationInvariance}.

 Now, since $\gamma \circ \xi$ is constant near $1$, $(\gamma \circ
 \xi; 0, 1)$ and $(\gamma \circ \xi; 1, 2)$ are also objects of
 $\Path_c(M)$, and we get
 \begin{equation*}
 Z(\gamma \circ \xi; 0, 2)
 = Z(\gamma \circ \xi; 1, 2) \circ Z(\gamma \circ \xi; 0, 1).
 \end{equation*}
 However, by definition, we have
 \begin{equation*}
 (\gamma \circ \xi)(t) =
 \begin{cases}
 (\mathfrak{s}_{0, a}\gamma \circ \chi)(t), & t \in [0, 1],
 \\ (\mathfrak{s}_{a, 1}\gamma \circ \chi)(t-1), & t \in [1, 2].
 \end{cases}
 \end{equation*}
 Hence, $Z(\gamma \circ \xi; 0, 1) = P(\mathfrak{s}_{0, a}\gamma)$
 and, again by \eqref{ParametrizationInvariance},
 \begin{equation*}
 Z(\gamma \circ \xi; 1, 2)
 = Z\big(\mathfrak{s}_{a, 1}\gamma \circ \chi \circ T_{-1}; 1, 2\big)
 = Z(\mathfrak{s}_{a, 1}\gamma \circ \chi; 0, 1)
 = P(\mathfrak{s}_{a, 1}\gamma).
 \end{equation*}
 Therefore,
 \begin{equation*}
 P(\gamma) = P(\mathfrak{s}_{a, 1}\gamma)P(\mathfrak{s}_{0, a}\gamma ),
 \end{equation*}
 that is, $P$ is indeed multiplicative. By Proposition~\ref{PropMultiplicativity}, there exists a connection $\nabla$ on
 $\V$ such that $P(\gamma)$ is given by parallel transport along
 $\gamma$ with respect to $\nabla$.

 To conclude the argument, let $s_i\colon [1] \rightarrow [n]$, $i=1,
 \dots, n$ be the map with $s_i(0) = i-1$ and $s_i(1) = i$, and
 notice that
 \begin{equation*}
 {f}_i = s_i^*Z(\gamma; \tau_0, \dots, \tau_n)
 = Z\bigl(s_i^*(\gamma; \tau_0, \dots, \tau_n)\bigr)
 = Z(\gamma; \tau_{i-1}, \tau_i).
 \end{equation*}
 Define $\tilde{\gamma}$ by $\tilde{\gamma}(t) := \gamma(\tau_{i-1} +
 (\tau_i - \tau_{i-1})t)$; then $Z(\gamma;\tau_{i-1}, \tau_i) =
 Z(\tilde{\gamma}; 0, 1)$, once more by~\eqref{ParametrizationInvariance}.
 Since $\tilde{\gamma}$ has sitting instants at $t = 0$ and $1$, we
 have $\tilde{\gamma} = \tilde{\gamma} \circ \chi$ for a suitable
 modification function (just choose $\chi$ in such a way that
 $\chi(t) = t$ wherever $\tilde{\gamma}$ is not constant). Finally,
 we get
 \begin{equation*}
 {f}_j = Z(\tilde{\gamma}; 0, 1)
 = Z(\tilde{\gamma}\circ \chi; 0, 1)
 = P(\tilde{\gamma}).
 \end{equation*}
 Since parallel transport is parametrization independent, this
 coincides with the parallel transport along $\gamma$ from
 $\tau_{i-1}$ to $\tau_i$, with respect to the connection $\nabla$.
 Hence $Z$ coincides with the restriction of $Z_{\V, \nabla} =
 \Phi(\V, \nabla)$ to $\Path_c(M)$, where $Z_{\V, \nabla}$ is the
 functor from Construction~\ref{ConstructionZVNabla}.
\end{proof}

\subsection{General paths} \label{SectionGeneralCase}

In this section, we finish the proof of Theorem~\ref
{ThmEquivalencePathCat}. Having Proposition~\ref
{PropClassificationSittingInstants} at hand, this will be achieved by
establishing the following result.

\begin{Proposition}
 \label{PropConstructionOfSection}
 The functor $\res$ from the previous section is an
 equivalence of categories. More specifically, there is a functor
 \begin{equation*}
 \ext \colon\ \mathrm{Fun}\bigl(\Path_c(M), \Vect\bigr)
 \longrightarrow \mathrm{Fun}\bigl(\Path(M), \Vect\bigr)
 \end{equation*}
 such that $\res \circ \ext = \id$, together with a natural
 isomorphism $\eta\colon \id \rightarrow \ext \circ \res $.
\end{Proposition}

We need several lemmas, for which we fix a smooth functor $Z\colon
\Path(M)\rightarrow \Vect$. We will assume that $Z$ is a strict
functor (which is possible by Lemma~\ref{LemmaStrictification}), and
we will also make the simplifying assumptions discussed in
Remark~\ref{RemarkSimplification}.

\begin{Lemma}[introducing a sitting instant]
 \label{LemmaSittingInstant}
 Fix numbers $a \leq b$ and let $F\colon \R \rightarrow \R$ be a
 smooth monotonically increasing function such that $F(t) = t$ for
 $t$ near $a$ and near $b$. Then, for all paths $\gamma\colon \R
 \rightarrow M$, we have $Z(\gamma; a, b) = Z(\gamma \circ F; a, b)$.
\end{Lemma}

We remark that if we choose $F$ to be constant somewhere in between $a$ and $b$ (as allowed by the lemma), the path $\gamma \circ F$ will have a sitting instant somewhere in between $a$ and $b$.

\begin{proof}
 If $F^\prime(x) \neq 0$ for all $x \in [a, b]$, so that $F$ is a
 diffeomorphism onto its image in a~neigh\-borhood of $[a, b]$, the
 simplicial structure yields the commutative diagram
 \begin{equation*}
 \begin{tikzcd}
 Z(\gamma; a)
 \ar[rr, "{Z(\gamma; a, b)}"]
 \ar[d, "{Z([F]_a)}"']
 & &
 Z(\gamma; b)
 \ar[d, "{Z([F]_b)}"]
 \\
 Z(\gamma \circ F; a)
 \ar[rr, "{Z(\gamma \circ F; a, b)}"']
 & &
 Z(\gamma \circ F; b).
 \end{tikzcd}
 \end{equation*}
 This, together with the fact that $F = \id$ near $a$ and $b$, so
 that $Z(\gamma\circ F; a) = Z(\gamma; a)$ and $Z([F]_a) = \id$, and
 similarly for $b$, proves the result in this case.

 The general case now follows from the fact that
 \begin{equation*}
 F_\varepsilon(t) := (1-\varepsilon)F(t) + \varepsilon t
 \end{equation*}
 is a one-parameter family of maps such that $F_0 = F$ and such that
 $F_{\varepsilon}$ is a diffeomorphism whenever $\varepsilon \in (0,
 1]$. Therefore, by the observations above, we have $Z(\gamma; a, b)
 = Z(\gamma \circ F_\varepsilon; a, b)$ for all $\varepsilon \in (0, 1]$, and, by
 continuity, the equality persists for $\varepsilon = 0$.
\end{proof}

We are now able to prove the following essential lemma, which proves
the independence of the choice of modification function (see
Definition~\ref{DefModFunction}).

\begin{Lemma}[independence of $\chi$]
 \label{LemmaIndependenceOfChi}
 Let $Z\colon \Path_c(X) \rightarrow \Vect$ be a functor. Then for
 any two modification functions $\chi$, $\tilde{\chi}$, we have
 \begin{equation*}
 Z(\gamma \circ \chi_{a, b}, a, b)
 = Z(\gamma \circ \tilde{\chi}_{a, b}, a, b)
 \end{equation*}
 as morphisms from $Z(\gamma(a), a)$ to $Z(\gamma(b), b)$.
\end{Lemma}

\begin{proof}
 Clearly, we may assume for simplicity that $a=0$ and $b=1$. Now, we
 first argue that we may furthermore assume that $\chi(t) =
 \tilde{\chi}(t) = t$ for $t$ in a neighborhood of $\tfrac{1}{2}$.
 Since $\chi$ is not constant, there exists some $t_0 \in (0, 1)$
 such that $\chi^\prime(t_0) \neq 0$. We may arrange that $t_0 =
 \tfrac{1}{2}$: Choose some diffeomorphism $F$ that is the identity
 near $t=0, 1$ and sends $\tfrac{1}{2}$ to $t_0$. Then $\xi := \chi
 \circ F$ is again a modification function, now satisfying
 $\xi^\prime(\tfrac{1}{2}) \neq 0$. We get that there exists a~neighborhood of $\tfrac{1}{2}$, where $\xi^\prime$ is invertible.
 Hence there exists some small $\varepsilon >0$ such that we can find
 a diffeomorphism $G$ of $\R$ with
 \begin{equation*}
 G(t) =
 \begin{cases}
 t
 & \text{if} \quad t \in
 (-\infty, \varepsilon] \cup [1-\varepsilon, \infty),
 \\[.5ex]
 \xi^{-1}(t)
 & \text{if} \quad t \in \big[\tfrac{1}{2}-\varepsilon,
 \tfrac{1}{2}+\varepsilon \big].
 \end{cases}
 \end{equation*}
 Then $\xi \circ G$ is a modification function that is the identity
 near $t = \tfrac{1}{2}$, and we have
 \begin{equation*}
 Z\bigl(\gamma \circ (\xi \circ G); 0, 1\bigr)
 = Z\bigl(\gamma \circ \chi \circ (F \circ G); 0, 1\bigr)
 = Z(\gamma \circ \chi; 0, 1)
 \end{equation*}
 since $F \circ G$ is the identity near $t = 0,\ 1$.

 By the above, after replacing $\chi$ and $\tilde{\chi}$ with
 equivalent modification functions, we may assume that $\chi(t) =
 \tilde{\chi}(t) = t$ near $t=\tfrac{1}{2}$. Now let $F$ and
 $\tilde{F}$ be monotonic functions on $\R$ such that
 \begin{enumerate}[(1)]\itemsep=0pt
 \item $F(t) = \tilde{F}(t) = t$ on $\big({-}\infty, -\tfrac{1}{2}+\varepsilon\big]$,
 \item $F(t) = \tilde{F}(t)$ on $(-\infty, 0]$,
 \item $F(t) = \chi(t)$ and $\tilde{F}(t) = \tilde{\chi}(t)$ on $[0, \infty)$.
 \end{enumerate}
 This is possible since $\chi$ and $\tilde{\chi}$ are both constant
 near zero. Now since $F$, $\tilde{F}$ are the identity near~$-\tfrac{1}{2}$ and $\tfrac{1}{2}$, Lemma~\ref{LemmaSittingInstant}
 gives
 \begin{equation}
 \label{FandTilde}
 Z\bigg(\gamma \circ F; -\frac{1}{2}, \frac{1}{2}\bigg)
 = Z\bigg(\gamma;-\frac{1}{2}, \frac{1}{2}\bigg)
 = Z\bigg(\gamma \circ \tilde{F};-\frac{1}{2}, \frac{1}{2}\bigg).
 \end{equation}
 On the other hand,
 \begin{gather}
 Z\bigg(\gamma \circ F; -\frac{1}{2}, \frac{1}{2}\bigg)
 = Z\bigg(\gamma \circ \chi; 0, \frac{1}{2}\bigg)
 \circ Z\bigg(\gamma \circ F;-\frac{1}{2}, 0\bigg),\nonumber
 \\[.5ex]
 Z\bigg(\gamma \circ \tilde{F}; -\frac{1}{2}, \frac{1}{2}\bigg)
 = Z\bigg(\gamma \circ \tilde{\chi}; 0, \frac{1}{2}\bigg)
 \circ Z\bigg(\gamma \circ F;-\frac{1}{2}, 0\bigg),\label{FandTilde2}
 \end{gather}
 by the construction of $F$ and $\tilde{F}$. As all morphisms
 involved are invertible in view of Lemma~\ref{LemmaIso}, combining~\eqref{FandTilde} with~\eqref{FandTilde2} implies that $Z\big(\gamma
 \circ \chi; 0, \tfrac{1}{2}\big) = Z\big(\gamma \circ \tilde{\chi}; 0,
 \tfrac{1}{2}\big)$. A similar argument shows that $Z\big(\gamma \circ \chi;
 \tfrac{1}{2}, 1\big) = Z\big(\gamma \circ \tilde{\chi}; \tfrac{1}{2}, 1\big)$;
 combining these observations finishes the proof.
\end{proof}

We are now ready to give the proof of the main result of this section.

\begin{proof}[Proof of Proposition~\ref{PropConstructionOfSection}]
We first define the functor $\ext $. For a smooth functor
 \begin{equation*}
 Z\colon\ \Path_c(M)\allowbreak \rightarrow \Vect,
 \end{equation*}
 we set $\ext Z(\gamma,
 a) := Z(\gamma(a); a)$ and
 \begin{equation*}
 \ext Z(\gamma; \tau_0, \dots, \tau_n)
 = Z(\bar\gamma; \tau_0, \dots, \tau_n),
 \end{equation*}
 where $\bar\gamma$ is given on the interval $[\tau_{j-1}, \tau_j]$
 by
 \begin{equation*}
 \bar\gamma(t) = \big(\gamma\circ\chi_{\tau_{i-1}, \tau_i}\big)(t)
 \end{equation*}
 for some modification function $\chi$. By
 Lemma~\ref{LemmaIndependenceOfChi}, this is independent of the
 choice of modification function. Let $F$ be a diffeomorphism of
 $\R$, which defines a morphism
 \begin{equation*}
 [F]_{a} \colon\ (\gamma \circ F, a) \longrightarrow (\gamma, F(a))
 \end{equation*}
 in $\Path(M)_0$. Let $T_a$, $a \in \R$ be the translations defined
 in~\eqref{Translations} above; we then set
 \begin{equation*}
 \ext Z([F]_a) := Z\bigl( \big[T_{F(a)-a}\big]_a\colon \bigl(\gamma(F(a)), a\bigr)
 \rightarrow \bigl(\gamma(F(a)), F(a)\bigr)\bigr).
 \end{equation*}
 Using the simplicial structure and Remark~\ref{RemarkStackification}, this determines the functor
 $\ext Z$ completely.

 In order to show that $\ext Z$ is well defined, we need to
 check functoriality. Let $F$ be a diffeomorphism of $\R$, which
 defines an automorphism
 \begin{equation*}
 [F]_{[a, b]} \colon\
 (\gamma\circ F;a, b) \longrightarrow (\gamma; a^\prime, b^\prime),\qquad
 \text{where}\quad
 a^\prime := F(a),\quad b^\prime := F(b),
 \end{equation*}
 in $\Path(M)_q$. We need to show that the square
 \begin{equation*}
 \begin{tikzcd}
 Z\bigl((\gamma \circ F)(a), a\bigr)
 \ar[d, "{Z([T_{a^\prime - a}]_{a}})"']
 \ar[rr, "{Z(\overline{\gamma\circ F}; a, b)}"]
 & &
 Z\bigl((\gamma \circ F)(b); b\bigr)
 \ar[d, "{Z([T_{b^\prime - b}]_{b}})"]
 \\
 Z\bigl(\gamma(a^\prime), a^\prime\bigr)
 \ar[rr, "{Z(\bar{\gamma}; a^\prime, b^\prime)}"']
 & &
 Z\bigl(\gamma(b^\prime), b^\prime\bigr)
 \end{tikzcd}
 \end{equation*}
 commutes. We have
 \begin{equation*}
 Z\big(\overline{\gamma\circ F}; a, b\big)
= Z\big(\gamma \circ F \circ \chi_{a, b}; a, b\big)
= Z([F]_b)^{-1}
 Z\big(\gamma \circ F \circ \chi_{a, b} \circ F^{-1}; a^\prime, b^\prime\big) Z([F]_a).
 \end{equation*}
 Now first notice that $F \circ \chi_{a, b} \circ F^{-1} =
 \tilde{\chi}_{a^\prime, b^\prime}$ for some modification function
 $\tilde{\chi}$, hence the middle term equals $Z(\bar{\gamma};
 a^\prime, b^\prime)$ (here, of course, we use
 Lemma~\ref{LemmaIndependenceOfChi} again). Secondly,
 \begin{equation*}
 Z([F]_a) = Z\big([T_{a^\prime - a}]_a\big) Z\big([T_{a^\prime - a}^{-1} \circ F]_a\big),
 \end{equation*}
 with $Z\big([T_{a^\prime - a}^{-1} \circ F]_a\big) = \id$ by
 Lemma~\ref{LemmaActionTrivial}, as $\big(T_{a^\prime - a}^{-1} \circ
 F\big)(a) = a$. Using a similar argument for $Z([F]_b)$, we obtain
 \begin{equation*}
 Z\big(\overline{\gamma\circ F}; a, b\big) = Z\big([T_{b^\prime-b}]_b\big)^{-1}
 Z(\bar{\gamma}; a^\prime, b^\prime) Z\big([T_{a^\prime-a}]_a\big),
 \end{equation*}
 which was the claim.

 If $\eta\colon Z \rightarrow Z^\prime$ is a smooth natural
 transformation between smooth functors $Z, Z^\prime\colon
 \Path_c(M) \rightarrow \Vect$, we set
 \begin{equation*}
 (\ext \eta)_{(\gamma, a)} := \eta_{(\gamma(a), a)}.
 \end{equation*}
 It is straightforward to check that this gives a natural
 transformation $\ext \eta\colon \ext Z \rightarrow
 \ext Z^\prime$ and that the assignment $\eta \mapsto
 \ext \eta$ is functorial in $\eta$. This finishes the
 definition of $\ext $.

\medskip

To see that $\res \circ \ext = \id$, we have to check
that for paths $\gamma$ that already have sitting instants at $\tau_0,
\dots, \tau_n$, one has $Z(\gamma; \tau_0, \dots, \tau_n) =
Z(\bar\gamma; \tau_0, \dots, \tau_n)$. However, for such a path
$\gamma$, there exists a modification function $\chi$ such that on
each subinterval $[\tau_{j-1}, \tau_j]$, we have $\gamma = \gamma
\circ \chi_{\tau_{j-1}, \tau_j}$. Moreover, the composition of two
modification functions is again a modification function, hence the
claim follows once more from Lemma~\ref{LemmaIndependenceOfChi}.

\medskip

Finally, we construct a natural isomorphism $\eta\colon \id
\rightarrow \ext \circ \res$. To this end, choose
numbers $0 < \delta < \ell$ and a modification function $\chi$ such
that $\chi(t) = t$ for $t$ near $\delta$ and set
\begin{equation*}
 \eta_{\gamma, a}
 := Z(\gamma\circ \chi_{a, a+\ell}; a, a+\delta)^{-1}
 \circ Z(\gamma; a, a+\delta)
 \colon\ Z(\gamma; a) \longrightarrow Z(\gamma(a), a).
\end{equation*}

We claim that this definition is independent of the modification
function $\chi$ and the choice of~$\delta$ and $\ell$. For notational
simplicity, let $a = 0$ and suppose $\ell = 1$; the case $a \neq 0$ is
similar. Now, let $\tilde{\chi}$ be another modification function
also satisfying $\tilde{\chi}(t) = t$ for $t$ near $\delta$. Since the
values of a~modification function on $[\delta, 1]$ are irrelevant for
the definition of $\eta$, we may as well assume that $\tilde{\chi} =
\chi$ on $[\delta, 1]$ (which then implies that they in fact agree on
a neighborhood of $[\delta, 1]$). Now, by~Lemma~\ref{LemmaIndependenceOfChi}, we have $Z(\gamma\circ \chi; 0, 1)
= Z(\gamma \circ \tilde{\chi}; 0, 1)$. On the other hand
\begin{equation*}
 Z(\gamma\circ \chi; 0, 1)
 = Z(\gamma \circ \chi; \delta, 1)
 \circ Z(\gamma \circ \chi; 0, \delta)
 = Z(\gamma \circ \tilde{\chi}; \delta, 1)
 \circ Z(\gamma \circ \chi; 0, \delta)
\end{equation*}
since $\chi$ and $\tilde{\chi}$ agree in a neighborhood of
$\delta$. But this equals
\begin{equation*}
 Z(\gamma\circ \tilde{\chi}; 0, 1)
 = Z(\gamma \circ \tilde{\chi}; \delta, 1)
 \circ Z(\gamma \circ \tilde{\chi}; 0, \delta),
\end{equation*}
from which we obtain the desired equality $Z(\gamma \circ \chi; 0,
\delta) = Z(\gamma \circ \tilde{\chi}; 0, \delta)$, by virtue of
Lemma~\ref{LemmaIso}. To see the independence from $\delta$, let
$\tilde{\chi}$ be a modification function with $\tilde{\chi}(t) = t$
near $\tilde{\delta}$. Without loss of generality, suppose that
$\tilde{\delta} < \delta$. By the first step, we are free to choose
the modification function $\chi$ any way we like, under the constraint
that $\chi(t) = t$ near $\delta$. We now choose it in such a way that
in fact $\chi(t) = t$ in a neighborhood of the interval
$\big[\tilde{\delta}, \delta\big]$. Then
\begin{equation*}
 Z(\gamma \circ \chi; 0, \delta)
 = Z\big(\gamma \circ \chi; \tilde{\delta}, \delta\big)
 \circ Z\big(\gamma\circ \chi; 0, \tilde{\delta}\big).
\end{equation*}
However, on $\big[\delta, \tilde{\delta}\big]$, we have $\gamma \circ \chi =
\gamma$, and by the first step, we have $Z\big(\gamma\circ \chi; 0,
\tilde{\delta}\big) = Z\big(\gamma\circ \tilde{\chi}; 0, \tilde{\delta}\big)$,
so that
\begin{equation*}
 Z(\gamma \circ \chi; 0, \delta)
 = Z\big(\gamma; \tilde{\delta}, \delta\big) \circ
 Z\big(\gamma\circ \tilde{\chi}; 0, \tilde{\delta}\big).
\end{equation*}
This finishes the argument that $\eta_{\gamma, a}$ does not depend on
$\delta$. The independence of the choice of $\ell$ now follows
immediately.

In order to show that $\eta$ indeed gives rise to a natural
transformation, we have to show that for any diffeomorphism $F$ of
$\R$, we have the equivariance property
\begin{equation}
 \label{EquivarianceEta}
 \eta_{\gamma\circ F, a} = \eta_{\gamma, F(a)} \circ Z([F]_a).
\end{equation}
First notice that for any path $\xi$ and any diffeomorphism $F$, we
have a commuting square
\begin{equation*}
 \begin{tikzcd}
 Z(\xi \circ F; a) \ar[rr, "Z(\xi \circ F; a\text{,} b)"]
 \ar[d, "Z(\text{[}F\text{]}_a)"']
 & &Z(\xi; F(a)) \ar[d, "Z(\text{[}F\text{]}_b)"]
 \\
 Z(\xi; F(a)) \ar[rr, "Z(\xi; F(a)\text{,} F(b))"']
 & &Z(\xi; F(b)).
 \end{tikzcd}
\end{equation*}
Applying this to $\xi = \gamma \circ F \circ \chi_{a, a+1} \circ
F^{-1}$, we have
\begin{align*}
 \eta_{\gamma \circ F, a}
 & = Z(\xi \circ F; a, a+\delta)^{-1}
 \circ Z(\gamma \circ F; a, a+\delta)
 \\
 & = Z([F]_a)^{-1}
 \circ Z\bigl(\xi; F(a), F(a+\delta)\bigr)^{-1}
 \circ Z\bigl(\gamma; F(a), F(a+\delta)\bigr)
 \circ Z([F]_a)
\end{align*}
Notice that the two appearances of $Z([F]_a)$ above in fact denote
different things: The one on~the right is an isomorphism $Z(\gamma
\circ F; a+\delta) \rightarrow Z(\gamma; F(a+\delta))$, while the one
on the left is an~isomorphism $Z(\xi \circ F; a) \rightarrow Z(\xi;
F(a))$. The latter is trivial by Lemma~\ref{LemmaActionTrivial} since
$\xi \circ F$ is constant near $a$ and $\xi$ is constant near $F(a)$.
Finally, it is easy to see that $F \circ \chi_{a, a+1} \circ F^{-1} =
\tilde{\chi}_{F(a), F(a+1)}$ for some modification function
$\tilde{\chi}$. Hence the last expression equals $\eta_{\gamma, F(a)}
\circ Z([F]_a)$, which finishes the proof of identity~\eqref{EquivarianceEta}.

Now the functions $\eta_{\gamma, a}$ give a natural transformation
$\eta\colon\id \rightarrow \ext \circ \res$ as
follows. For each functor $Z\colon \Path_c(M)\! \rightarrow\! \Vect$, we
need to give a natural transformation $\eta^{Z}\colon Z\! \rightarrow\!
\ext \big(Z|_{\Path_c(M)}\big)$. The components of this natural
transformation are given by
\begin{equation*}
 \eta^Z_{(\gamma; \tau_0, \dots, \tau_n)}
 = (\eta_{\gamma, \tau_0}, \dots, \eta_{\gamma, \tau_n}),
\end{equation*}
with the $\eta_{\gamma, \tau_j}$ as constructed above. This gives an
isomorphism from
\begin{equation*}
Z(\gamma; \tau_0, \dots, \tau_n) \longrightarrow \ext
Z(\gamma; \tau_0, \dots, \tau_n) = Z(\overline{\gamma}; \tau_0, \dots,
\tau_n),
\end{equation*}
 as required. The equivariance property~\eqref{EquivarianceEta} shows that these isomorphisms indeed fit
together to give a natural transformation.

The last thing to show is that if $\mu\colon Z \rightarrow Y$ is a
smooth natural transformation of smooth functors $Y, Z\colon
\Path_c(M) \rightarrow \Vect$, then $\eta^Y \circ \mu = \ext
\mu \circ \eta^Z$. But this is trivial, since $\eta^Z$ and $\eta^Y$
act as the identity on $\mathrm{Path}_c(M)$ and $\ext
\mu|_{\Path_c(M)} = \mu$.
\end{proof}

\subsection{Proof of the classification theorem} \label{SectionFinishProof}

We are now in a position to prove Theorem~\ref{Thm:Classification};
more precisely, we will prove that the functor~$\Phi$ from
Proposition~\ref{PropConstructionOfFieldTheories} is essentially
surjective.

\begin{proof}[Proof of Theorem~\ref{Thm:Classification}]
 Let $Z\colon 1\Bord(M) \rightarrow \Vect$ be a field theory. By the
 results of the previous section, we may assume that $Z|_{\Path(M)} =
 \Phi(\V, \nabla)$, where
 \begin{equation*}
 \Phi\colon\
 \mathrm{Fun}\bigl(\Path(M), \Vect\bigr) \longrightarrow \Vect_\nabla^\sim(M)
 \end{equation*}
 is the equivalence constructed in Section~\ref{SectionFunctorsOnPathM}. This means that $Z(\gamma, a) =
 \V_{\gamma(a)}$ for all paths $\gamma$ in~$M$ and all $a$, and that
 $Z(\gamma; a, b)$ is given by parallel transport along $\gamma$ from
 $a$ to $b$, with respect to the connection $\nabla$.

 To get a bilinear form $\beta$ on $\V$, consider the \emph{constant
 right elbow}, which is the bordism $R = ((\R\times M)/M; \rho_0,
 \rho_1; \gamma_{\mathrm{const}})$, where $\gamma_{\mathrm{const}}(t,
 p) = p$, $\rho_1 \equiv -1$ and $\rho_0(t, p) := t(1-t)$. Then,
 canonically,
 \begin{equation*}
 d_0^* R \cong (\gamma_{\mathrm{const}}, 0)
 \amalg (\gamma_{\mathrm{const}}, 1)\qquad
 \text{and}\qquad
 d_1^*R = \varnothing.
 \end{equation*}
 Hence $Z(R)$ is a linear map from $Z(\gamma_{\mathrm{const}}, 0)
 \otimes Z(\gamma_{\mathrm{const}}, 1)$ to $\mathbb{K}$. However,
 $Z(\gamma_{\mathrm{const}}, 0) = Z(\gamma_{\mathrm{const}}, 1) = \V$,
 so we get a bilinear form $\beta := Z(R)$ on $\V$. It is symmetric
 because the diffeomorphism $F\colon \mathbb R \to \mathbb R$,
 \begin{equation*}
 F(t) = \frac{1}{2} - t,
 \end{equation*}
 determines an automorphism of $R$ in $1\Bord(M)$, the restriction of which
 to $R_0^0$ gets mapped to
 the symmetry isomorphism of $\V \otimes \V$ under $Z$. This is because
 $F$ swaps the two components of~$R^0_0 = (\{0, 1\} \times M)/M$.

 To see that $\beta$ is nondegenerate, let $L$ be the \emph{constant left
 elbow}, which is the bordism given by $L = (\R\times M/M;
 \vartheta_0, \vartheta_1; \gamma_{\mathrm{const}})$, where
 $\vartheta_0(t) \equiv 1$ and $\vartheta_1(t) = -\rho_0(t)$. Set
 $\tau := Z(L)$, which is a section of the bundle
 $\mathrm{Hom}(\mathbb{K}, \V \otimes \V)$, i.e., a section of $\V
 \otimes \V$. The ``snake identity''
 \begin{equation*}
 (\id \otimes \beta) \circ (\tau \otimes \id) = \id
 \end{equation*}
 is satisfied. (To be precise, the left hand side is in fact a map
 from $\mathbb{K} \otimes \V$ to $\V \otimes \mathbb{K}$, but this
 can be canonically be identified with an endomorphism of $\V$; the
 requirement is that this endomorphism be the identity.) We now
 analyze this identity on each fiber; to this end, write $\tau_p
 = \sum_{ij} v_i \otimes w_j$ for some elements $v_i, w_j \in
 \V_p$. Then for any $u \in \V_p$, we have
 \begin{equation*}
 (\id \otimes \beta) \circ (\tau \otimes \id)(1 \otimes u)
 = \sum_{i, j}(\id\otimes \beta)(v_i\otimes w_j \otimes u)
 = \sum_{i,j} \beta(w_j, u) v_i \otimes 1.
 \end{equation*}
 The requirement that this be equal to $u \otimes 1$ implies that
 $\beta$ must be nondegenerate, as claimed. Note that it also
 implies that $\tau = \sum_j \varepsilon_j b_j \otimes b_j$, where
 $b_1, \dots, b_n$ is a generalized orthonormal basis for $\beta$;
 thus, $\epsilon$ is determined by $\beta$.

 We now want to show that $Z = Z_{\V, \nabla, \beta} = \Phi(\V,
 \nabla, \beta)$, where $\Phi$ denotes the functor constructed in
 Proposition~\ref{PropConstructionOfFieldTheories}. We already know
 that $Z= Z_{\V, \nabla, \beta}$ on simplicial level zero and $Z(B) =
 Z_{\V, \nabla, \beta}(B)$ on all bordisms $B$ that are (tensor
 products of) intervals and/or constant elbows. Using the techniques
 from the previous section, we can introduce sitting instants into
 any non-constant elbow $B$ and then express $B$ as the composition
 of two intervals and a constant elbow. This determines the field
 theory on all intervals and elbows, as well as on circles, since any
 circle can be decomposed into two elbows. Hence we have $Z =
 \Phi(\V, \nabla, \beta)$.
\end{proof}

\subsection{The oriented case}
\label{SectionTheOrientedCase}

We now briefly comment on the oriented case. The geometry considered
here is the one considered in Example~\ref{Ex:Orientations}, where manifolds are endowed with
orientations. In dimension one, the resulting bordism category will be
denoted $1\Bord^{\mathrm{or}}(M)$ and we write
$1\TFT^{\mathrm{or}}(M)$ for the corresponding groupoid of field
theories.

The main difference of $1\Bord^{\mathrm{or}}(M)$ to the unoriented
bordism category is that there are now two different kinds of points:
Remember that a point is given as the zero set of the cut function~$\rho_0$ on a one-dimensional manifold (respectively a family of
such). Now because $\rho_0(x) = 0$ for $x \in X$, the orientation allows to
ask whether ${\rm d}\rho_0(x)$ (which is non-zero and hence a basis of~$T_x
X$) is positively oriented or negatively oriented. This leads to
\emph{positive} respectively \emph{negative points}. Correspondingly,
we have one more elementary bordism, as displayed in
Figure~\ref{fig:or-1-bord}.

\begin{figure}[h]
 \centering
 \begin{tikzpicture}[x=1cm, y=2cm]
 \draw[bordism, oriented]
 (1, 1) pic{dot} node[above]{$+$}
 -- (1, 0) pic{dot} node[below]{$+$};
 \draw[bordism, oriented]
 (2, 0) pic{dot} node[below]{$-$}
 -- (2, 1) pic{dot} node[above]{$-$} ;
 \draw[bordism, oriented=0.25]
 (3, 1) pic{dot} node[above]{$+$}
 to[out=-90, in=-90, looseness=5]
 (4, 1) pic{dot} node[above]{$-$};
 \draw[bordism, oriented=0.25]
 (5, 0) pic{dot} node[below]{$-$}
 to[out=90, in=90, looseness=5]
 (6, 0) pic{dot} node[below]{$+$};
 \draw[bordism, oriented]
 (7.5, 0.5) circle[x radius = .5, y radius = .35];
 \draw[cut] (0, 1) -- (8.5, 1);
 \draw[cut] (0, 0) -- (8.5, 0);
 \draw[axis, ->] (0, 1.2) -- (0, -0.2) node[below]{$t$};
 \draw \foreach \i in {0,1} {(0, 1- \i) node[left] {$\tau_\i$}};
 \end{tikzpicture}
 \caption{All possible connected oriented bordisms. We call them
 positively and negatively oriented intervals, left elbow, right
 elbow, and circle, respectively. The cut functions are $\rho_i = t
 - \tau_i$, so these pictures are read from top to bottom.}
 \label{fig:or-1-bord}
\end{figure}
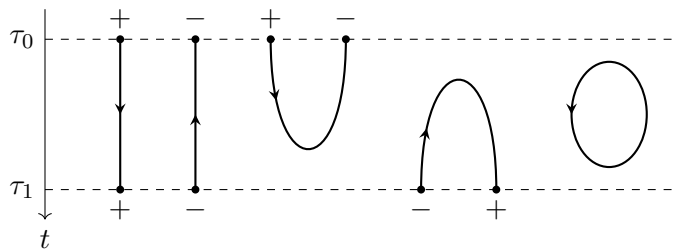

A field theory $Z \in 1\TFT^{\mathrm{or}}$ will assign vector spaces
$V^\pm$ to the two different points. The elbow then provides a
pairing $V^+ \otimes V^- \rightarrow \C$, which must be non-degenerate
due to the snake identity; hence $V^-$ can be identified with the dual
space of $V^+$ and vice versa. In particular, we~do \emph{not} obtain
the additional datum of a nondegenerate bilinear form on our vector
spaces.

Passing from ordinary field theories to oriented field theories over a
target $M$, the result is the following.

{\samepage\begin{Theorem}[classification of $1$-TFTs, oriented case]
 \label{Thm:ClassificationOriented}
 There is an equivalence of groupoids
 \begin{equation*}
 1\TFT^{\mathrm{or}}(M) \cong \Vect^{\sim}_{\nabla}(M),
 \end{equation*}
 which is natural in $M$.
\end{Theorem}

}

Here $\Vect^{\sim}_{\nabla}(M)$ denotes the groupoid of vector bundles
with connection on $M$, with morphisms given by connection-preserving
bundle isomorphisms. The proof is analogous to that of~Theorem~\ref{Thm:Classification}.

\section{Field theories with values in sheaves}
\label{sec:values-in-sheaves}

In this section, we consider notions of ``families of vector spaces''
more general than vector bundles. Let $\mathcal V$ denote some stack
of $C^\infty$-modules. This means that for each manifold $S$, objects
of $\mathcal V(S)$ are modules over $C^\infty(S)$, possibly with a
bornology or topology of a particular kind. Below, we will fix suitable
conditions $\mathcal V$ should satisfy, but we would like, at a
minimum, that the operation of taking global sections determines an
embedding $\Vect \to \mathcal V$, and a symmetric monoidal structure
on $\mathcal V$ compatible with this embedding. In line with
Definition~\ref{def:gft}, we write
\begin{equation*}
 1\TFT(M; \mathcal V)
 = \Fun^\otimes(1\Bord(M), \mathcal V)
\end{equation*}
for the groupoid of $1$-dimensional field theories over $M$ taking
values in $\mathcal V$, and similarly for the oriented variant. In
the TFT case, it is expected that this introduces no new examples,
that is, a field theory with values in $\mathcal V$ automatically
takes values in $\Vect$. In this section, we verify that this is
indeed the case.

\begin{Definition}
 %\label{def:admissible}
 We will call a symmetric monoidal smooth stack $\mathcal V$ an
 \emph{admissible stack of $C^\infty$-modules} if the following
 holds.
 \begin{enumerate}\itemsep=0pt
 \item[(C1)] For each $S$, $\mathcal V(S)$ is an additive category
 with kernels, naturally in $S$, such that $\otimes$ is additive in
 both variables.
 \item[(C2)] There is a linear symmetric monoidal fibered functor
 $\Vect \to \mathcal V$ which determines, for each~$S$, an
 equivalence between $\Vect(S)$ and the dualizable objects of
 $\mathcal V(S)$.
 \item[(C3)] Let $V \in \mathcal V(S \times \mathbb R)$, $W \in
 \mathcal V(S)$, and $\pr\colon S \times \mathbb R \to S$ be the
 projection. If $\sigma\colon V \to \pr^* W$ is such that
 $\sigma\vert_U = 0$ for every open $U \subset S \times \R$ with $U
 \cap (S \times \{0\}) = \varnothing$, then $\sigma = 0$.
 \end{enumerate}
\end{Definition}

Condition (C3) is a kind of separation axiom. It is often useful to
regard a morphism $W^\prime \to W$ in $\mathcal V(S)$ as a generalized
element of $W$. Then a morphism $\sigma\colon V \to \pr^*W$ as above
can be interpreted as a ``generalized path of elements'' of $W$, and
the axiom says that a generalized path vanishes altogether provided
it vanishes on $\mathbb R \setminus 0$.

In Section~\ref{sec:admissible-stacks}, we give examples of
admissible stacks of $C^\infty$-modules. Before doing this, we state
and prove the following theorem, which contains the previously
mentioned result that the target does not matter for topological field
theories.

\begin{Theorem}
\label{thm:sheafy}
 For any admissible stack of $C^\infty$-modules $\mathcal V$ and any manifold $M$, the
 inclusions
 \begin{equation*}
 1\TFT^{\mathrm{or}}(M) \to 1\TFT^{\mathrm{or}}(M; \mathcal V),
 \qquad
 1\TFT(M) \to 1\TFT(M; \mathcal V)
 \end{equation*}
 are equivalences of groupoids.
\end{Theorem}

\begin{Remark}
 \label{rmk:final-version}
 Combining the above with Theorems~\ref{Thm:Classification} and~\ref{Thm:ClassificationOriented} gives us equivalences of
 groupoids
 \begin{equation*}
 1\TFT(M; \mathcal V) \cong \Vect_{\nabla, \beta}^\sim(M),
 \qquad
 1\TFT^{\mathrm{or}}(M; \mathcal V) \cong \Vect_\nabla^\sim(M)
 \end{equation*}
 for any admissible $\mathcal V$. This is the exact meaning we
 intend for Theorems~\ref{MainThm_or} and~\ref{MainThm_nonor}, stated
 in~less detail in the introduction.
\end{Remark}

\subsection{Proof of Theorem~\ref{thm:sheafy}}

We will focus on the unoriented case, the oriented one being similar.
Throughout, we fix a field theory $Z\colon 1\Bord(M) \to \mathcal V$,
which we assume to be strict as a morphism of simplicial stacks.

\begin{Lemma}
% \label{lem:dualizability}
 Let $c\colon M \times \mathbb R \to M$ be the tautological
 $M$-family of constant paths. Then the object $Z(c; 0) \in
 \mathcal V$ is dualizable.
\end{Lemma}

\begin{proof}
 For $\gamma\colon S \times \mathbb R \to M$ and $a < b$, we denote
 by $L_{\gamma; a, b}$ the left elbow given by this data, that is,
 the $S$-family of bordisms with underlying map $\gamma$ and cut
 functions $\rho_0(s, t) = (t - a)(t - b)$ and $\rho_1(s, t)=
 - \frac{1}{4}(b-a)^2$, where $t$ is the coordinate of $\mathbb R$
 (here any $\rho_1$ can be chosen such that it has no zeroes, and such that
 $\rho_1 \leq \rho_0$ everywhere). We write
 similarly $R_{\gamma; a, b}$ for the right elbow. To prove the lemma,
 we consider in particular $S = M$ and $\gamma = c$; then for each
 $\delta > 0$, the images of $L_{c; 0, \delta}$ and $R_{c; 0,
 \delta}$ are canonically identified as maps
 \begin{gather*}
 Z(L_{c; 0, \delta}) \colon\ Z(c; 0) \otimes Z(c; 0) \to \epsilon_M,
 \\
 Z(R_{c; 0, \delta})\colon\ \epsilon_M \to Z(c; 0) \otimes Z(c; 0),
 \end{gather*}
 where $\epsilon_M$ is the monoidal unit of $\mathcal V(M)$. We
 claim that these morphisms are independent of $\delta$. To see
 this, pick any diffeomorphism $\chi\colon \mathbb R \to \mathbb R$
 with $\chi(0) = 0$, $\chi(1) = \delta$ which is affine of slope $1$
 near $0$ and $1$. Then $\chi$ determines a morphism in
 $1\Bord(M)_1$ between $L_{c; 0, \delta}$ and $L_{c \circ \chi; 0, 1}
 = L_{c; 0, 1}$, compatible with the usual identifications of their
 boundary components with $(c; 0)$. This proves the claim.

 Write $L = Z(L_{c; 0, \delta})$, $R = Z(R_{c; 0, \delta})$ for this
 common value. Then, expressing the interval bordism $(c, 0,
 \delta)$ as a suitable composition of $L_{c; 0, \delta}$ and $R_{c;
 0, \delta}$, we get
 \begin{equation*}
 (\mathrm{id} \otimes L) \circ (R \otimes \mathrm{id})
 = Z(c; 0, \delta)
 \end{equation*}
 for any $\delta > 0$. Now consider, for instance, $(c, 0,
 \delta^2)$ as an $(M \times \mathbb R)$-family of intervals, where
 $\delta$ now denotes the coordinate on the factor $\mathbb R$. Its
 image under $Z$ is a family of endomorphisms of $Z(c; 0)$ that is
 constant away from $\delta = 0$, and equal to the identity at
 $\delta = 0$. Thus, from the separation axiom of $\mathcal V$,
 $Z(c; 0, \delta) = Z(c; 0, 0) = \mathrm{id}_{Z(c; 0)}$ for all
 $\delta$. This proves that $L$ and $R$ provide the desired
 evaluation and coevaluation maps.
\end{proof}

Recall our notation $\Path_c(M) \subset 1\Bord(M)$ for
the subcategory of paths with sitting ins\-tants (Definition~\ref{DefPathCat}). Using the above lemma
and axiom (C2) for $\mathcal V$, we find a factorization of~$Z\vert_{\Path_c(M)}$ through $\Vect \hookrightarrow \mathcal V$
(essentially uniquely). Proposition~\ref{PropClassificationSittingInstants} then gives us a vector bundle with
connection $(V, \nabla)$ on $M$ and an isomorphism
\begin{equation*}
 Z\vert_{\Path_c(M)} \cong Z_{V, \nabla}\vert_{\Path_c(M)}
\end{equation*}
of smooth functors, where $Z_{V, \nabla}$ denotes the field theory
determined by parallel transport on $V$. We will eventually show that
$Z \cong Z_{V, \nabla, \beta}$ for some pairing $\beta$.

Lemma~\ref{LemmaSittingInstant} is easily restated for families of
bordisms and reproved for $\mathcal V$-valued field theories, using
the separation axiom. We now adapt some other definitions and lemmas
from Section~\ref {sec:classification} for which special care with the
family aspect is crucial. Let $\gamma\colon S \times \mathbb R \to M$
be an $S$-family of paths and let $a\colon S \to \mathbb R$ be any smooth
function, so that we can talk about the $S$-family of objects
$(\gamma; a)$ (where its cut function is $\rho_a(s, t) = t - a(s)$). Note
that by the implicit function theorem, any cut function $\rho$ on $S
\times \mathbb R$ is equivalent to some $\rho_a$: Just take $a$
defined by $\rho(s, a(s)) = 0$. Similarly, two functions $a \leq
b\colon S \to \mathbb R$ determine a family of intervals $(\gamma; a,
b)$. These examples exhaust all possible isomorphism classes of objects, respectively bordisms.

\begin{Definition}
 \label{def:mod-fun-2}
 Given functions $a, b\colon S \to \mathbb R$ with $a < b$, a
 \emph{left modification function} is a smooth map $\chi\colon S \times
 \mathbb R \to S \times \mathbb R$ of the form $\chi(s, t) = (s,
 \tilde{\chi}(s, t))$ for some $\tilde\chi\colon S \times \R \to \R$
 such that
 \begin{enumerate}[(1)]\itemsep=0pt
 \item for each $s \in S$, $t \mapsto \tilde\chi(s, t)$ is nondecreasing,
 \item there exists a neighborhood of $\{ t = b \} \subset S \times
 \mathbb R$ where $\chi$ equals the identity,
 \item there exists a neighborhood of $\{ t = a\}$, where $\tilde\chi(s, t) = a(s)$.
 \end{enumerate}
 We define \emph{right} modification functions by swapping the roles
 of $a$ and $b$, and \emph{two-sided} modification functions by
 requiring fiberwise constancy near both ends.
\end{Definition}

\begin{Lemma} \label{LemmaExistenceOfModFunct}
 For any $a > b\colon S \times \mathbb R \to \mathbb R$, there exist
 left, right and two-sided modification functions on $S
 \times \mathbb R$.
\end{Lemma}

\begin{proof}
 We can certainly find a nonnegative smooth function $f\colon S \times
 \mathbb R \to \mathbb R$ with support in the open set $\{(s, t) \mid a(s) < t < b(s) \}
 \subset S \times \mathbb R$ and such that $f|_{\{ s \} \times
 \mathbb R}$ is not identically zero for any $s \in S$. Setting
 \begin{equation*}
 g(s, t)
 = \frac
 {\int_{-\infty}^t f(s, u)\, \mathrm{d} u}
 {\int_{-\infty}^{+\infty} f(s, u)\, \mathrm{d} u}
 \end{equation*}
 gives us a smooth, fiberwise nondecreasing function which is
 identically $0$ in a neighborhood of $\{ t \leq a \}$ and
 identically $1$ in a neighborhood of $\{ t \geq b \}$. Then
 \begin{equation*}
 \chi_{LR}(s, t) = (s, a(s) + (b(s) - a(s))g(s, t))
 \end{equation*}
 is a two-sided modification function. Next, write
 \begin{equation*}
 \kappa(s) =
 \frac1{\int_{-\infty}^{+\infty} f(s,u)\,\mathrm{d} u}
 \bigg( b(s) - a(s) - \int_{-\infty}^{b(s)} g(s, u)\, \mathrm{d}u \bigg).
 \end{equation*}
 Due to the bounds on $g$, this is nonnegative. Finally, set
 \begin{equation*}
 h_s(t) = a(s) + \int_{-\infty}^t (\kappa(s)f(s,u) + g(s, u))\, \mathrm{d}u.
 \end{equation*}
 Clearly, $h_s$ is nondecreasing, constant near $t = a(s)$, and
 affine with slope $1$ near $t = b(s)$. Moreover, by the choice of
 $\kappa(s)$, we have $h_s(a(s)) = a(s)$ and $h_s(b(s)) = b(s)$. Thus,
 \begin{equation*}
 \chi_L(s, t) = (s, h_s(t))
 \end{equation*}\looseness=1
 defines a left modification function. Right modification functions
 can be constructed similarly.
\end{proof}

\begin{Lemma}
 \label{lem:independence-of-modification}
 Let $\gamma\colon S \times\mathbb R \to M$ be an $S$-family of
 paths, fix smooth functions $a < b\colon S \to \mathbb R$, and let $\chi_0$,
 $\chi_1$ be left modification functions. Then $Z(\gamma \circ
 \chi_0; a, b) = Z(\gamma \circ \chi_1; a, b)$. The same holds for
 right modifications and two-sided modifications.
\end{Lemma}

Note that all modification functions agree in a neighborhood of the
boundaries, so the domain and codomain of the maps $Z(\gamma \circ
\chi, a, b)$ are canonically identified, as in Remark~\ref{RemarkSimplification}.

\begin{proof}
 The case of two-sided modifications follows immediately from
 Lemma~\ref{LemmaIndependenceOfChi}, since $\Vect$ is full in
 $\mathcal V$, so that the maps $Z(\gamma \circ \chi_i, a, b)$ are
 identified with maps of vector bundles. Suppose now that
 $\chi_0$, $\chi_1$ are left modification functions. Let $U$ be a
 neighborhood of $\{ t = b \}$, where~$\chi_0$ and $\chi_1$ both agree
 with the identity. Then we can find a function $a'$ with $a > a' >
 b$ and $\{ t = a' \} \subset U$, as well as a map $F\colon S \times
 \mathbb R \to S \times \mathbb R$ over $S$, fiberwise nondecreasing,
 which is the identity away from $U$ and near $\{ b = 0 \}$, and
 fiberwise constant on $\{ t = a' \}$. Thus, using Lemma~\ref{LemmaSittingInstant}, we get
 \begin{equation*}
 Z(\gamma \circ \chi_i, a, b)
 = Z(\gamma \circ \chi_i \circ F, a, b)
 = Z(\gamma \circ F, a', b)
 \circ Z(\gamma \circ \chi_i \circ F, a, a')
 \end{equation*}
 for both $i=0, 1$. Finally, notice that
 \begin{equation*}
 Z(\gamma \circ \chi_0 \circ F, a, a')
 = Z(\gamma \circ \chi_1 \circ F, a, a')
 \end{equation*}
 by the two-sided version of this lemma. This finishes the proof of
 the left-sided case; the right-sided case is similar.
\end{proof}

\begin{Lemma} \label{LemmaDecomposition}
 The field theory $Z$ determines a stack map $Z'\colon 1\Bord(M)_0
 \to \mathcal V$ and a distinguished decomposition
 \begin{equation*}
 Z(\gamma; a) = V_{\gamma(a)} \oplus Z'(\gamma; a)
 \end{equation*}
 for each $S$-family of objects $(\gamma, a)$, where $V_{\gamma(a)}$
 is the pullback of the vector bundle $V \to M$ via the map
 $\gamma(a)\colon S \to M$.
\end{Lemma}

Above, we used the suggestive notation $\gamma(a)$ for the map $s \mapsto \gamma(s, a(s))$.
Note that $V_{\gamma(a)} \cong Z(\gamma(a), a)$ is
what $Z$ assigns to the family of constant germs with the same value as
$\gamma$ on $\{ t = a \}$.

\begin{proof}
 Choose smooth functions $a^\prime, a^{\prime\prime} \colon S \to \R$ with
 $a^\prime < a < a^{\prime\prime}$ and smooth, fiberwise
 nondecreasing functions $\chi_1, \chi_2\colon S \times \R \to S \times
 \R$ such that $\chi_1$ is fiberwise constant near $a^\prime$ and
 $a^{\prime\prime}$ and equal to the identity near $a$, while
 $\chi_2$ is fiberwise constant near $a^\prime$, $a$ and
 $a^{\prime\prime}$. In other words, $\chi_1$ is a left modification
 function for $(\gamma; a^{\prime}, a)$ and a right modification
 function for $(\gamma; a, a^{\prime\prime})$, while~$\chi_2$ is a
 two-sided modification function for both $(\gamma; a^{\prime}, a)$
 and $(\gamma; a, a^{\prime\prime})$. In particular, both~$\chi_1$
 and~$\chi_2$ are two-sided modification functions for $(\gamma;
 a^{\prime}, a^{\prime\prime})$. The existence of such functions is
 clear from Lemma~\ref{LemmaExistenceOfModFunct} and the fact that
 modification functions can be ``patched together'' in the obvious
 way for neighboring intervals. We obtain the following diagram of
 bordisms:
 
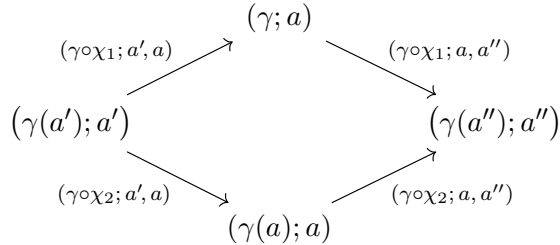
\begin{equation*}
 \begin{tikzcd}
 &
 (\gamma; a)
 \ar[rd, "{(\gamma\circ\chi_1;\,a,\,a^{\prime\prime})}"]
 & \\
 \bigl(\gamma(a^{\prime}); a^{\prime}\bigr)
 \ar[ru, "{(\gamma\circ\chi_1;\,a^{\prime}\!,\,a)}"]
 \ar[rd, "{(\gamma\circ\chi_2;\,a^{\prime}\!,\,a)}"']
 & &
 \bigl(\gamma(a^{\prime\prime}); a^{\prime\prime}\bigr)
 \\ &
 (\gamma(a); a)
 \ar[ru, "{(\gamma\circ\chi_2;\,a,\,a^{\prime\prime})}"']
 \end{tikzcd}
 \end{equation*}
 It certainly does not commute in the sense of there being an
 isomorphism between the two compositions: the bottom one has a
 sitting instant at $a$, while in general the top one does
 not. However, since the composition of the upper two bordisms is
 $(\gamma \circ \chi_1; a^\prime, a^{\prime\prime})$ while the lower
 composition is $(\gamma \circ \chi_2; a^\prime, a^{\prime\prime})$,
 Lemma~\ref{lem:independence-of-modification} implies that the
 diagram obtained after applying $Z$
 \begin{equation}
 \label{DiagramZip}
 \begin{tikzcd}
 &
 Z(\gamma; a) \ar[dd, dashed, shift left=2, "p_a"]
 \ar[rd, "{Z(\gamma\circ\chi_1;\,a,\,a^{\prime\prime})}"]
 & \\
 Z\bigl(\gamma(a^{\prime}); a^{\prime}\bigr)
 \ar[ru, "{Z(\gamma\circ\chi_1;\,a^{\prime}\!,\,a)}"]
 \ar[rd, "{Z(\gamma\circ\chi_2;\,a^{\prime}\!,\,a)}"']
 & &
 Z\bigl(\gamma(a^{\prime\prime}); a^{\prime\prime}\bigr)
 \\ &
 Z(\gamma(a); a) \ar[uu, dashed, shift left=2, "i_a"]
 \ar[ru, "{Z(\gamma\circ\chi_2;\,a,\,a^{\prime\prime})}"']
 \end{tikzcd}
 \end{equation}
 commutes. We now define maps $i_a$ and $p_a$ as indicated in the
 diagram by requiring that they make the left, respectively right,
 triangle commute. These maps exist and are well-defined as the
 bottom maps are invertible, by Lemma~\ref{LemmaIso}. By
 commutativity of the diagram, we have $p_a i_a = \mathrm{id}$.

 We now set $Z'(\gamma; a) = \Ker p_a$. By standard arguments, $i_a$
 and the inclusion map $Z'(\gamma; a) \to Z(\gamma; a)$ exhibit
 $Z(\gamma; a)$ as the direct sum
 \begin{equation*}
 Z(\gamma; a) \cong V_{\gamma(a)} \oplus Z'(\gamma; a).
 \end{equation*}
 This identification does not depend on any of the choices by
 Lemma~\ref{lem:independence-of-modification}.
\end{proof}

\begin{Lemma}
 For any nowhere thin $S$-family of bordisms $(\gamma; a, b)$, by
 which we mean that \mbox{$a(s) < b(s)$} for all $s \in S$, the map
 $Z(\gamma;a, b)$ has matrix representation
 \begin{equation*}
 Z(\gamma; a, b) =
 \begin{pmatrix}
 P(\gamma; a, b) & 0 \\
 0 & 0
 \end{pmatrix}
 \end{equation*}
 with respect to the direct sum decompositions of $Z(\gamma; a)$ and
 $Z(\gamma; b)$ from Lemma~$\ref{LemmaDecomposition}$.
\end{Lemma}

\begin{proof}
 Choose $a' < a < a'' < b' < b < b''$ (as functions $S \to \mathbb
 R$) and consider the commutative diagram
 \begin{equation*}
 \begin{tikzcd}
 &
 Z(\gamma; a)
 \ar[rrr, "{Z(\gamma;\,a,\,b)}"]
 \ar[dd, "p_a", shift left=2]
 \ar[dr]
 &&&
 Z(\gamma; b)
 \ar[dd, "p_b", shift left=2]
 \ar[dr]
 \\
 V_{\gamma(a')} \ar[dr] \ar[ur]
 &
 &
 V_{\gamma(a'')} \ar[r]
 &
 V_{\gamma(b')} \ar[dr] \ar[ur]
 &
 &
 V_{\gamma(b'')} \\
 &
 V_{\gamma(a)}
 \ar[uu, "i_a", shift left=2]
 \ar[ur]
 &&&
 V_{\gamma(b)}
 \ar[uu, "i_b", shift left=2]
 \ar[ur]
 &
 \end{tikzcd}
 \end{equation*}
 obtained by patching together the corresponding diagrams~\eqref{DiagramZip} in the obvious way. In~par\-ti\-cu\-lar, all bottom
 maps are parallel translations along segments of $\gamma$.

 To prove that the second columns of the matrix representing
 $Z(\gamma; a, b)$ is zero, we need to show that this map, restricted
 to the kernel of $p_a$, is zero. This follows from the diagram,
 since it shows that $Z(\gamma; a, b)$ factors through $p_a$.

 A diagram chase shows that the square composed by $Z(\gamma; a, b)$,
 $i_a$, $i_b$, and the parallel translation $P\colon V_{\gamma(a)}
 \to V_{\gamma(b)}$ commutes. This implies that the first column of
 $Z(\gamma; a, b)$ is as claimed.
\end{proof}

The proof of Theorem~\ref{thm:sheafy} concludes with the next lemma.

\begin{Lemma}
 For any $S$-family $(\gamma; a)$, $Z'(\gamma; a) = 0 \in \mathcal
 V(S)$.
\end{Lemma}

\begin{proof}
 Let $\pr\colon S \times \mathbb R \to S$ be the projection and
 let $\delta\colon S \times \R \rightarrow \R$ be the coordinate
 function on the $\R$-factor. Consider the $(S \times \R)$-family of
 bordisms $B = \big((\mathrm{pr} \times \mathrm{id})^*{\gamma};
 \mathrm{pr}^*a- \delta^2, \mathrm{pr}^*a\big)$. Then
 \begin{equation*}
 Z'(B)\colon\ Z'\big(\pr^*\gamma, a - \delta^2\big)
 \to Z'(\pr^*\gamma, a) \cong \pr^*(Z'(\gamma, a))
 \end{equation*}
 vanishes away from $S \times 0$ by the previous lemma, and therefore
 is identically $0$ by the separation axiom. But this implies that
 \begin{equation*}
 \mathrm{id}_{Z'(\gamma; a)} = Z'(B)\vert_{S \times 0} = 0,
 \end{equation*}
 so $Z'(\gamma; a)$ is the zero object of $\mathcal V(S)$.
\end{proof}

\subsection[Examples of admissible stacks of C infty-modules]
{Examples of admissible stacks of $\boldsymbol{C^\infty}$-modules}
\label{sec:admissible-stacks}

In this section, we give two examples of suitable target stacks
$\mathcal V$ for TFTs:
\begin{enumerate}[1.]\itemsep=0pt

\item The stack $\mathcal V_\alg$ of sheaves of
 $C^\infty$-modules with the algebraic tensor product.
\item The stack $\mathcal V_\vN$ of sheaves of complete bornological
 $C^\infty$-modules with the completed borno\-lo\-gical tensor product,
 constructed in Appendix~\ref{sec:construction-V-vN} and briefly
 reviewed below. Here, $C^\infty(S)$ is endowed with its von
 Neumann bornology (see Example~\ref{ex:von-Neumann}), hence the
 notation.
\end{enumerate}

Recall that a \emph{bornology} on a vector space is a collection of
subsets deemed to be bounded and satisfying appropriate axioms.
Bornological vector spaces and bounded linear maps form a category
$\Born$. There is an appropriate notion of completeness, and thus a
full subcategory $\CBorn \subset \Born$ of \emph{complete}
bornological vector spaces. We are interested in $\CBorn$ for its
pleasant categorical properties. In particular, there is a completed
tensor product $\otimeshat$ which makes $\CBorn$ into a closed
symmetric monoidal category. As is well known, none of the many
possible tensor products on the various categories of topological
vector spaces have this property. See Appendices~\ref{sec:BVSs} to~\ref{sec:smooth-functions} for a quick overview, or Meyer~\cite[Chapter~1]{MR2337277} for a comprehensive introduction to the
theory of bornological vector spaces.

The goal of this subsection, which will follow immediately by
combining Propositions~\ref{prop:V-alg-C2} and~\ref{prop:V-vN-C3} and
Corollaries~\ref{cor:V-vN-C2} and~\ref{cor:V-alg-C3} below, is to
prove the following.

\begin{Theorem}
% \label{thm:admissibility}
 $\mathcal V_\alg$ and $\mathcal V_\vN$ are admissible stacks of
 $C^\infty$-modules.
\end{Theorem}

\begin{Remark}
 The purely algebraic $\mathcal V_\alg$ is familiar and attractive in
 its simplicity, but it has the drawback of not admitting a
 reasonable ``sheaf of sections'' functor from any category of
 infinite-dimensional vector bundles. Moreover, our proof that it
 satisfies the separation axiom requires the corresponding fact for
 $\mathcal V_\vN$.
\end{Remark}

Before proceeding, let us briefly review the definitions of $\mathcal
V_\alg$ and $\mathcal V_\vN$. It will be convenient to use the
sheaf-of-categories approach to stacks.

We define $\mathcal V_\alg\colon \Man^{\mathrm{op}} \to
\mathrm{Cat}^\otimes$ to be the functor assigning to each manifold $S$
the symmetric monoidal category of sheaves of $C^\infty_S$-modules and
$C^\infty_S$-linear maps, with the algebraic tensor product. To a map
$f\colon T \to S$ is associated the pullback functor $f^*$ defined by
\begin{equation}
 \label{eq:pullback-sheaf}
 f^*V = C^\infty_T \otimes_{f^{-1}(C^\infty_S)} f^{-1}(V).
\end{equation}
That this indeed defines a symmetric monoidal stack on the site of
manifolds and satisfies axiom~(C1) is a standard fact.

Similar constructions can be carried out in the setting of (complete)
bornological vector spaces. This is outlined in
Appendices~\ref{sec:bornological-sheaves} and~\ref{sec:construction-V-vN} and is mostly a formality, except for the
fact that sheafification of presheaves with values in $\CBorn$
requires additional care. In a nutshell, we promote the sheaf of
smooth functions on a manifold $S$ to a sheaf of complete bornological
algebras, denoted by $C^\infty_{\vN,S}$, and then define $\mathcal
V_\vN(S)$ to be the category of sheaves of complete bornological
$C^\infty_{\vN, S}$-modules. We have a completed tensor product of
bornological sheaves, and the pullback functors $f^*\colon \mathcal
V_\vN(S) \to \mathcal V_\vN(T)$ are defined similarly to~\eqref{eq:pullback-sheaf}, using this completed tensor product. Thus
$\mathcal V_\vN$ is a symmetric monoidal stack satisfying axiom (C1).

Next, we construct a functor $i\colon \mathcal V_\alg \to \mathcal V_\vN$.
Let $S$ be a manifold and $V \in \mathcal{V}_{\alg}(S)$.
Given $U \subset S$ open, we say that a subset $B \subset V(U)$ is $\vN$-bounded
if each point $x \in U$ has an open neighborhood $W \subset U$ such that
\begin{equation*}
 B|_W \subset C_1 b_1 + \dots + C_n b_n
\end{equation*}
for $\vN$-bounded subsets $C_1, \dots, C_n \subset C^\infty(U)$ and $b_1, \dots, b_n \in V(W)$.
It is routine to check that this indeed satisfies the axioms of a (vector) bornology on $V(U)$.
We denote the resulting presheaf of bornological vector spaces by $V_{\vN}$.

\begin{Example}
 If $V = \Gamma_E$ is the sheaf of sections of a finite-dimensional
 vector bundle over $E \to S$, then the bornology constructed above
 coincides with the bornology associated to its Fr\'echet topology,
 where a set of sections is bounded if all derivatives of a given order are uniformly bounded
 on compact sets.
\end{Example}

\begin{Proposition}
For any sheaf $V$ of $C^\infty_S$-modules, $V_\vN$ is a sheaf of complete bornological $C^\infty_{\vN, S}$-modules.
\end{Proposition}

\begin{proof}
It follows directly from the definition of the $\vN$-bornology that for each open $U \subseteq S$, the bornological vector space $V_{\vN}(U)$ is a bornological $C^\infty_{\vN}(U)$-module (i.e., the module action map is bounded), and that restriction maps are bounded.
We have to show that $V_{\vN}$ satisfies descent.
Let $U \subset S$ be open and let $\{U_i\}_{i \in I}$ be an open cover of $U$.
Since $V$ is an (algebraic) sheaf, we know that the map
\begin{equation*}
 V_\vN(U) \stackrel{r}{\longrightarrow}
 \Ker \bigg(\prod\limits_{i \in I} V_\vN(U_i)
 \longrightarrow \prod\limits_{i, j \in I} V_\vN(U_i \cap U_j)\bigg)
\end{equation*}
is a vector space isomorphism. We have to show that it is an
isomorphism of bornological vector spaces. Since we already know that
restriction maps are bounded, we have to show that the inverse image under $r$ of bounded sets in
$\prod_{i \in I} V_\vN(U_i)$ is bounded. It suffices to show this for
sets of the form $\prod_{i \in I} B_i$ with $B_i \subset V_{\vN}(U_i)$
bounded, as these generate the product bornology. Here we have
\begin{equation*}
 r^{-1}\bigg(\prod\limits_{i \in I} B_i \bigg)
 = \{ v \in V(U) \mid v|_{U_i} \in B_i \}.
\end{equation*}
That this is $\vN$-bounded follows again directly from the definition
of the bornology, as boundedness can be checked locally. This shows
that $V_{\vN}$ is a sheaf of bornological $C^\infty_{\vN, S}$-modules.

We now show that $V_{\vN}(U)$ is complete for each open $U \subseteq S$.
To this end, we have to show that each bounded subset of $V_{\vN}(U)$ is contained in a complete bounded disc.
Here a~bounded disc is complete if the normed space $V(U)_B = \mathrm{span}(B) \subset V(U)$ is complete, where the norm is the Minkowski functional defined by $B$.
Let $\{W_i\}_{i \in I}$ be an open cover of $U$, let $b_1^i, \dots, b_{n_i}^i \in V(W_i)$ and $C_1^i, \dots, C_{n_i}^i \subset C^\infty(W_i)$ be complete $\vN$-bounded discs.
Then the sets $B_i := C_1^i b_1^i + \dots + C_{n_i}^i b_{n_i}^i \subset V(W_i)$ and
\begin{equation*}
 B:= \big\{v \in V(U) \mid v|_{W_i} \in B_i \big\}
\end{equation*}
are $\vN$-bounded discs.
Clearly, any $\vN$-bounded set in $V(U)$ is contained in such a disc, for some choice of $\{W_i\}$, $\{b_i\}$ and $\{C_k^i\}$, hence it suffices to show that $V(U)_B$ is a Banach space.

We first show that each of the normed spaces $V(W_i)_{B_i}$ are complete.
To this end, consider the linear map
\begin{equation*}
 \pi\colon\quad X \longrightarrow V(W_i)_{B_i}, \qquad X := \prod_{k=1}^{n_i} C^\infty(W_i)_{C_k^i},
\end{equation*}
sending $(f_1, \dots, f_{n_i})$ to $f_1 b_1 + \dots + f_{n_i} b_{n_i}$.
It is surjective by definition of $B_i$, and it is bounded since any bounded set $C \subset X$ is contained in $\lambda(C_1^i \times \cdots \times C_{n_i}^i)$ for some $\lambda > 0$, hence $\pi(C) \subset \lambda \pi(C_1^i \times \cdots \times C_{n_i}^i) = \lambda B$.
Notice that $X$ is a Banach space as the discs $C_k^i$ are complete.
Because $\pi$ is bounded, its kernel is closed, so that the quotient $X / \Ker(\pi)$ is a Banach space.
The quotient map $\bar\pi\colon X / \Ker(\pi) \to V(W_i)$ is a bounded vector space isomorphism.
To see that it is a homeomorphism, it remains to show that the inverse is bounded.
To this end, let $v \in V(W_i)_{B_i}$.
Let $\lambda > 0$ be such that $v \in \lambda B_i$.
Because $\pi\big(C_1^i \times \cdots \times C_{n_i}^i\big) = B_i$, we can choose a preimage $v^\prime \in \lambda (C_1^i \times \cdots \times C_{n_i}^i) \subset X$.
Then
\begin{align*}
 \big\|\bar\pi^{-1}(v)\big\|_{X/\Ker(\pi)}
 &= \inf_{w \in \Ker(\pi)} \|v^\prime + w\|_X
\leq \|v^\prime\|_X
 \\
 &= \inf \big\{ \mu > 0
 \mid v^\prime \in \mu \big(C_1^i \times \cdots \times C_{n_i}^i\big)\big\}
 \leq \lambda.
 \end{align*}
This holds for any $\lambda$ such that $v \in \lambda B_i$. Therefore, we get
\begin{equation*}
 \big\|\bar\pi^{-1}(v)\big\|_{X/\Ker(\pi)} \leq \inf\{ \lambda >0 \mid v \in \lambda B_i\} = \|v\|_{B_i}.
\end{equation*}
This shows that $\bar\pi^{-1}$ is bounded, hence $\bar\pi\colon X /\Ker(\pi) \to V(W_i)_{B_i}$ is a homeomorphism.
Since $X/\Ker(\pi)$ is complete, so is $V(W_i)_{B_i}$.

Now consider the map
\begin{equation*}
 V(U)_B \stackrel{r}{\longrightarrow} \prod_{i \in I} V(W_i)_{B_i},
 \qquad
 v \mapsto (v|_{W_i})_{i \in I},
\end{equation*}
where the right hand side is a product of Banach spaces,
endowed with the norm obtained by taking the supremum over all the norms of the $V(W_i)_{B_i}$.
In other words, the unit ball of the product is $\prod_{i \in I} B_i$.
Observe that
this map is well-defined since if $v \in \lambda B$, then $v|_{W_i}
\in \lambda B_i$ and~$(v|_{W_i}) \in \lambda \prod_{i \in I}
B_i$. The map is bornologically proper, since $r^{-1}\big(\prod_{i \in I}
B_i\big) = B$. On the other hand, its image is the kernel of the bounded
map
\begin{equation*}
 \prod_{i \in I} V(W_i)_{B_i}
 \longrightarrow \prod_{i, j \in I} V(W_i \cap W_j)_{B_{ij}},
 \qquad
 (v_i)_{i \in I}
 \longmapsto
 (v_i|_{W_i \cap W_j} - v_j|_{W_i \cap W_j})_{i, j \in I},
\end{equation*}
hence closed.
Here $B_{ij} = B_i|_{W_i \cap W_j} + B_j|_{W_i \cap W_j}$.
 This shows that $V(U)_B$ is isomorphic to a~closed
subspace of a Banach space, hence complete.
\end{proof}

To promote the assignment $V \mapsto V_\vN$ to a functor, we have to
deal with morphisms. Here we have the following lemma.

\begin{Lemma}
 Let $V$ and $V^\prime$ be $($algebraic$)$ $C^\infty_S$-modules and let
 $\Phi\colon V \to V^\prime$ be a morphism of sheaves. Then, for
 each $U \subseteq S$, $\Phi_U\colon V(U) \to V^\prime(U)$ is
 $\vN$-bounded.
\end{Lemma}

\begin{proof}
Let $B \subset V(U)$ be $\vN$-bounded.
We have to show that also $\Phi_U(B)$ is $\vN$-bounded.
To~this end let $x \in U$ and choose $W$ such that $B|_W \subseteq C_1 b_1 + \cdots+ C_n b_n$ for elements $b_1, \dots, b_n \allowbreak\in V(W)$ and $\vN$-bounded subsets $C_1, \dots, C_n\subset C^\infty(W)$.
Then since $\Phi$ is a morphism of sheaves and $C^\infty_S$-linear,
\begin{gather*}
\Phi_U(B) = \Phi_W(B|_W) \subseteq \Phi(C_1b_1 + \cdots + C_n b_n) = C_1 \Phi(b_1) + \cdots + C_n \Phi(b_n).
\end{gather*}
Hence $\Phi_U(B)$ is $\vN$-bounded.
\end{proof}

We therefore obtain a functor $i\colon \mathcal{V}_{\alg} \to \mathcal{V}_{\vN}$ sending $V$ to $V_\vN$ and which sends morphisms $V \to V^\prime$ of sheaves to the corresponding morphism $V_\vN \to V^\prime_\vN$, which exists by the previous lemma.
Since conversely, each morphism of sheaves $V_\vN \to V^\prime_\vN$ gives a morphism of the underlying algebraic sheaves of $C^\infty_S$-modules, we obtain the following corollary.

\begin{Corollary}
The functor $i\colon \mathcal{V}_{\alg} \to \mathcal{V}_{\vN}$ described above is fully faithful.
\end{Corollary}

\begin{Proposition}
 \label{prop:V-alg-C2}
 $\mathcal V_\alg$ satisfies axiom {\rm (C2)}.
\end{Proposition}

\begin{proof}
 For each $S$, there is a ``sheaf of sections'' functor $\Gamma\colon
 \Vect(S) \to \mathcal V_\alg(S)$. To see that these fit together to
 a map of stacks, we need to show that for each vector bundle $V \in
 \Vect(S)$ and each map $f\colon T \to S$, the natural maps
 $f^*(\Gamma_V) \to \Gamma_{f^*V}$
 are isomorphisms. On the stalk at~$x \in T$, this yields a map
 \begin{equation*}
 f^*(\Gamma_V)_x
 \cong C^\infty_{T, x} \otimes_{C^\infty_{S, f(x)}} \Gamma_{V,f(x)}
 \to \Gamma_{f^*V, x}
 \end{equation*}
 compatible with the identifications $\Gamma_{V, f(x)} \cong
 C^\infty_{S, f(x)} \otimes V_{f(x)}$ and $\Gamma_{f^*V, x} \cong
 C^\infty_{T, x} \otimes V_{f(x)}$. Here, $V_{f(x)}$ is the fiber of
 the vector bundle $V$; all other subscripts denote stalks at the
 given point. Thus, the above is an isomorphism on stalks and hence
 an isomorphism of sheaves. That $\Gamma$ has all dualizable objects
 as essential image is the content of the Serre--Swan theorem.
\end{proof}

\begin{Corollary}
 \label{cor:V-vN-C2}
 $\mathcal V_\vN$ satisfies axiom {\rm (C2)}.
\end{Corollary}

\begin{proof}
 Composition with $i$ gives the map $\Gamma\colon \Vect \to \mathcal
 V_\vN$. Since every dualizable $C^\infty_{\vN, S}$-module has
 finite rank, any dualizable object in $\mathcal V_\vN(S)$ is already
 dualizable in $\mathcal V_\alg(S)$.
\end{proof}

\begin{Proposition}
 \label{prop:V-vN-C3}
 $\mathcal V_\vN$ satisfies the separation axiom {\rm (C3)}.
\end{Proposition}

\begin{proof}
 Fix a $C^\infty_{\vN, S}$-module $W$, and denote by $V$ its pullback
 to $S \times \mathbb R$. It suffices to show that if a section
 $\sigma \in V(U)$ is such that $\sigma\vert_{U \setminus (S \times
 0)}$ vanishes, then in fact $\sigma = 0$. By definition, $V$ is
 the sheafification of the presheaf
 \begin{equation*}
 V'\colon\ U \mapsto
 C^\infty_\vN(U) \otimeshat_{\left(\pr^{-1}C^\infty_{\vN, S}\right)(U)} \big(\pr^{-1}W\big)(U),
 \end{equation*}
 so we know that $\sigma$ is determined by a coherent collection of
 sections $\sigma_i \in V'(S_i \times T_i)$ for some collection of
 opens $S_i \subset S$, $T_i \subset \mathbb R$, since products $S_i
 \times T_i$ form a basis for the topology of~$U \subset S \times
 \mathbb R$. Our question is reduced to showing that if $\sigma_i$
 vanishes on $S_i \times (T_i \setminus 0)$, then it vanishes on $S_i
 \times T_i$. Now,
 \begin{equation*}
 \begin{aligned}
 V'(S_i \times T_i)
 &
 \cong C^\infty_\vN(T_i \times S_i) \otimeshat_{C^\infty_\vN(S_i)} V(S_i)
 \\ &
 \cong (C^\infty(T_i) \otimeshat C^\infty(S_i))
 \otimeshat_{C^\infty(S_i)} V(S_i)
 \\ &
 \cong C^\infty_\vN(T_i) \otimeshat V(S_i)
 \\ &
 \cong C^\infty_\vN(T_i; V(S_i)).
 \end{aligned}
 \end{equation*}
 The first identification is by definition, the second follows from
 Proposition~\ref{prop:smooth-functions-on-product}, the third from
 the fact that $C^\infty(T_i) \otimeshat \mathemdash$ commutes with
 the colimit defining the tensor product over $C^\infty(S_i)$, and
 the fourth from Proposition~\ref{prop:smooth-functions-in-V}. Thus,
 $\sigma_i$ gets identified with a smooth function $f\colon T_i \to
 V(S_i)$ which, by definition, is a smooth function with values in
 some Banach space $V_B \subset V(S_i)$ (cf. Appendix~\ref{sec:smooth-functions}). Hence it vanishes
 identically if it vanishes away from~$0$.
\end{proof}

\begin{Corollary}
 \label{cor:V-alg-C3}
 $\mathcal V_\alg$ satisfies the separation axiom {\rm (C3)}.
\end{Corollary}

\begin{Remark}
 One might wish for a more elementary proof of the corollary, and we
 would like to note that the bornological tensor product is handy in
 this case as well. Arguing as in the proof of
 Proposition~\ref{prop:V-vN-C3}, one reduces property (C3) for
 $\mathcal V_\alg$ to the claim that if $V$ is a~$C^\infty(S)$-mo\-dule~and
 \begin{equation*}
 \sigma \in C^\infty(\mathbb R \times S) \otimes_{C^\infty(S)} V
 \end{equation*}
 maps to $0$ in $C^\infty((\mathbb R \setminus 0) \times S)
 \otimes_{C^\infty(S)} V$, then $\sigma = 0$. This is not
 immediately clear. Notice, however, that if we give $V$ the fine
 bornology, Proposition~\ref{prop:fine-otimes} allows us to regard
 $\sigma$ as an element of
 \begin{equation*}
 C^\infty_\vN(\mathbb R \times S) \otimeshat_{C^\infty(S)} V
 \cong C^\infty_\vN(\mathbb R) \otimeshat C^\infty_\vN(S)
 \otimeshat_{C^\infty(S)} V,
 \end{equation*}
 and therefore as a smooth function from $\mathbb R$ to some Banach
 space. The claim follows.
\end{Remark}

\appendix

\section{Bornological sheaves}
\label{sec:bornological}

In this appendix, we construct the symmetric monoidal stack $\mathcal
V_\vN$ of sheaves of complete bornological $C^\infty$-modules on the
site $\Man$ of smooth manifolds. Formally, this construction is very
similar to that of its well-known algebraic counterpart $\mathcal
V_\alg$, and our main goal here is to highlight the differences. We
start recalling several basic facts about bornological vector spaces,
providing proofs for the bits that are not easily located in the
literature. For this, our main reference is Meyer~\cite[Chapter~1]{MR2337277}. The basic facts about sheaves of
(complete) bornological vector spaces and modules are quoted from
Houzel~\cite{MR0393552}.

\subsection{Bornological vector spaces}
\label{sec:BVSs}

A (convex) \emph{bornology} on a vector space $V$ over $\mathbb{K} = \R$ or $\C$ is a collection
$\mathfrak S$ of subsets of $V$ deemed to be bounded. These have to
satisfy appropriate axioms which we will not repeat here. For any
collection $\mathfrak S^\prime$ of subsets of $V$, there is a smallest
bornology containing $\mathfrak S^\prime$, the bornology
\emph{generated by} $\mathfrak S^\prime$. A linear map $f\colon V \to
W$ is \emph{bounded} if it sends bounded subsets to bounded subsets.

A \emph{disk} in a vector space $V$ is a convex, balanced subset $B
\subset V$. We denote by $V_B = \mathbb R \cdot B \subset V$ the
subspace spanned by $B$. The closed ball of a seminorm on $V_B$ is an
absorbing disk, and conversely a disk $B \subset V$ determines a
unique seminorm on $V_B$. We say that $B$ is \emph{norming}
respectively \emph{complete} if $V_B$ is a normed respectively a
Banach space. A bornological vector space is \emph{separated} if
every bounded disk is norming, and \emph{complete} if every bounded
subset is contained in a complete bounded disk.

\begin{Example}
 The smallest possible bornology on a vector space $V$ is the one
 generated by~convex hulls of finite subsets. This is called the
 \emph{fine bornology}. It is always complete. We~denote by
 $\Fin(V)$ this bornological vector space. With this bornology, we
 have the relation
 \begin{math}
 V \cong \colim V_\alpha,
 \end{math}
 where $V_\alpha$ runs through all finite-dimensional subspaces of
 $V$, endowed with their fine bornology.
\end{Example}

\begin{Example}
 \label{ex:von-Neumann}
 Let $V$ be a locally convex topological vector space.
 Traditionally, a subset of~$B \subset V$ is called bounded if it is
 absorbed by any neighborhood of the origin. This defines a
 bornology on $V$, called the \emph{von Neumann bornology}. It is
 complete if $V$ is complete as a~topological vector space.
 We denote by $\vN(V)$ this bornological vector space.
\end{Example}

We are interested in bornological vector spaces for their convenient
categorical properties. We denote by
\begin{equation*}
 \CBorn \subset \Born
\end{equation*}
the category of bornological vector spaces and its full subcategory of
complete bornological vector spaces.\footnote{Meyer~\cite{MR2337277}
 denotes $\Born$ by $\mathfrak{Born}^{1/2}$, reserving the notation
 $\mathfrak{Born}$ for the full subcategory of separated borno\-logical
 spaces.}

\begin{Theorem}[{\cite[Proposition~1.126]{MR2337277}}]
% \label{thm:Born-additive-bicomplete}
 $\Born$ and $\CBorn$ are additive categories and admit all limits
 and colimits.
\end{Theorem}

The forgetful functor $\Born \to \Vect$ preserves limits and colimits;
in other words, limits and colimits in $\Born$ are obtained by putting
appropriate bornologies on the corresponding constructions with plain
vector spaces. The bornology on a direct sum $\bigoplus_{i \in I}
V_i$ is generated by images of bounded subsets of each $V_i$ via the
standard inclusions; a subset $B \subset \prod_{i \in I} V_i$ is
bounded if and only if each projection $p_i(B) \subset V_i$ is
bounded; kernels and cokernels are endowed with the subspace and
quotient bornology, respectively.

The inclusion $i\colon \CBorn \to \Born$ has a left adjoint, the
\emph{completion} functor $\Born \to \CBorn$, $V \mapsto V^c$. This
means that there are natural isomorphisms
\begin{equation*}
 \Hom(V, W) \cong \Hom\big(V^c, W\big)
\end{equation*}
for $W$ complete. In particular, limits in $\CBorn$ are calculated as
limits in $\Born$. This is also true for direct sums. However,
completion does not preserve colimits in general: the cokernel of a
map $f\colon V \to W$ in $\CBorn$ is obtained by modding out the
bornological closure of the image, $\Coker f = W / \overline{f(V)}$~\cite[Section~1.3.3]{MR2337277}. We may write $\sep\colim$,
$\sep\Coker$, etc.\ to emphasize that a construction is taken inside
$\CBorn$.

The space $\Hom(V, W)$ of all bounded linear maps has a natural
bornology generated by the uniformly bounded subsets, i.e., those $L
\subset \Hom(V, W)$ such that
\begin{equation*}
 L(B) = \{ f(x) \mid f \in L,\ x \in B \} \subset W
\end{equation*}
is bounded for all bounded $B \subset V$. This is complete if $W$ is.

We equip the algebraic tensor product $V_1 \otimes V_2$ with the
bornology generated by subsets of~the form $B_1 \otimes B_2$, with
$B_1 \subset V_1$, $B_2 \subset V_2$ bounded. We define the
\emph{completed bornological tensor product} by
\begin{equation*}
 V_1 \otimeshat V_2 = (V_1 \otimes V_2)^c.
\end{equation*}

\begin{Theorem}[{\cite[Proposition~1.111]{MR2337277}}]
 With the completed tensor product and the bornology on hom-sets
 defined above, $\CBorn$ becomes a closed symmetric monoidal
 category. This means that $\otimeshat$ is unital, commutative and
 associative in the appropriate sense, and there are natural
 isomorphisms
 \begin{equation*}
 \Hom\big(V_1 \otimeshat V_2, W\big) \cong \Hom(V_1, \Hom(V_2, W))
 \end{equation*}
 for all $V_1, V_2, W \in \CBorn$. Likewise, $\Born$ is closed
 symmetric monoidal with its uncompleted tensor product $\otimes$.
\end{Theorem}

For $V \in \Born$, denote by $\mathfrak S_d(V)$ the set of bounded disks
of $V$, with the partial order relation of being absorbed: $B_1 \leq
B_2$ if $B_1 \subset c B_2$ for some constant $c > 0$. This implies that
we have an injective, continuous linear map $V_{B_1} \to V_{B_2}$ of
seminormed spaces. Denote by $\mathfrak S_c(V)$ the subset of complete
bounded disks. By definition, $\mathfrak S_c(V)$ is cofinal in
$\mathfrak S_d(V)$ if and only if $V$ is complete.

\begin{Proposition}
 \label{prop:dissection}
 Let $V$ be a complete bornological vector space. Then there is a
 canonical isomorphism
 \begin{equation*}
 V \cong \colim_{B \in \mathfrak S_c(V)} V_B
 \end{equation*}
 in $\CBorn$. If $V \in \Born$, replacing $\mathfrak S_c (V)$ with
 $\mathfrak S_d (V)$ gives a similar isomorphism in $\Born$.
\end{Proposition}

\begin{Remark}
 Complete bornological vector spaces are directly related to the
 category $\mathrm{Ind}(\mathrm{Ban})$ of inductive systems of Banach
 spaces. The assignment of the directed system $\{ V_B \}_{B \in
 \mathfrak S_c (V)}$ to~$V \in \CBorn$ defines the \emph{dissection}
 functor $\mathrm{diss}\colon \Born \to \mathrm{Ind}(\mathrm{Ban})$.
 There are also a versions of~dissection for non-complete and
 non-separated spaces.
\end{Remark}

In sheaf theory, filtered colimits and their commutation properties
with certain kinds of limits play an important role. Thus, we analyze
now such colimits. Recall that the underlying set of a filtered
colimit of vector spaces (and hence also of bornological vector
spaces) is the filtered colimit of the underlying diagram of sets.

\begin{Lemma}
 \label{lem:bornology-of-filtered-colim}
 Let $V = \colim_{i \in I} V_i$ be a filtered colimit in $\Born$, and
 denote by $j_i\colon V_i \to V$ the standard map. Then a subset of
 $V$ is a bounded disk if and only if it is of the form $j_i(B)$ for
 some $i \in I$ and bounded disk $B \subset V_i$.
\end{Lemma}

\begin{proof}
 Expressing $V$ in terms of coproducts and a coequalizer, we find a
 map
 \begin{equation*}
 p\colon\ \bigoplus\nolimits_{i \in I} V_i \to V
 \end{equation*}
 such that bounded subsets of $V$ are precisely those of the form
 $p(B)$ with $B$ bounded. A bounded disk $B$ in the direct sum is
 given by the convex hull of a finite collection $B_{i_1}, \dots,
 B_{i_n}$ of bounded disks $B_{i_j} \subset V_{i_j}$. Pick an upper
 bound $i$ for the collection $\{i_1, \dots, i_n\} \subset I$, and let
 $B_i \subset V_i$ be the smallest bounded disk containing the images
 in $V_i$ of each $B_{i_j} \subset V_{i_j}$. Then clearly $p(B) =
 p(B_i)$, which finishes the proof.
\end{proof}

\begin{Proposition} %\label{prop:filtered-colims-exact}
 Small filtered colimits commute with finite limits in $\Born$.
\end{Proposition}

\begin{proof}
 Let $I$ be a finite indexing category, and $J$ filtered. Fix a
 diagram $V\colon I \times J \to \Born$. We want to show that the
 natural bounded linear map
 \begin{equation*}
 \phi\colon\
 \colim_{j \in J} \lim_{i \in I} V(i, j)
 \to \lim_{i \in I} \colim_{j \in J} V(i, j)
 \end{equation*}
 is an isomorphism. Since filtered colimits of vector spaces commute
 with finite limits, we know that $\phi$ is a linear isomorphism. It
 remains to show $\phi$ that has a bounded inverse, that is, every bounded
 disk $B$ in the codomain is contained in the image of a bounded
 disk.

 By Lemma~\ref{lem:bornology-of-filtered-colim}, each of the bounded
 disks $B_i = p_i(B) \subset \colim_{j\in J} V(i, j)$ is the image,
 through the standard map $\iota_{i,k_i}\colon V(i, k_i) \to
 \colim_{j \in J} V(i, j)$, of a bounded disk $B_{i,k_i} \subset V(i,
 k_i)$, for some $k_i \in J$. Since $I$ is finite, we can pick an
 upper bound $k$ for the collection $\{ k_i \}_{i \in I}$, and we can
 write $B_i = j_k(B_{i,k})$ for suitable bounded disks $B_{i,k}
 \subset V(i, k)$.

 Now, $\lim_{i \in I} V(i, k)$ is a subspace of the direct product
 $\prod_{i \in I} V(i,k)$, which, again by finiteness, is also a
 direct sum. The bounded disks $B_{i, k} \subset V(i, k)$ span a
 bounded disk in that direct sum, and by restriction, a bounded disk
 $B_k \subset \lim_{i \in I} V(i, k)$. This in turns determines a
 bounded disk $B' = \iota_k(B_k)$ in the domain of $\phi$, and
 $\phi(B') = B$.
\end{proof}

In general, a filtered colimit of separated, or even complete,
bornological spaces need not be separated, as the following example
shows.

\begin{Example}
 Let $X$ be a manifold and $x \in X$. Consider the space of germs
 of smooth functions around $x$,
 \begin{equation*}
 C^\infty_{X, x} = \colim_{U \ni x} C^\infty(U).
 \end{equation*}
 Endowing $C^\infty(U)$ with its usual von Neumann bornology and
 taking the colimit in $\Born$ gives~$C^\infty_{X, x}$ a bornology.
 Then the germ of a function $f \in C^\infty(X)$ which vanishes to
 infinite order at~$0$ is in the closure of~$\{ 0 \}$ in
 $C^\infty_{X, x}$. In fact, we can find a sequence of functions
 $f_n \in C^\infty(X)$ with trivial germ at $0$ converging to $f$ in
 the Fréchet sense, so (equivalently) in the bornological sense.
 Therefore, the same holds for their images in $C^\infty_{X, x}$.
\end{Example}

On the other hand, filtered colimits $V = \colim V_i$ of separated or
complete bornological spaces remain separated or complete for diagrams
with only injective morphisms or, more generally, if~every bounded
$B_i \subset V_i$ with $j_i(B_i) = 0 \subset V$ is already in the
kernel of some morphism $V_i \to V_j$ of the diagram. Such inductive
systems are called \emph{stable} by Houzel~\cite{MR0393552}.

\begin{Proposition}
 \label{prop:stable-filtered-colims-exact}
 Let $\{ V_i \}_{i \in I}$ be a stable filtered diagram of separated
 or complete bornological spaces. Then $V = \colim_{i \in I} V_i$ is
 separated or complete, respectively.
\end{Proposition}

\begin{proof}
 Let $B \subset V$ be a bounded disk. Then there exists $i \in I$
 and a bounded disk $B_i \subset V_i$ such that $j_i(B_i) = B$. Now,
 $B_i \cap \Ker j_i$ is a bounded disk with trivial image in $V$, so
 it has trivial image already in $V_k$ for some $k \in I$. Let $B_k$
 be the image of $B_i$ in $V_k$, so that $B = j_k(B_k)$. Then
 $(V_k)_{B_k} \to V_B$ is injective, so $V_B$ is a normed space.
\end{proof}

\begin{Proposition} \label{prop:fine-otimes}
 Let $V,\, W \in \CBorn$ and assume $V$ has the fine bornology. Then
 $V \otimes W$ is already complete. If both $V$ and $W$ have the
 fine bornology, the same is true of their tensor product.
\end{Proposition}

\begin{proof}
 The first assertion is clear if $V$ is finite-dimensional.
 Otherwise, we have, by Proposition~\ref{prop:dissection}, $V \cong
 \colim V_\alpha$, where $V_\alpha$ ranges through the collection of
 finite-dimensional subspaces of $V$, endowed with their fine
 bornologies and partially ordered by inclusion. Since, like any
 left adjoint, tensor product commutes with colimits, we have
 \begin{equation*}
 V \otimes W \cong \colim (V_\alpha \otimes W).
 \end{equation*}
 This is a stable filtered colimit of complete bornological spaces,
 hence complete. If also $W$ is fine, then
 \begin{equation*}
 V \otimes W \cong \colim (V_\alpha \otimes W_\beta),
 \end{equation*}
 where $V_\alpha \subset V$ and $W_\beta \subset W$ range through all
 finite-dimensional subspaces. This is a cofinal subcollection of
 the finite-dimensional subspaces of $V \otimes W$, and therefore
 induces the fine bornology.
\end{proof}

\subsection{Bornological algebras and modules}

It is straightforward to define the notion of bounded bilinear maps
$V_1 \times V_2 \to W$ between complete bornological vector spaces,
and a bornology on the space $\Hom^{(2)}(V_1, V_2; W)$ of such.
Moreover, there is a natural isomorphism
\begin{equation*}
 \Hom^{(2)}(V_1, V_2; W) \cong \Hom\big(V_1 \otimeshat V_2, W\big).
\end{equation*}

A bornological algebra is a bornological vector space $A$ with a
bounded, associative bilinear product $A \times A \to A$. A
bornological module is a bornological vector space $M$ with a bounded,
associative, bilinear action map $A \times M \to M$. Any of those is
called complete if the underlying bornology is complete.

In this paper we only consider unital, commutative bornological
algebras and unital modules, so we drop the extra adjectives.
Moreover, we focus on the complete case. Thus we get a~category
$\CAlg$ of complete bornological algebras and, for each $A \in \CAlg$,
a category $\CMod(A)$ of complete bornological modules. Module
categories are enriched in themselves, that is, for~$M, N \in
\CMod(A)$, the space of $A$-linear maps $\Hom_A(M, N)$, with its
subspace bornology, is naturally a complete $A$-module.

By Proposition~\ref{prop:fine-otimes}, the functor $\Fin\colon
\mathrm{Vect} \to \CBorn$ sends algebras and modules over them to
complete bornological algebras and modules. Similarly, a Fréchet
algebra $A$ induces a~complete bornological algebra structure on
$\vN(A)$, and a Fréchet $A$-module $M$ gives a complete
$\vN(A)$-module structure on $\vN(M)$~\cite[Theorem~1.29]{MR2337277}.

Given modules $M$ and $N$ over the bornological algebra $A$, we set
\begin{equation*}
 M \otimeshat_A N = \Coker \big(M \otimeshat A \otimeshat N \to M \otimeshat N\big),
\end{equation*}
where the map is specified by $(m, a, n) \mapsto am \otimes n - m
\otimes an$.

The usual adjunction between extension and restriction of scalars
carries over to the bor\-no\-lo\-gi\-cal setting. Below, we use $j$ to regard
$B$ and $N$ as $A$-modules.

\begin{Proposition}% \label{prop:CMod-hom-tensor-adjunction}
 Let $j\colon A \to B$ be a map of complete bornological algebras and let
 $M$, $N$ be complete modules over $A$ and $B$ respectively. Then there
 are natural isomorphisms
 \begin{equation*}
 \Hom_B\big(B \otimeshat_A M, N\big) \cong \Hom_A(M, N).
 \end{equation*}
 Moreover, the extension of scalars functor $B \otimeshat_A
 \mathemdash\colon \CMod(A) \to \CMod(B)$ is monoidal.
\end{Proposition}

\begin{proof} Since limits in $\CBorn$ commute with the forgetful functor to
 $\mathrm{Vect}$, we can identify the left-hand side with those
 bounded bilinear maps $f\colon B \times M \to N$ such that
 \begin{equation*}
 f(b b', m) = b f(b', m),
 \qquad
 f(b j(a), m) = f(b, am)
 \end{equation*}
 for all $a \in A$, $b, b' \in B$ and $m \in M$. With this
 observation, the usual algebraic manipulations give a linear
 bijection between the hom spaces in question, and it is easy to
 check that everything is compatible with the bornologies.

 Letting $j = \id\colon A \to A$, we see that $A \otimeshat_A M \cong
 M$. From this, and the fact that tensor products over $A$ and $B$
 are associative, one deduces that $B \otimeshat_A \mathemdash$ is
 monoidal.
\end{proof}

\subsection{Smooth functions}\label{sec:smooth-functions}

Given a smooth manifold $S$, we denote by $C^\infty_\vN(S)$ the space
of smooth functions with the von Neumann bornology associated to its
usual Fréchet topology. This can be described as the bornology of
uniform boundedness of all derivatives of each given order on
compacts. If $V$ is a~complete bornological space and $f\colon S \to
V$ a~function, we say that $f$ is \emph{smooth} if there exists a~complete bounded disk $B \subset V$ such that $f$ takes values in
$V_B$, and $f\colon S \to V_B$ is smooth as a~function with values in
a Banach space.

\begin{Proposition} \label{prop:smooth-functions-in-V}
 For any smooth manifold $S$ and complete bornological vector space
 $V$, there is a natural isomorphism of vector spaces
 \begin{equation*}
 C^\infty(S; V) \cong C^\infty_\vN(S) \otimeshat V.
 \end{equation*}
 This defines a bornology on the left-hand side, which we denote
 $C^\infty_\vN(S; V)$.
\end{Proposition}

\begin{proof}
 Using the dissection isomorphism $V = \colim_B V_B$, where $B$
 ranges over all complete bounded disks in $V$
 (Proposition~\ref{prop:dissection}), we get
 \begin{equation*}
 C^\infty(S) \otimeshat V \cong \colim_B C^\infty(S) \otimeshat V_B.
 \end{equation*}
 On the other hand, we have
 \begin{math}
 C^\infty(S; V) = \colim_B C^\infty(S, V_B)
 \end{math}
 by definition. This reduces the proposition to the case $V =
 \vN(V_\nu)$ for a Banach space $V_\nu$.

 The rest of the argument is functional analysis. We have
 $C^\infty(S; \vN(V_\nu)) = \vN(C^\infty(S; V_\nu))$ by~\cite[Corollary~3.9]{MR2097966}. It is well known~\cite[Theorem~44.1]{MR0225131} that $C^\infty(S; V_\nu) \cong C^\infty(S)
 \otimeshat_\pi V_\nu$, where we used the completed projective tensor product.
 Finally, $\vN(V) \otimeshat \vN(W) \cong \vN(V \otimeshat_\pi W)$ if
 $V$ is nuclear~\cite[Theorem~1.91]{MR2337277}, which finishes the
 proof, since $C^\infty(S)$ is nuclear for any manifold~$S$ (see,
 e.g., the corollary of Theorem~51.5 in~\cite{MR0225131}).
\end{proof}

\begin{Proposition}
 \label{prop:smooth-functions-on-product}
 For any smooth manifolds $S$ and $T$, we have
 \begin{equation*}
 C^\infty_\vN(S \times T) \cong
 C^\infty_\vN(S) \otimeshat C^\infty_\vN(T).
 \end{equation*}
\end{Proposition}

\begin{proof}
 This follows from the corresponding statement in the Fréchet
 setting, using the completed projective tensor product~\cite[Theorem~51.6]{MR0225131}, and the fact that $\vN(V) \otimeshat
 \vN(W) \cong \vN(V \otimeshat_\pi W)$ if $V$ is nuclear.
\end{proof}

\subsection{Bornological sheaves}\label{sec:bornological-sheaves}

In this section, we outline the theory of sheaves of bornological
vector spaces. We follow Houzel~\cite[Section~2]{MR0393552}.

Let $\mathcal C$ be a category with limits, and denote by
$\PSh_{\mathcal C}(X)$ and $\Sh_{\mathcal C}(X)$ the categories of~pre\-sheaves and sheaves with values in $\mathcal C$ on a topological
space (or, more generally, a Gro\-then\-dieck site) $X$. We are
interested in the cases $\mathcal C = \CBorn$ and, as a preliminary
step, $\mathcal C = \Born$. We~will say that a (pre)sheaf with values
in $\Born$ is \emph{complete} if it takes values in $\CBorn$.

The first item in our wish list is to have a left adjoint to the
inclusion $\Sh_{\mathcal C}(X) \to \PSh_{\mathcal C}(X)$, which we
call the \emph{sheafification} functor. Let $F \in \PSh_{\mathcal
 C}(X)$, $U \subset X$ open, and $\mathfrak U = \{ U_i \subset U
\}_{i \in I}$ be an open cover. As usual, we write
\begin{equation*}
 F(\mathfrak U) = \lim \bigg(\prod_{i \in I} F(U_i)
 \rightrightarrows \prod_{i, j \in I} F(U_i \cap U_j)\bigg)
\end{equation*}
and say that $F$ is a sheaf if the canonical map $F(U) \to F(\mathfrak
U)$ an isomorphism for all $\mathfrak U$. We will call $F$ a
\emph{semisheaf} if the maps $F(U) \to F(\mathfrak U)$ are always
monomorphisms. Since the inclusion $\CBorn \subset \Born$ preserves
limits, the condition of being a (semi)sheaf in one of these
categories is equivalent to being a (semi)sheaf in the other. Set
\begin{equation*}
 F^+(U) = \colim_{\mathfrak U} F(\mathfrak U),
\end{equation*}
where $\mathfrak U$ runs through all open covers of $X$. For the
traditional choices $\mathcal C = \mathrm{Set}$, $\mathrm{Vect}$,
etc., $F^+$ is a semisheaf, and is a sheaf if $F$ is already a
semisheaf. This defines the sheafification functor $F \mapsto
F^{++}$. When $\mathcal C = \Born$, $F^{++}$ is obviously a sheaf
when regarded as taking values in $\Vect$, and it can be checked
directly that it is in fact a bornological sheaf.

If $F \in \PSh_\Born(X)$ is a complete semisheaf, then the inductive
system defining $F^+$ is stable, so it follows, from
Proposition~\ref{prop:stable-filtered-colims-exact}, that $F^+$ is a
complete sheaf. However, if $F$ fails to be a~semi\-sheaf, then $F^+$
may not be complete even if $F$ is. This issue is resolved in Houzel~\cite[p.~32]{MR0393552} as follows. Let $F \in \PSh_\Born$ be~arbitrary and let $E$ be the smallest sub-presheaf of $U \mapsto
\widehat{F(U)}$ such that $F'\colon U \mapsto \widehat{F(U)}/E(U)$ is
a complete semisheaf. Then $(F')^+$ is a complete sheaf, and this
construction gives a left adjoint to the inclusion $\Sh_\CBorn(X) \to
\PSh_\Born(X)$. By~restriction, this left adjoint gives us both a
sheafification functor $\PSh_\CBorn(X) \to \Sh_\CBorn$ and a
completion functor $\Sh_\Born \to \Sh_\CBorn$.

The second item in our wish list is to get tensor products and
internal homs. If $E, F \in \PSh_\Born(X)$, then $\Hom(E, F)$ gets a
bornology as a subspace of the product $\prod_{U \subset X} \Hom(E(U),\allowbreak
F(U))$, $U \subset X$ open, and we get also a $\Born$-valued presheaf
$\underline\Hom(E, F)$. It is a sheaf if so is $F$, and it is
complete if so is $F$. We define $E \otimes F$ as the sheafification
of the presheaf $U \mapsto E(U) \otimes F(U)$, and $E \otimeshat F$ as
the associated complete sheaf.

The following statement summarizes this discussion.

\begin{Theorem}
 Let $\mathcal C = \Born$ or $\CBorn$ and let $X$ be a topological space.
 Then the inclusion $\Sh_{\mathcal C}(X) \to \PSh_{\mathcal C}(X)$
 admits a left adjoint. Moreover, with the hom-sheaf and tensor
 product defined above, $\Sh_{\mathcal C}(X)$ is a closed symmetric
 monoidal category.
\end{Theorem}

If $f\colon X \to Y$ is a continuous map and $E \in \PSh_\Born(X)$,
then the direct image $f_*E$ is the presheaf $U \mapsto E(f^{-1}(U))$,
$U \subset Y$ open. It is a sheaf if $E$ is. For $F \in
\Sh_\Born(Y)$, let $f^{-1}(F)$ is the sheafification of the presheaf
$U \mapsto \colim_{V \supset f(U)} F(V)$. This defines a left adjoint
\begin{equation*}
 f^{-1}\colon\ \Sh_\Born(Y) \to \Sh_\Born(X).
\end{equation*}
to $f_*$. If $g\colon Y \to Z$ is another map, then $g_* \circ f_*
\cong (g \circ f)_*$ and $f^{-1} \circ g^{-1} \cong (g \circ f)^{-1}$.
The first fact is obvious and the second follows because $f^{-1}$ is
adjoint to $f_*$.

Note also that if $f$ is an open map and $F$ takes values in $\CBorn$,
then so does $f^{-1}(F)$.

Finally, we turn to sheaves of algebras and modules. Let $A$ be a
sheaf of complete bornological algebras over $X$, and denote by
$\CMod(A)$ the category of sheaves of complete $A$-modules. If~$M, M'
\in \CMod(A)$, we denote by $\underline\Hom_A(M, M')$ the subsheaf of
$\underline\Hom(M, M')$ consisting of~$A$-linear morphisms. It is
complete if $M'$ is. The tensor product $M \otimeshat_A M'$ is the
complete sheaf associated to the presheaf $U \mapsto M(U)
\otimes_{A(U)} M'(U)$. Assume now we are given a map $f\colon X \to
Y$. Then $f_*A$ is naturally a sheaf of bornological algebras, and
$f_*M$ an $f_*A$-module. Also, if $B$ is a sheaf of bornological algebras
on $Y$ and $N \in \CMod(B)$ is a module, then $f^{-1}(B)$ is naturally
a sheaf of bornological algebras, and $f^{-1}(N)$ an
$f^{-1}(B)$-module.

\subsection[The stack VvN]{The stack $\boldsymbol{\mathcal V_\vN}$}
\label{sec:construction-V-vN}

The assignment $U \mapsto C^\infty_\vN(U)$ defines a sheaf of complete
bornological algebras on the site $\Man$ of smooth manifolds. Its
restriction to the small site of a manifold $S$ will be denoted
$C^\infty_{\vN, S}$.

We denote by $\mathcal V_\vN(S) = \CMod(C^\infty_{\vN, S})$ the
category of sheaves of complete $C^\infty_{\vN, S}$-modules. As
observed in the previous subsection, this is a closed symmetric
monoidal category.

Any smooth map $f\colon T \to S$ induces a homomorphism
$C^\infty_{\vN, S} \to f_*C^\infty_{\vN, T}$ of sheaves of complete
bornological algebras. We denote by $f^\sharp\colon
f^{-1}(C^\infty_{\vN, S}) \to C^\infty_{\vN, T}$ its adjoint. This
makes $C^\infty_{\vN, T}$ into a sheaf of $f^{-1}(C^\infty_{\vN,
 S})$-modules. Given $M \in \mathcal V_\vN(S)$, we set
\begin{equation*}
 \hat f^*M = C^\infty_{\vN, T}
 \otimeshat_{f^{-1}(C^\infty_{\vN, S})}
 f^{-1}(M).
\end{equation*}
This defines a symmetric monoidal functor
\begin{equation*}
 \hat f^*\colon\ \mathcal V_\vN(S) \to \mathcal V_\vN(T),
\end{equation*}
and the assignments $S \to \mathcal V_\vN(S)$ and $f \mapsto \hat
f^*$ determine a prestack of symmetric monoidal categories
\begin{equation*}
 \mathcal V_\vN\colon\ \Man^{\mathrm{op}} \to \mathrm{Cat}^\otimes.
\end{equation*}

Finally, we need to show that $\mathcal V_\vN$ is in fact a stack.
This is the case because the pretopology of all open covers of a
manifold $S$ and the pretopology of open covers subordinated to a
given open cover generate the same Grothendieck site. More
concretely, if $\mathfrak U = \{ U_i \subset S \}$ is an open cover
and $V_i \in \mathcal V_\vN(U_i)$ is a collection of objects with
coherent isomorphisms between their restrictions to overlaps $U_i \cap
U_j$, we can construct a presheaf of $C^\infty_{\vN, S}$-modules $V'$
which agrees with $V_i$ on $U_i$ and assigns $0$ to any open not
contained in some $U_i$. Sheafification then gives the desired
``glued together'' object $V \in \mathcal V_\vN(S)$.

\subsection*{Acknowledgments}

We thank Dmitri Pavlov, Peter Teichner, Konrad Waldorf
and, especially, Stephan Stolz for helpful discussions. We are further
indebted to the Max Planck Institute for Mathematics in~Bonn and the
University of Greifswald, where part of this research was conducted.
The first-named author was partially supported by the ARC Discovery
Project grant FL170100020 under Chief Investigator and Australian
Laureate Fellow Mathai Varghese.

\pdfbookmark[1]{References}{ref}
\LastPageEnding

\end{document}